\documentclass[11pt, reqno]{amsart}

\usepackage{comment}
\usepackage[rgb,dvipsnames]{xcolor}
\usepackage{stmaryrd}
\usepackage{amsfonts}
\usepackage{amssymb}
\usepackage{amsmath,amsthm}
\usepackage{latexsym}
\usepackage{amscd}
\usepackage{esint}
\usepackage{mathrsfs}
\usepackage{dsfont}
\usepackage{setspace}
\usepackage{enumitem}
\usepackage[margin=1.25in]{geometry}
\usepackage{bm}

\usepackage{hyperref}
\usepackage{comment}
\usepackage{enumitem}

\newtheorem{prop}{Proposition}[section]
\newtheorem{thm}[prop]{Theorem}
\newtheorem{lemm}[prop]{Lemma}
\newtheorem{coro}[prop]{Corollary}

\newtheorem*{claim*}{Claim}

\theoremstyle{definition}
\newtheorem{defi}[prop]{Definition}
\newtheorem{rmk}[prop]{Remark}

\newcommand{\CC}{\mathbb{C}}

\newcommand{\NN}{\mathbb{N}}

\newcommand{\RR}{\mathbb{R}}

\newcommand{\ZZ}{\mathbb{Z}}

\newcommand{\cA}{\mathcal A}
\newcommand{\cB}{\mathcal B}
\newcommand{\cC}{\mathcal C}
\newcommand{\cD}{\mathcal D}
\newcommand{\cE}{\mathcal E}
\newcommand{\cF}{\mathcal F}

\newcommand{\cK}{\mathcal K}

\newcommand{\cM}{\mathcal M}
\newcommand{\cN}{\mathcal N}

\newcommand{\cP}{\mathcal P}

\newcommand{\cS}{\mathcal S}

\newcommand{\cU}{\mathcal U}

\def\bB{\mathbf{B}}
\def\bT{\mathbf{T}}

\DeclareMathOperator{\im}{Im}

\DeclareMathOperator{\supp}{supp}

\DeclareMathOperator{\Div}{div}
\DeclareMathOperator{\loc}{loc}
\DeclareMathOperator{\osc}{osc}
\DeclareMathOperator{\re}{Re}

\DeclareMathOperator{\dist}{dist}

\DeclareMathOperator{\Ind}{Ind}

\DeclareMathOperator{\Ran}{Ran}

\DeclareMathOperator{\Vol}{\text{Vol}}

\newcommand{\ep}{\varepsilon}
\newcommand{\wep}{\widetilde{\varepsilon}}

\newcommand{\pa}[2]{\frac{\partial #1}{\partial #2}}

\newcommand{\paop}[1]{\pa{}{#1}}

\newcommand{\rom}[1]{\expandafter\romannumeral #1}
\newcommand{\Rom}[1]{\uppercase\expandafter{\romannumeral #1}}

\setlist[enumerate]{leftmargin = 2em}
\allowdisplaybreaks

\numberwithin{equation}{section}

\title[Existence of free boundary CMC disks]{Existence of free boundary disks with constant mean curvature in $\RR^3$}
\author{Da Rong Cheng}
\address{Department of Mathematics, University of Miami, Coral Gables, FL 33146}
\email{darong.cheng@miami.edu}

\begin{document}

\begin{abstract} 
Given a surface $\Sigma$ in $\RR^3$ diffeomorphic to $S^2$, Struwe~\cite{Struwe88} proved that for almost every $H$ below the mean curvature of the smallest sphere enclosing $\Sigma$, there exists a branched immersed disk which has constant mean curvature $H$ and boundary meeting $\Sigma$ orthogonally. We reproduce this result using a different approach and improve it under additional convexity assumptions on $\Sigma$. Specifically, when $\Sigma$ itself is convex and has mean curvature bounded below by $H_0$, we obtain existence for all $H \in (0, H_0)$. Instead of the heat flow in~\cite{Struwe88}, we use a Sacks-Uhlenbeck type perturbation. As in previous joint work with Zhou~\cite{ChengZhou-cmc}, a key ingredient for extending existence across the measure zero set of $H$'s is a Morse index upper bound. 
\end{abstract}

\maketitle

\section{Introduction}
Let $\Sigma$, henceforth referred to as the constraint, be a closed surface in $\RR^3$ diffeomorphic to $S^2$. In this paper we adapt the approach in previous joint work with Zhou on constant mean curvature spheres~\cite{ChengZhou-cmc} to study the existence of constant mean curvature disks with boundary meeting $\Sigma$ orthogonally. For our main results we assume that $\Sigma$ is enclosed by a closed surface $\Sigma'$ which is strictly convex and has mean curvature bounded from below by some $H_0 > 0$ with respect to the inward pointing normal. For example, define 
\[
R_0 = \inf\{ R>0\ |\ \Sigma \subseteq \overline{B_R(y)} \text{ for some }y \in \RR^3 \}.
\]
Then there exists $y_0 \in \RR^3$ such that $\Sigma \subseteq \overline{B_{R_0}(y_0)}$, and we can take $\Sigma' = \partial B_{R_0}(y_0)$ and $H_0 = \frac{2}{R_0}$. Our first result is a refinement of the seminal work of Struwe~\cite{Struwe88}. Below we denote the open unit disk in $\RR^2$ by $\bB$. The bounded open regions enclosed by $\Sigma$ and $\Sigma'$ are denoted $\Omega$ and $\Omega'$, respectively.
\begin{thm}\label{thm:main-1}
For almost every $H \in (0, H_0)$, there exists a non-constant, free boundary, branched immersion $u: (\overline{\bB}, \partial \bB) \to (\overline{\Omega'}, \Sigma)$ with constant mean curvature $H$ and Morse index at most $1$.
\end{thm}
Theorem~\ref{thm:main-1} with the barrier $\Sigma'$ being a sphere and without the Morse index statement was first obtained by Struwe~\cite{Struwe88} using the heat flow approach he initiated in~\cite{Struwe1985} (see also the survey~\cite{Struwe1988survey}). Here, to bound the Morse index using the techniques in~\cite{ChengZhou-cmc}, we go back to the approach of Struwe's earlier work~\cite{Struwe1984} and perturb the relevant functional, although we do not use the $\alpha$-energy of Sacks-Uhlenbeck~\cite{Sacks-Uhlenbeck81}. While getting a Morse index bound appears to be only a slight improvement, it turns out to be the key to recovering the measure zero set of mean curvatures left out in Theorem~\ref{thm:main-1}, under additional convexity assumptions that makes $\Sigma$ itself a barrier. This leads to our second existence result. 
\begin{thm}\label{thm:main-2}
Suppose $\Sigma$ is strictly convex and has mean curvature bounded from below by some $H_0 > 0$. Then for all $H \in (0, H_0)$, there exists a non-constant, free boundary, branched immersion $u: (\overline{\bB}, \partial \bB) \to (\overline{\Omega}, \Sigma)$ with constant mean curvature $H$.
\end{thm}
We next elaborate on the statements. The branched immersed, free boundary constant mean curvature disks we produce are smooth solutions to the following system of equations.
\begin{align}
&\Delta u= H \cdot u_{x^1} \times u_{x^2} \text{ on }\overline{\bB}, \label{eq:cmc-H}\\
&|u_{x^1}|  = |u_{x^2}|,\ \langle u_{x^1}, u_{x^2}\rangle = 0  \text{ on }\overline{\bB}, \label{eq:weakly-conformal}\\
&u(x) \in \Sigma \text{ and }u_r(x) \perp T_{u(x)}\Sigma  \text{ for all }x \in \partial \bB.\label{eq:fb}
\end{align}
It is standard~\cite{Hildebrandt-Nitsche1979} that any non-constant smooth map satisfying~\eqref{eq:cmc-H},~\eqref{eq:weakly-conformal} and~\eqref{eq:fb} is an immersion away from finitely many branch points. Moreover, since the domain is a disk,~\eqref{eq:weakly-conformal} follows from~\eqref{eq:cmc-H} and~\eqref{eq:fb} by a standard Hopf differential argument (see for example~\cite[Lemma 2.2, especially page 27]{Struwe88}). Note that under the weak conformality condition~\eqref{eq:weakly-conformal}, equation~\eqref{eq:cmc-H} expresses, up to the conformal factor $\frac{|\nabla u|^2}{2}$, that away from the branch points, the mean curvature of the image of $u$ equals $H$ times a unit normal.

Although our approach is different from that in~\cite{Struwe88}, we still produce solutions to~\eqref{eq:cmc-H} and~\eqref{eq:fb} as critical points of the functional
\[
E_H: (u, \gamma) \mapsto \left(\int_{\bB} \frac{|\nabla u|^2}{2} \right) + H \cdot V(\gamma)=: D(u) + H \cdot V(\gamma),
\]
defined for pairs $(u, \gamma)$ consisting of a map $u: (\overline{\bB}, \partial \bB)\to (\RR^3, \Sigma)$ and a path $\gamma$ in the space of such maps leading from a constant to $u$. Here $V(\gamma)$ can be thought of as the volume of the region swept out by $\gamma$. See Section~\ref{subsec:volume} for the definition, which is based on the one in~\cite{Struwe88} but differs technically. In any case, it still turns out that changing the choice of $\gamma$ alters the value of $E_H$ by $H$ times an integer multiple of the volume enclosed by $\Sigma$. In particular, the derivatives of $E_H$ depend only on $u$, and maps at which the first variation vanishes are, at least formally, solutions to~\eqref{eq:cmc-H} and~\eqref{eq:fb}. The Morse index in the statement of Theorem~\ref{thm:main-1} refers to the index of $u$ as a critical point of $E_H$.

Finally, as in~\cite{Struwe88}, the upper bound $H_0$ in Theorems~\ref{thm:main-1} and~\ref{thm:main-2} is there to ensure that when ``energy concentration'' occurs, we obtain disks rather than spheres. Through appropriate local estimates and a classification result due to Brezis-Coron~\cite{BrezisCoron1985}, this is reduced to the fact that there are no round spheres of mean curvature $H$ inside a region whose boundary mean curvature is everywhere larger than $H$.


\vskip 1em
Among the previous variational constructions of free boundary, constant mean curvature surfaces, most directly related to this paper are the works of Struwe~\cite{Struwe1984,Struwe88}, already mentioned above. The existence of free boundary minimal ($H = 0$) disks proved in~\cite{Struwe1984} was extended to more general ambient spaces by Fraser~\cite{Fraser2000}. Both authors used the $\alpha$-energy introduced by Sacks-Uhlenbeck~\cite{Sacks-Uhlenbeck81}, with the result in~\cite{Struwe1984} later reproduced using a heat equation as a special case of the main theorem in~\cite{Struwe88}. In turn,~\cite{Fraser2000} was refined by Lin-Sun-Zhou~\cite{Lin-Sun-Zhou2020}, and independently by Laurain-Petrides~\cite{Laurain-Petrides2019}, using the harmonic replacement procedure developed by Colding and Minicozzi~\cite{Colding-Minicozzi08b}. The very recent work of Sun~\cite{YuchinSun22} extends~\cite{Lin-Sun-Zhou2020,Laurain-Petrides2019} from disks to surfaces of other topological types. On the other hand, in a series of papers (for instance~\cite{Riviere17,Pigati-Riviere-CPAM,Pigati-Riviere-Duke}), an alternative mapping approach together with a regularity theory have been developed by Rivi\`ere and Pigati in the case of closed surfaces, and then adapted by Pigati~\cite{Pigati2022} to obtain free boundary minimal surfaces. Returning to the case $H \neq 0$, we note that B\"urger-Kuwert~\cite{Wolfram-Kuwert2008} obtained free boundary, constant mean curvature disks by a constrained minimization of the Dirichlet energy, motivated by partitioning problems. However, in this setting the value of the mean curvature occurs as a Lagrange multiplier and seems hard to prescribe in advance.

Another major approach to the existence problem uses geometric measure theory. For example, Gr\"uter and Jost~\cite{GruterJost1986,GruterJost1986b} produced embedded free boundary minimal disks in convex bodies in $\RR^3$. More general existence results along this line for free boundary minimal surfaces with controlled topology were obtained by Li~\cite{Li2015}. On the other hand, specializing to the unit ball in $\RR^3$, free boundary minimal surfaces of arbitrary genus and possessing various symmetries were constructed by Ketover~\cite{Ketover2016} and Carlotto-Franz-Schulz~\cite{Carlotto-Franz-Schulz2020}. Going beyond surfaces in $3$-manifolds, we refer to De Lellis-Ramic~\cite{DeLellis-Ramic16} and Li-Zhou~\cite{Li-Zhou2021} for the existence of embedded, free boundary minimal hypersurfaces in higher dimensions. See also~\cite{Guang-Wang-Zhou2018,Guang-Li-Wang-Zhou2021,Wang2020,Sun-Wang-Zhou2020}, for example, for other recent progress on the existence theory of free boundary minimal hypersurfaces, which has undergone rapid development parallel to advances in the closed case. As for the case $H \neq 0$, very recently, Li-Zhou-Zhu~\cite{Li-Zhou-Zhu2021} developed a new boundary regularity theory and proved that in a compact $3$-manifold with boundary, for generic metrics, there exist almost properly embedded surfaces with any prescribed constant mean curvature and contacting the boundary at any prescribed angle, extending the existence result for closed constant mean curvature hypersurfaces established earlier by Zhou-Zhu~\cite{Zhou-Zhu19,Zhou-Zhu20}. Independently, De Masi-De Philippis~\cite{DMDP2021} obtained minimal surfaces with any prescribed boundary contact angle in a bounded convex domain in $\RR^3$.
\subsection*{Outline of proofs}
Our task is to find non-constant critical points of $E_H$ using min-max methods, since direct minimization may result in trivial solutions, as would be the case if $\Sigma$ is convex. It is well-known that the Dirichlet energy barely fails the Palais-Smale condition, and thus a certain perturbation is required, which in~\cite{Struwe88} was effected through a heat flow. Here, we instead opt for a Sacks-Uhlenbeck type perturbation of the functional. Specifically, we consider, for $\ep \in (0, 1]$, $p > 2$ and $H < H_0$, 
\[
E_{\ep, p, H}: (u, \gamma) \mapsto \left(\int_{\bB} \frac{|\nabla u|^2}{2} + \frac{\ep^{p-2}(1 + |\nabla u|^2)^{\frac{p}{2}}}{p}\right) + H \cdot V(\gamma) = : D_{\ep, p}(u) + H \cdot V(\gamma).
\]
Perturbing the Dirichlet energy by adding a small multiple of the $p$-energy was already considered by Uhlenbeck in~\cite{Uhlenbeck1981}. It turns out that for $p$ sufficiently close to $2$ depending only on $\Sigma$, critical points of $E_{\ep, p, H}$ are smooth and satisfy various \textit{a priori} estimates. (See Section~\ref{sec:estimates}.) Such a $p$ will be fixed throughout, while $\ep$ will eventually be sent to zero.

We now outline the proof of Theorem~\ref{thm:main-1}. Given $H \in (0, H_0)$, we follow the steps in~\cite[Section 3]{ChengZhou-cmc} to obtain, for almost every $r \in (0, 1]$, critical points $u_{\ep}$ of $E_{\ep, p, rH}$ with $D_{\ep, p}(u_\ep)$ bounded uniformly in $\ep$, and with Morse index at most $1$. Even though in proving the deformation lemma~\cite[Lemma 3.9]{ChengZhou-cmc}, the Hilbert manifold structure of the underlying space of mappings is used to obtain local coordinates (\cite[Proposition 2.23]{ChengZhou-cmc}) that help us deform sweepouts away from high-index critical points, it turns out to be enough for the second variation at critical points to extend to a Hilbert space (see Section~\ref{sec:2nd-var}, especially Proposition~\ref{prop:coordinates}), and a deformation lemma is still valid in the present setting (Proposition~\ref{prop:bypassing}). Another difference from~\cite{ChengZhou-cmc} is that since $\RR^3$ is non-compact, we need a uniform $L^{\infty}$-bound on $u_\ep$. As in~\cite{Struwe88}, this is where the barrier $\Sigma'$ enters, although we use the maximum principle differently. In the perturbed functional, in place of the constant $H$, we actually put a smooth function $f$ on $\RR^3$ which equals $H$ on a neighborhood of $\overline{\Omega'}$ and has compact support in $\Omega'_{t_0} := \{\dist(\cdot, \Omega') < t_0\}$ for some $t_0$ so small that $\partial\Omega'_t$ still has mean curvature at least $\frac{H + H_0}{2}$ for $t \in (0, t_0]$. A simple application of the maximum principle (Proposition~\ref{prop:max-principle}(a)) using the convexity of $\Sigma'$ gives the preliminary bound $u_\ep(\overline{\bB}) \subseteq \overline{\Omega'_{t_0}}$, which will later be improved. 

The next step is passing to the limit as $\ep \to 0$. For that we adapt the interior and boundary estimates in~\cite{Sacks-Uhlenbeck81} and~\cite{Fraser2000}. As already mentioned, it is the proofs of these estimates that inform our choice of $p$. Then, similar to~\cite{Sacks-Uhlenbeck81}, whether or not there is concentration of energy, we obtain a non-constant solution $v$ to~\eqref{eq:cmc-H} and~\eqref{eq:fb} with $rf(v)$ in place of $H$. Another application of the maximum principle (Proposition~\ref{prop:max-principle}(b)) gives $v(\overline{\bB})\subseteq \overline{\Omega'}$, and hence $v$ solves~\eqref{eq:cmc-H} and~\eqref{eq:fb} with $rH$ in place of $H$ by our choice of $f$. An important point in this process is to ensure that when energy concentration occurs, the parameter $\ep$ still converges to zero after rescaling, and that the rescaled limit is defined on a half-plane as opposed to $\RR^2$. These boil down to comparing the rate of rescaling with $\ep$ (Proposition~\ref{prop:rescale}) and with the distance to $\partial \bB$ from the center of rescaling (Proposition~\ref{prop:no-interior-bubble}). Our choice of perturbation is essential to the former, while the latter uses, similar to~\cite{Struwe88}, the condition $H < H_{\Sigma'}$ and the classification by Brezis-Coron~\cite{BrezisCoron1985} of finite-energy solutions to~\eqref{eq:cmc-H} on $\RR^2$. Finally, arguing as in~\cite[Section 4]{ChengZhou-cmc} shows that the index upper bound on $u_{\ep}$ passes to $v$. 

For Theorem~\ref{thm:main-2}, we apply Theorem~\ref{thm:main-1} with $\Sigma' = \Sigma$ and take a sequence $(H_n)$ in the resulting full-measure set that converges to a given $H \in (0, H_0)$. The corresponding solutions $u_n$ then all map into $\overline{\Omega}$. We next generalize the computation in~\cite[Section 5.1]{ChengZhou-cmc}, which in turn was inspired by the work of Ejiri-Micallef~\cite{EM} on closed minimal surfaces, to pass the bound on the index of $u_n$ as a critical point of $E_{H_n}$ to what is essentially its index as a critical point of the functional $\text{Area}(u) + H_n \cdot V(\gamma)$. (A similar computation was carried out for free boundary minimal surfaces in~\cite{Lima2017}.) We then use the Hersch trick (see for instance~\cite{LiYau}, or~\cite{ABCS2019} for a more relevant setting) and the convexity of $\Sigma$ to obtain a uniform bound on $D(u_n)$. We finish the proof of Theorem~\ref{thm:main-2} by a bubbling analysis similar to the one in the proof of Theorem~\ref{thm:main-1}.
 
\subsection*{Organization} 
In Section~\ref{sec:perturbed} we define the enclosed volume and the perturbed functional, $E_{\ep, p, f}$. In Section~\ref{sec:estimates} we prove that critical points of the perturbed functional are smooth and derive a number of \textit{a priori} estimates. In Section~\ref{sec:2nd-var} we study the behavior of $E_{\ep, p, f}$ near a critical point and establish a substitute for~\cite[Proposition 2.23]{ChengZhou-cmc}. In Section~\ref{sec:existence-perturbed} we carry out the min-max construction of critical points of $E_{\ep, p, f}$ with index at most $1$, using the estimates in Section~\ref{sec:2nd-var} to ensure the index bound. In Section~\ref{sec:passage} we analyze these critical points as $\ep \to 0$, based on the \textit{a priori} estimates obtained in Section~\ref{sec:estimates}. Theorem~\ref{thm:main-1} is proved at the end of Section~\ref{sec:passage}. Finally, in Section~\ref{sec:improved} we prove Theorem~\ref{thm:main-2}. 
\subsection*{Acknowledgement} I would like to thank Xin Zhou for numerous enlightening discussions.
\section{The perturbed functional}\label{sec:perturbed}
\subsection{Function spaces and special coordinates}\label{subsec:spaces-and-coordinates}
Let $\Sigma \subseteq \RR^3$ be a surface which is the image of the standard $S^2$ under a smooth diffeomorphism from $\RR^3$ to itself, and let $\Omega$ denote the bounded open region enclosed by $\Sigma$. Below we set up some additional notation.
\vskip 1mm
\begin{flushleft}
\begin{tabular}{ll}
$(x^1, x^2), (y^1, y^2, y^3)$ & Standard coordinates on $\RR^2$ and $\RR^3$, respectively.\\
$B_r(y)$ & An open ball in $\RR^3$.\\
$\bB_r(x)$ & An open disk in $\RR^2$.\\
$\RR^2_+$ & The open upper half-plane.\\
$\bB$ & The open unit disk in $\RR^2$.\\
$\bB^+$ & $\bB \cap \RR^2_+$.\\
$\bT$ & $\bB \cap \partial \RR^2_+$.\\
$\bB^+_{r}(x)$ & $\bB_r(x) \cap \RR^2_+$.\\
$\bT_r(x)$ & $\bB_r(x) \cap \partial \RR^2_+$.\\
\end{tabular}
\end{flushleft}
\vskip 1mm
For $p > 2$, we define
\[
\cM_p = \{u \in W^{1, p}(\bB; \RR^3)\ |\ u(x) \in \Sigma \text{ for all }x \in \partial \bB\}.
\]
We will eventually fix $p$ depending only on $\Sigma$. The set $\cM_p$ is an embedded smooth submanifold of the Banach space $W^{1, p}(\bB; \RR^3)$ and is closed as a subset. The tangent space of $\cM_p$ at $u$ is given by 
\[
T_u\cM_p = \{\psi \in W^{1, p}(\bB; \RR^3)\ |\ \psi(x) \in T_{u(x)}\Sigma \text{ for all }x \in \partial \bB\},
\]
which is a closed subspace of $W^{1, p}(\bB; \RR^3)$ with a closed complement. Moreover, restricting the $W^{1, p}$-norm to each tangent space yields a Finsler structure on $\cM_p$. 

As explained for instance in~\cite[Section 7.2]{ACS2017}, by interpolating the product metric on a tubular neighborhood of $\Sigma$ with the Euclidean metric $g_{\RR^3}$, we obtain a metric $h$ on $\RR^3$ which differs from $g_{\RR^3}$ only near $\Sigma$, such that $\Sigma$ is totally geodesic with respect to $h$. For $u \in \cM_p$, the map
\[
\Theta_u : \psi \mapsto \exp^{h}_u(\psi),
\]
for $\psi \in T_u\cM_p$ with $\|\psi\|_{1, p}$ sufficiently small, is a chart for $\cM_p$. Here and below, unless otherwise stated, the standard metric and its Levi-Civita connection will be used in computing norms, distances and derivatives on $\RR^3$, while $h$ will only be an auxiliary device.

Certain coordinates adapted to $\Sigma$ are needed to derive local boundary estimates in later sections. First, note that we can find $\rho_\Sigma > 0$ and a finite open covering $\cF$ of $\Sigma$ such that for all $y \in \Sigma$, the ball $B_{\rho_\Sigma}(y)$ is contained in some $U \in \cF$, and that for each $U\in \cF$, there exists a diffeomorphism $\Psi$ which maps $U$ onto some open ball $V \subseteq \RR^3$ and has the following additional properties:
\begin{enumerate}
\item[(f1)] $\Psi(U \cap \Sigma) = V \cap \{y^3 = 0\}$,  $\Psi(U \cap \Omega) = V \cap \{y^3 > 0\}$, and $(d\Psi)_{p}(\eta(p)) = \paop{y^3} \text{ for }p \in U \cap \Sigma$, where $\eta(p)$ denotes the unit normal to $\Sigma$ pointing into $\Omega$.
\vskip 1mm
\item[(f2)] Letting $\Upsilon = \Psi^{-1}$, $g = \Upsilon^{\ast}g_{\RR^3}$ and $Q = \Upsilon^{\ast}\Vol_{\RR^3}$, we have that $g$ is uniformly equivalent to the Euclidean metric on $V$, and that the components of $g$ and $Q$ along with their derivatives of all orders are bounded on $V$. Here by components we mean 
\[
g_{ij}(y) = \langle \pa{\Upsilon}{y^i},  \pa{\Upsilon}{y^j} \rangle,\text{ and }Q_{ikl}(y) =\langle \pa{\Upsilon}{y^i}, \pa{\Upsilon}{y^k} \times \pa{\Upsilon}{y^l} \rangle.
\]
Also, by (f1), $g_{i3}(y) = 0$ for all $y \in V \cap \{y^3 = 0\}$, $i = 1, 2$.
\end{enumerate}

Secondly, we will sometimes need to locally flatten $\partial\bB$. Specifically, there exist a universal constant $r_{\bB}$ and, for each $x \in \partial \bB$, a conformal map $F$ from $\overline{\RR^2_+}$ onto $\overline{\bB}\setminus \{-x\}$ such that $|F_z - 1| \leq \frac{1}{4}$ on $\bB_{r_{\bB}}^+$, that
\begin{equation}\label{eq:F-interpolate}
\overline{\bB_{\frac{3}{4}r}(x)} \cap \overline{\bB} \subseteq F(\overline{\bB_{r}^+}) \subseteq \overline{\bB_{r}(x)} \cap \overline{\bB}, \text{ for all }r \in (0, r_{\bB}],
\end{equation}
and that all the derivatives of $F$ on $\bB_{r_{\bB}}^+$ are bounded independently of $x \in \partial \bB$. Below, whenever such an $F$ is introduced, we write $\lambda$ for $|F_z|$, so that $F^{\ast}g_{\RR^2} = \lambda^2 g_{\RR^2}$. As an example of a typical situation where the above coordinates will be introduced, take $u \in \cM_p$ and suppose that for some $L > 0$ and $r \in (0, r_{\bB}]$ there holds
\[
r^{p-2}\int_{\bB_{r}(x) \cap \bB} |\nabla u|^p \leq L^p \text{ for all }x \in \partial \bB.
\]
Then by Sobolev embedding there exists $\theta \in (0, \frac{1}{4})$ depending only on $p, \Sigma, L$ such that
\[
u(\overline{\bB_{\theta r}(x)} \cap \overline{\bB}) \subseteq B_{\rho_\Sigma}(u(x)) \text{ for all }x \in \partial \bB.
\]
Thus we can find $U, V, \Psi, \Upsilon, ...$ as above, and define 
\[
\widehat{u}(x) = (\Psi\circ u \circ F)(r x), \text{ for }x \in \overline{\bB_{\theta}^+}.
\]
The condition $u(\partial\bB) \subseteq \Sigma$ implies $\widehat{u}^3 = 0$ on $\bT_{\theta}$. Also, for all $\widehat{\psi} \in W^{1, p}(\bB^+_{\theta}; \RR^3)$ such that $\widehat{\psi} = 0$ on $\partial \bB_{\theta} \cap \RR^2_+$ and $\widehat{\psi}^3 = 0$ on $\bT_{\theta}$, if we define $\psi$ by the following relation
\[
\psi(F(rx)) = (d\Upsilon)_{\widehat{u}(x)}(\widehat{\psi}(x))  \text{ for }x \in \overline{\bB_{\theta}^+}, 
\] 
and extend $\psi$ to be zero on $\overline{\bB} \setminus F(\overline{\bB_{\theta r}^+})$, then $\psi \in T_u \cM_p$.

\subsection{The enclosed volume}\label{subsec:volume}
Suppose $f: \RR^3 \to \RR$ is a bounded smooth function such that $\int_{\Omega}f > 0$. In the case $f \equiv 1$, what we now define can be viewed as the volume swept out by a continuous path $\gamma: [0, 1] \to \cM_p$. The idea originated with~\cite{Struwe88}, even though our definition is technically different, primarily because we work in a perturbed setting and have the embedding $W^{1, p} \to C^0$ at our disposal. 

By this embedding, any continuous path $\gamma: [0, 1] \to \cM_p$ induces a continuous map $G_{\gamma}: ([0, 1] \times \overline{\bB}, [0, 1] \times \partial \bB) \to (\RR^3, \Sigma)$. Moreover, the weak derivatives $(G_{\gamma})_{x^1}, (G_\gamma)_{x^2}$ exist and lie in $L^{p}([0, 1] \times \bB)$. For the moment, we assume in addition that $\gamma$ is piecewise $C^1$, in which case the weak derivative $(G_\gamma)_t$ exists as well, and lies in $L^{\infty}([0, 1] \times \bB)$. We then define
\begin{equation*}
V_f(\gamma) = \int_{[0, 1] \times \bB}G_{\gamma}^{\ast}(f \Vol_{\RR^3}) := \int_{[0, 1] \times \bB} f(G_{\gamma}) \big\langle (G_\gamma)_t, (G_\gamma)_{x^1} \times (G_\gamma)_{x^2} \big\rangle.
\end{equation*}
Letting ``$+$'' denote concatenation and ``$-$'' the reversal of orientation, it is easy to see that $V_f(\gamma + \widetilde{\gamma}) = V_f(\gamma) + V_f(\widetilde{\gamma})$, and that $V_f(-\gamma) = -V_f(\gamma)$.

The next lemma allows us to extend $V_f$, henceforth referred to as the enclosed volume, to continuous paths in general. 

\begin{lemm}\label{lemm:homotopy}
There exists a constant $\delta_0$ depending only on $\Sigma$ such that if $\gamma, \widetilde{\gamma}: [0, 1] \to \cM_p$ are two continuous and piecewise $C^1$ paths with $\gamma(0) = \widetilde{\gamma}(0)$, $\gamma(1) = \widetilde{\gamma}(1)$, and if
\begin{equation}\label{eq:closeness}
\sup_{t \in [0, 1]}\|\gamma(t) - \widetilde{\gamma}(t)\|_{\infty} < \delta_0, 
\end{equation}
then $V_f(\gamma) = V_f(\widetilde{\gamma})$.
\end{lemm}
\begin{proof}
We begin with some preliminary estimates. Since $h$ agrees with $g_{\RR^3}$ outside of a tubular neighborhood of $\Sigma$, there exist $\delta_0$ and a smooth function 
\[
c: [0, 1] \times \{(y, y') \in \RR^3 \times \RR^3\ |\ |y - y'| < \delta_0\} \to \RR^3,
\]
with bounded derivatives of all orders, such that $c(\cdot, y, y')$ is the unique minimizing $h$-geodesic from $y$ to $y'$, and maps into $\Sigma$ whenever $y, y' \in \Sigma$. For $u, v \in \cM_p$ with $\|u - v\|_{\infty} < \delta_0$, we define a $C^1$-path $l_{u, v}:[0, 1] \to \cM_p$ by $l_{u, v}(s)(x) = c(s, u(x), v(x))$. Then we have the following pointwise estimates. 
\begin{equation}\label{eq:slice-energy}
|(l_{u, v}'(s))(x)| \leq C|u(x) - v(x)| \text{ and } |\nabla (l_{u, v}(s))(x)| \leq C(|\nabla v(x)| + |\nabla u(x)|),
\end{equation}
with both constants depending only on $\Sigma$. Hence, 
\begin{align}\label{eq:segment-vol}
|V_f(l_{u, v})| \leq  C\|f\|_{\infty} \int_{[0, 1] \times \bB}\big| u - v\big|  | \nabla l_{u, v}(s) |^2 &\leq C\|f\|_{\infty}\|u - v\|_{\infty}(\|\nabla u\|_2^2 + \|\nabla v\|_2^2).
\end{align}

To continue, by assumption~\eqref{eq:closeness}, we can define a homotopy $\Gamma: [0, 1] \times [0, 1] \to \cM_p$ by $\Gamma(s, t) = c(s, \gamma(t), \widetilde{\gamma}(t))$. Note that $\Gamma(\cdot, t)$ is a $C^1$-path in $\cM_p$ for all $t \in [0, 1]$, while $\Gamma(s, \cdot)$ is piecewise $C^1$ for all $s\in [0, 1]$. Letting $L = \sup_{t \in [0, 1]} \left( \|\nabla \gamma(t)\|_{2} + \| \nabla \widetilde{\gamma}(t) \|_2\right)$, we see from~\eqref{eq:slice-energy} that 
\[
\sup_{(s, t) \in [0, 1] \times [0, 1]} \|\nabla \Gamma(s, t)\|_2 \leq CL.
\]
Hence by~\eqref{eq:segment-vol} we get
\[
|V_f(\Gamma(\cdot, t)|_{[s, s']})| \leq CL^2\|f\|_{\infty} \| \Gamma(s, t) - \Gamma(s', t) \|_{\infty}, \text{ for all } 0 \leq s \leq s' \leq 1 \text{ and }0 \leq t \leq 1.
\]
In particular, we can choose a partition $0 = s_0 < \cdots < s_k = 1$ such that
\begin{equation}\label{eq:short-cut}
|V_f(\Gamma(\cdot, t)|_{[s_{i-1}, s_i]})| < \frac{1}{4}\int_{\Omega}f, \text{ for }i = 1, \cdots, k \text{ and }t \in [0, 1].
\end{equation}
Next, take a partition $0 = t_0 < \cdots < t_l = 1$ in the $t$-direction such that each $\Gamma(s_i, \cdot)$ is continuously differentiable on the sub-intervals. By the continuity of each $t \mapsto V_f(\Gamma(s_i, \cdot)|_{[0, t]})$, we can refine the partition if necessary so that
\[
|V_f(\Gamma(s_i, \cdot)|_{[t_{j-1}, t_j]})| < \frac{1}{4}\int_\Omega f \text{ for }i = 1, \cdots, k \text{ and } j = 1, \cdots, l.
\]
Then, letting 
\[
\beta_{ij} = \Gamma(s_{i-1}, \cdot)|_{[t_{j-1}, t_j]} + \Gamma(\cdot, t_j)|_{[s_{i-1}, s_i]} + (-\Gamma(s_i, \cdot)|_{[t_{j-1}, t_j]}) + (-\Gamma(\cdot, t_{j-1})|_{[s_{i-1}, s_i]}),
\]
we have, for all $1 \leq i \leq k$ and $1 \leq j \leq l$, that
\begin{equation}\label{eq:short-circuit}
\Big |V_f (\beta_{ij}) \Big| < \int_{\Omega}f.
\end{equation}
On the other hand, as each $\beta_{ij}$ is a closed, piecewise $C^1$ path, the induced map $G_{\beta_{ij}}$ descends to a map $F_{ij}:(S^1 \times \overline{\bB}, S^1 \times \partial \bB) \to (\RR^3, \Sigma)$ such that
\[
V_f(\beta_{ij}) = \deg((F_{ij})|_{S^1 \times \partial \bB})\int_{\Omega}f \in \big(\int_{\Omega}f \big)\cdot\ZZ.
\]
Thus,~\eqref{eq:short-circuit} implies that $V_f (\beta_{ij}) = 0$. Since $\Gamma(s_{i-1}, \cdot)$ and $\Gamma(s_i, \cdot)$ have common endpoints, we have for $i = 1, \cdots, k$ that 
\begin{align*}
V_f(\Gamma(s_{i-1}, \cdot)) - V_{f}(\Gamma(s_i, \cdot))=\ & \sum_{j = 1}^l \big[V_f(\Gamma(s_{i-1}, \cdot)|_{[t_{j-1}, t_j]}) - V_f(\Gamma(s_i, \cdot)|_{[t_{j-1}, t_j]})\big]\\
=\ & \sum_{j = 1}^l V_f(\beta_{ij}) = 0.
\end{align*}
We complete the proof upon recalling that $\gamma = \Gamma(0, \cdot)$ and $\widetilde{\gamma} = \Gamma(1, \cdot)$.
\end{proof}
Now, given a continuous path $\gamma:[0, 1] \to \cM_p$, we define $V_f(\gamma) = V_f(\gamma_1)$, where $\gamma_1$ is any choice of piecewise $C^1$ path having the same endpoints as $\gamma$ and satisfying $\|\gamma(t) - \gamma_1(t)\|_{\infty} < \delta_0/2$ for all $t \in [0, 1]$. By Lemma~\ref{lemm:homotopy}, the choice of $\gamma_1$ is irrelevant. One way to produce such a $\gamma_1$, in fact with any prescribed threshold in place of $\delta_0/2$, is to take a fine enough partition $0 = t_0 < \cdots < t_k = 1$ and let $\gamma_1= \sum_{i = 1}^k l_{\gamma(t_{i - 1}), \gamma(t_i)}$.

Moreover, from the above definition, it is not hard to see that Lemma~\ref{lemm:homotopy} continues to hold for paths that are merely continuous. An immediate corollary of this extension of Lemma~\ref{lemm:homotopy} is the following.
\begin{coro}\label{coro:homotopy}
Let $\gamma, \widetilde{\gamma}$ be two homotopic continuous paths in $\cM_p$. Then $V_f(\gamma) = V_f(\widetilde{\gamma})$.
\end{coro}
Below we single out another important property of the enclosed volume.
\begin{lemm}\label{lemm:vol-prop}
Let $u \in \cM_p$ and let $\gamma, \widetilde{\gamma}$ be two continuous paths in $\cM_p$ such that $\gamma(0), \widetilde{\gamma}(0)$ are constant maps and $\gamma(1) = \widetilde{\gamma}(1) = u$. Then
\[
V_f(\gamma) - V_f(\widetilde{\gamma}) \in (\int_{\Omega}f)\cdot \ZZ.
\]
\end{lemm}
\begin{proof}
It suffices to prove this when both $\gamma$ and $\widetilde{\gamma}$ are piecewise $C^1$ and constant near $t = 0$. Then, similar to the previous proof, we can use the map $F: (B^3, \partial B^3) \to (\RR^3, \Sigma)$ which descends from $G_{\gamma + (-\widetilde{\gamma})}$ and is constant near a pair of antipodal points to compute $V_f(\gamma + (-\widetilde{\gamma})) = \deg(F|_{\partial B^3}) \cdot \int_{\Omega}f$. This gives the result.
\end{proof}
\subsection{The perturbed functional}\label{subsec:perturbed}
We continue to let $f$ be a bounded, smooth function on $\RR^3$ with $\int_{\Omega} f > 0$. Given $u \in \cM_p$, we let 
\[
\cE(u) = \{\gamma \in C^0([0, 1]; \cM_p)\ |\ \gamma(0) = \text{constant,\ }\gamma(1) = u\}.
\]
To see that $\cE(u)$ is non-empty, we first note that there exists a continuous map $F:([0, 1] \times \overline{\bB}, [0, 1] \times \partial \bB) \to (\RR^3, \Sigma)$ such that $F(0, \cdot)$ is a constant and $F(1, \cdot) = u$. As $C^{\infty}(\overline{\bB}, \partial \bB; \RR^3, \Sigma)$ is dense in $C^0(\overline{\bB}, \partial \bB; \RR^3, \Sigma)$, we obtain a path in $\cE(u)$ by taking a fine enough partition $0 = t_0 < \cdots < t_k = 1$, smoothly approximating $F(t_i, \cdot)$ for $i = 1, \cdots k-1$ to get $u_1, \cdots, u_{k-1}$, and letting
\[
\gamma = l_{F(0, \cdot), u_1} + l_{u_1, u_2} + \cdots + l_{u_{k-1}, u}.
\]
We now define the perturbed functional. For $p > 2$, $\ep > 0$, $u \in \cM_p$, and $\gamma \in \cE(u)$, we let
\[
E_{\ep, p, f}(u, \gamma) = \left(\int_{\bB}\frac{\ep^{p-2}}{p}(1 + |\nabla u|^2)^{\frac{p}{2}} + \frac{|\nabla u|^2}{2} \right) + V_f(\gamma) =: D_{\ep, p}(u) + V_f(\gamma).
\]
As in~\cite{ChengZhou-cmc}, locally the variable $\gamma$ in $E_{\ep, p, f}$ can be eliminated. Indeed, given a simply-connected neighborhood $\cA \subseteq \cM_p$, we fix $u_0 \in \cA$ and $\gamma_0 \in \cE(u_0)$, and define, for $u \in \cA$, 
\[
E^{\cA}_{\ep, p, f}(u) = E_{\ep, p, f}(u, \gamma_0 + l_u),
\]
where $l_u$ is any continuous path in $\cA$ from $u_0$ to $u$. By simply-connectedness and Corollary~\ref{coro:homotopy}, we see that the choice of $l_u$ is irrelevant. We refer to $E^{\cA}_{\ep, p, f}$ as the local reduction of $E_{\ep, p, f}$ on $\cA$ induced by $(u_0, \gamma_0)$. The dependence of $E^{\cA}_{\ep, p, f}$ on $(u_0, \gamma_0)$ is suppressed from the notation as the choice should always be clear from the context. 

\begin{lemm}\label{lemm:C2}
Given $\cA$ as above and $u_0 \in \cA$, $\gamma_0 \in \cE(u_0)$, the local reduction $E^{\cA}_{\ep, p, f}$ is a $C^2$-functional on $\cA$. 
\end{lemm}
\begin{proof}
It suffices to prove that given $u \in \cA$, the composition $E^{\cA}_{\ep, p, f}\circ \Theta_u$ is $C^2$ on a small ball $\cB$ (centered at the origin) in $T_u\cM_p$ with $\Theta_u(\cB) \subseteq \cA$, where, recall,
\[
\Theta_u(\psi) (x) = \exp_{u(x)}^h(\psi(x)),\text{ for }\psi \in \cB \text{ and }x \in \overline{\bB}.
\]
To that end, note that, for all $\psi \in \cB$,
\begin{align}
E_{\ep, p, f}^\cA(\Theta_u(\psi)) - E_{\ep, p, f}^\cA(u) =\ & D_{\ep, p}(\Theta_u(\psi)) - D_{\ep, p}(u)\nonumber\\
& + \int_{[0, 1] \times \bB} f(\Theta_u(s\psi)) \pa{\Theta_u(s\psi)}{s} \cdot \pa{\Theta_u(s\psi)}{x^1} \times \pa{\Theta_u(s\psi)}{x^2}.\label{eq:local-difference}
\end{align}
It is standard to verify that the right-hand side is $C^2$ in $\psi$. The details are omitted.
\end{proof}
In particular, each local reduction $E^{\cA}_{\ep, p, f}$ is continuous, and hence by Lemma~\ref{lemm:vol-prop} and the connectedness of $\cA$, switching to a different choice of $(u_0, \gamma_0)$ alters the value of $E^{\cA}_{\ep, p, f}$ by a constant integer multiple of $\int_{\Omega}f$. More generally, any two local reductions differ by a constant integer multiple of $\int_{\Omega}f$ on any connected component of their common domain.
\subsection{First variation and a Palais-Smale condition}
Given $u \in \cM$, choose a simply-connected neighborhood $\cA$ of $u$ and consider a local reduction $E^{\cA}_{\ep, p, f}$. We define
\[
\delta E_{\ep, p, f}(u) = \delta E^{\cA}_{\ep, p, f}(u).
\]
This is well-defined thanks to the last paragraph of the previous section. We say that $u$ is a critical point of $E_{\ep, p, f}$ if $\delta E_{\ep, p, f}(u) = 0$, which means that $\delta E^{\cA}_{\ep, p, f}(u) = 0$ for any local reduction whose domain contains $u$. In this case there is a well-defined Hessian $\delta^2 E^{\cA}_{\ep, p, f}(u)$, which we take to be $\delta^2 E_{\ep, p, f}(u)$. Again this does not depend on the choice of local reduction. 

The second variation is the subject of Section~\ref{sec:2nd-var}. For now we focus on the first variation, which is given as follows. For $u \in \cM_p$ and $\psi \in T_u\cM_p$, we have
\begin{equation}\label{eq:first-var-formula}
\langle \delta E_{\ep, p, f}(u), \psi \rangle = \int_{\bB}[1 + \ep^{p-2}(1 + |\nabla u|^2)^{\frac{p}{2} - 1}] \langle \nabla u, \nabla \psi \rangle + \int_{\bB} f(u)\langle\psi, u_{x^1} \times u_{x^2}\rangle.
\end{equation}
It follows that if $u \in \cM_p$ is a \textit{smooth} critical point of $E_{\ep, p, f}$, then $u$ satisfies
\begin{equation}\label{eq:perturbed-EL}
\left\{
\begin{array}{ll}
\Div\Big( [1 + \ep^{p-2}(1 + |\nabla u|^2)^{\frac{p}{2} - 1}] \nabla u\Big) = f(u)u_{x^1} \times u_{x^2} & \text{ in }\bB,\\
u_r(x) \perp T_{u(x)}\Sigma & \text{ for all }x \in \partial \bB.
\end{array}
\right.
\end{equation}
The next proposition is a Palais-Smale condition for $E_{\ep, p, f}$. It differs from the standard version in that the uniform upper bound is placed on $D_{\ep, p}$ instead of on $E_{\ep, p, f}$ itself.
\begin{prop}\label{prop:PS}
Given $C_0 > 0$, let $(u_m)$ be a sequence in $\cM_p$ such that $D_{\ep, p}(u_m) \leq C_0$ for all $m$ and that
\[
\|\delta E_{\ep, p, f}(u_m)\| \to 0 \text{ as }m \to \infty.
\] 
Then, passing to a subsequence if necessary, $(u_m)$ converges strongly in $W^{1, p}$ to a critical point $u$ of $E_{\ep, p, f}$ with $D_{\ep, p}(u) \leq C_0$.
\end{prop}
\begin{proof}
By assumption, Sobolev embedding and the compactness of $\Sigma$, $\|u_m\|_{p} + \|\nabla u_m\|_{p}$ is uniformly bounded. Thus, taking a subsequence if necessary, $(u_m)$ converges weakly in $W^{1, p}$ and uniformly to some $u \in \cM_p$. In particular, there exists $r \in (0, r_{\bB})$ and $m_0 \in \NN$ such that for all $m \geq m_0$ and $x \in \partial \bB$, we have that $u(\overline{\bB_r(x)} \cap \overline{\bB})$ and $u_m(\overline{\bB_r(x)} \cap \overline{\bB})$ are both contained in $B_{\rho_\Sigma}(u(x))$. Now fix $x \in \partial \bB$, let $U, V, \Psi, \Upsilon, g, Q$ and $F, \lambda$ be as in Section~\ref{subsec:spaces-and-coordinates}. In addition, let $\widehat{f} = f\circ \Upsilon$, write
\[
\widehat{u}_m(x) = (\Psi\circ u_m \circ F)(rx) \text{ for }x \in \overline{\bB^+}
\]
and define, for all $\widehat{\psi} \in W^{1, p}(\bB^+; \RR^3)$ such that $\widehat{\psi} = 0$ on $\partial \bB \cap \RR^2_+$ and $\widehat{\psi}^3 = 0$ on $\bT$, 
\[\psi^{(m)} = \left\{
\begin{array}{ll}
 (d\Upsilon)_{\widehat{u}_m}(\widehat{\psi}) \circ (\frac{1}{r}F^{-1}), & \text{ if } x \in F(\bB^+_r),\\
 0, & \text{ if }x \in \bB \setminus F(\bB^+_{r}).
 \end{array}
 \right.
\] 
For brevity, we also introduce $\wep = \ep/r$ and define
\[
N(x, z, \xi) =  (\lambda(rx ))^{-2}g_{kl}(z)\xi_{\alpha}^{k}\xi_{\alpha}^{l}, \text{ for }(x, z, \xi) \in \overline{\bB^+} \times V \times \RR^{3 \times 2}.
\]
Then the first variation formula takes the following form:
\begin{align}
\langle  \delta E_{\ep, p, f}(u_m), \psi^{(m)} \rangle =\ &\int_{\bB^{+}} \big[ 1 + (\wep)^{p-2}(r^2 + N(\cdot, \widehat{u}_m, \nabla \widehat{u}_m))^{\frac{p}{2} - 1} \big] g_{ij}(\widehat{u}_m) \nabla \widehat{u}_m^j \cdot \nabla \widehat{\psi}^i \nonumber\\
&+ \int_{\bB^+}  \big[ 1 + (\wep)^{p-2}(r^2 + N(\cdot, \widehat{u}_m, \nabla \widehat{u}_m))^{\frac{p}{2} - 1} \big]\frac{1}{2}g_{kl,i}(\widehat{u}_m)(\nabla\widehat{u}_m^k\cdot \nabla\widehat{u}_m^l) \widehat{\psi}^i   \nonumber\\
& + \int_{\bB^+}\widehat{f}(\widehat{u}_m)Q_{ikl}(\widehat{u}_m) (\widehat{u}_m^k)_{x^1}  (\widehat{u}_m^l)_{x^2} \widehat{\psi}^i. \label{eq:fermi-variation}
\end{align}
Indeed, as intermediaries between $\psi^{(m)}$ and $\widehat{\psi}$ and respectively $u_m$ and $\widehat{u}_m$, we define $\widetilde{\psi}: \bB \to \RR^3$ by
\[
\widetilde{\psi}(x) = \left\{
\begin{array}{ll}
\widehat{\psi}\big( \frac{1}{r}F^{-1}(x) \big), & \text{ if }x \in F(\bB^+_{r}),\\
0, & \text{ if } x \in \bB \setminus F(\bB^+_{r}),
\end{array}
\right.
\]
and let $\widetilde{u}_m(x) = (\Psi \circ u_m)(x) \text{ for } x \in F(\bB^+_{r})$. Then we have 
\begin{equation}\label{eq:m-to-tilde}
u_m = \Upsilon \circ \widetilde{u}_m,\ \ \psi^{(m)} = (d\Upsilon)_{\widetilde{u}_m}(\widetilde{\psi}),
\end{equation}
and that
\begin{equation}\label{eq:tilde-to-hat}
\widehat{u}_m = \widetilde{u}_m \circ F(r\ \cdot\ ),\ \ \widehat{\psi} = \widetilde{\psi} \circ F(r\ \cdot\ ).
\end{equation}
Using~\eqref{eq:m-to-tilde} and recalling the definition of $g_{ij}, Q_{ikl}$ in Section~\ref{subsec:spaces-and-coordinates}, we see that the integrands on the right-hand side of~\eqref{eq:first-var-formula} can be expressed in terms of $\widetilde{u}_m$ and $\widetilde{\psi}$ as 
\[
\begin{split}
|\nabla u_m|^2 =\ & g_{ij}(\widetilde{u}_m)\nabla \widetilde{u}^i_m \cdot \nabla \widetilde{u}^j_m,\\
\langle \nabla u_m, \nabla \psi^{(m)}\rangle =\ & g_{ij}(\widetilde{u}_m) \nabla \widetilde{u}_m^j \cdot \nabla \widetilde{\psi}^i + \frac{1}{2}g_{jk, i}(\widetilde{u}_m) (\nabla \widetilde{u}_m^j \cdot \nabla \widetilde{u}_m^k )\widetilde{\psi}^i,\\
f(u_m)\langle \psi^{(m)}, (u_m)_{x^1} \times (u_m)_{x^2} \rangle =\ & \widehat{f}(\widetilde{u}_m) Q_{ijk}(\widetilde{u}_m) (\widetilde{u}_m^j)_{x^1} (\widetilde{u}_m^k)_{x^2} \widetilde{\psi}^i.
\end{split}
\]
Consequently, 
\begin{align}
\langle  \delta E_{\ep, p, f}(u_m), \psi^{(m)} \rangle =\ &\int_{F(\bB^+_{r})} \big[ 1 + \ep^{p-2}(1 + g_{ij}(\widetilde{u}_m)\nabla \widetilde{u}^i_m \cdot \nabla \widetilde{u}^j_m)^{\frac{p}{2} - 1} \big] g_{ij}(\widetilde{u}_m) \nabla \widetilde{u}_m^j \cdot \nabla \widetilde{\psi}^i  \nonumber\\
&+ \int_{F(\bB^+_{r})}  \big[ 1 + \ep^{p-2}(1 + g_{ij}(\widetilde{u}_m)\nabla \widetilde{u}^i_m \cdot \nabla \widetilde{u}^j_m)^{\frac{p}{2} - 1} \big]\frac{1}{2}g_{jk, i}(\widetilde{u}_m) (\nabla \widetilde{u}_m^j \cdot \nabla \widetilde{u}_m^k )\widetilde{\psi}^i   \nonumber\\
& + \int_{F(\bB^+_{r})}\widehat{f}(\widetilde{u}_m)Q_{ijk}(\widetilde{u}_m) (\widetilde{u}_m^j)_{x^1} (\widetilde{u}_m^k)_{x^2} \widetilde{\psi}^i. \nonumber
\end{align}
Pulling back via the conformal map $F(r\ \cdot\ )$ and using~\eqref{eq:tilde-to-hat} yields~\eqref{eq:fermi-variation}.

To continue, let $\zeta \in C^{\infty}_0(\bB)$ be a cut-off function which is identically $1$ on $\bB_{\frac{1}{2}}$, and let $\widehat{\psi}_{mn} = \zeta^2(\widehat{u}_m - \widehat{u}_n)$, which is admissible since $\widehat{u}_m^3, \widehat{u}_n^3$ vanish on $\bT$. Then by a straightforward computation which we explain at the end of this proof, we get
\begin{equation}\label{eq:PS-estimate}
\langle \delta E_{\ep, p, f}(u_m), \psi_{mn}^{(m)} \rangle - \langle \delta E_{\ep, p, f}(u_n), \psi_{mn}^{(n)} \rangle \geq C\|\nabla \widehat{u}_m - \nabla \widehat{u}_n\|_{p; \bB^+_{\frac{1}{2}}}^p - R_{mn},
\end{equation}
where $R_{mn} \to 0$ as $m, n \to \infty$ since $(\widehat{u}_m)$ is bounded in $W^{1, p}$ and converges uniformly. On the other hand, since $\lim_{m \to \infty}\|\delta E_{\ep, p, f}(u_m)\| = 0$ and $\|\psi_{mn}^{(m)}\|_{1, p}, \|\psi_{mn}^{(n)}\|_{1, p}$ are bounded independently of $m, n$, we see that the left-hand side above also tends to zero as $m, n \to \infty$. Thus $(\widehat{u}_m)$ converges strongly in $W^{1, p}(\bB^+_{\frac{1}{2}}; \RR^3)$. As $x \in \partial \bB$ is arbitrary, we deduce that $(u_m)$ converges strongly in $W^{1, p}$ on a neighborhood of $\partial \bB$. A similar argument gives strong $W^{1, p}$-convergence on any compact subset of $\bB$. Hence we get strong $W^{1, p}$-convergence on $\overline{\bB}$. The remaining conclusions follow easily.
\end{proof}
\begin{proof}[Proof of~\eqref{eq:PS-estimate}]
Throughout this proof, the constants $c, C$ appearing in the estimates are independent of $m, n$ unless otherwise stated. To further abbreviate the first integrand on the right-hand side of~\eqref{eq:fermi-variation}, we define, for $(x, z, \xi) \in \overline{B^+} \times V \times \RR^{3 \times 2}$, the functions
\begin{equation}\label{eq:A-a-i-definition}
A_{\alpha, i}(x, z, \xi) = \big[ 1 + (\wep)^{p-2}(r^2 + N(x, z, \xi))^{\frac{p}{2} - 1} \big]g_{ij}(z)\xi_{\alpha}^j,
\end{equation}
where $\alpha, \beta, \cdots \in \{1, 2\}$ and $i, j, \cdots \in \{1, 2, 3\}$. Then, taking the $\widehat{\psi}_{mn}$ defined above as a test function and using~\eqref{eq:fermi-variation} to compute $\langle  \delta E_{\ep, p, f}(u_m), \psi_{mn}^{(m)} \rangle$ and $\langle  \delta E_{\ep, p, f}(u_n), \psi_{mn}^{(n)} \rangle$, we get
\[
\begin{split}
&\langle \delta E_{\ep, p, f}(u_m), \psi_{mn}^{(m)} \rangle - \langle \delta E_{\ep, p, f}(u_n), \psi_{mn}^{(n)} \rangle \\
\geq\ & \int_{\bB^+} \big[ A_{\alpha, i}(\cdot, \widehat{u}_{m}, \nabla\widehat{u}_m) - A_{\alpha, i}(\cdot, \widehat{u}_{n}, \nabla\widehat{u}_n) \big] \big(\partial_{\alpha}\widehat{u}^{i}_m - \partial_{\alpha}\widehat{u}^{i}_n \big)\zeta^2 \\
& - 2 \int_{\bB^+}\big(|A_{\alpha, i}(\cdot, \widehat{u}_{m}, \nabla\widehat{u}_m)| + |A_{\alpha, i}(\cdot, \widehat{u}_{n}, \nabla\widehat{u}_n) | \big) \cdot |\widehat{u}^i_m - \widehat{u}^i_n| |\partial_{\alpha} \zeta|\zeta\\
& - C\int_{\bB^+} \big[ 1 + (\wep)^{p-2}(r^2 + |\nabla \widehat{u}_m|^2)^{\frac{p}{2} - 1}\big] |\nabla \widehat{u}_m|^2 |\widehat{u}_m - \widehat{u}_n|\zeta^2\\
& - C\int_{\bB^+} \big[ 1 + (\wep)^{p-2}(r^2 + |\nabla \widehat{u}_n|^2)^{\frac{p}{2} - 1}\big] |\nabla \widehat{u}_n|^2 |\widehat{u}_m - \widehat{u}_n|\zeta^2.
\end{split}
\]
Using~\eqref{eq:A-a-i-definition} to estimate the second integral and also combining the third and fourth integrals leads to
\begin{equation}\label{eq:PS-estimate-difference}
\begin{split}
&\langle \delta E_{\ep, p, f}(u_m), \psi_{mn}^{(m)} \rangle - \langle \delta E_{\ep, p, f}(u_n), \psi_{mn}^{(n)} \rangle \\
\geq\ & \int_{\bB^+} \big[ A_{\alpha, i}(\cdot, \widehat{u}_{m}, \nabla\widehat{u}_m) - A_{\alpha, i}(\cdot, \widehat{u}_{n}, \nabla\widehat{u}_n) \big] \big(\partial_{\alpha}\widehat{u}^{i}_m - \partial_{\alpha}\widehat{u}^{i}_n \big)\zeta^2\\
& - C\int_{\bB^+}(1 + |\nabla \widehat{u}_m|^{p-2} + |\nabla \widehat{u}_n|^{p-2})(|\nabla \widehat{u}_m| + |\nabla\widehat{u}_n|) |\widehat{u}_m - \widehat{u}_n| |\nabla \zeta|\zeta\\
& - C\int_{\bB^+} (1 + |\nabla\widehat{u}_{m}|^{p-2} + |\nabla \widehat{u}_{n}|^{p-2})(|\nabla\widehat{u}_m|^2 + |\nabla \widehat{u}_{n}|^2) |\widehat{u}_{m} - \widehat{u}_{n} |\zeta^2.
\end{split}
\end{equation}
To bound the integrand involving $A_{\alpha, i}(\cdot, \widehat{u}_{m}, \nabla\widehat{u}_m) - A_{\alpha, i}(\cdot, \widehat{u}_{n}, \nabla\widehat{u}_n)$, we follow the standard procedure described, for instance, in~\cite[page 129]{Tolksdorf1984}; that is, recalling that $V$ is convex, we use the fundamental theorem of calculus to write
\[
\begin{split}
&A_{\alpha, i}(\cdot, \widehat{u}_{m}, \nabla\widehat{u}_m) - A_{\alpha, i}(\cdot, \widehat{u}_{n}, \nabla\widehat{u}_n)\\
=\ & \Big( \int_{0}^{1} \pa{A_{\alpha, i}}{\xi^{j}_{\beta}}(\cdot,\ t\widehat{u}_m + (1 - t)\widehat{u}_n,\ t\nabla \widehat{u}_m + (1 - t)\nabla\widehat{u}_n) dt\Big) (\partial_{\beta}\widehat{u}_m^{j} - \partial_{\beta}\widehat{u}_n^{j}) \\
& + \Big( \int_{0}^{1} \pa{A_{\alpha, i}}{z^k}(\cdot,\ t\widehat{u}_m + (1 - t)\widehat{u}_n,\ t\nabla \widehat{u}_m + (1 - t)\nabla\widehat{u}_n) dt\Big) (\widehat{u}_m^{k} - \widehat{u}_n^{k}).
\end{split}
\]
Using~\eqref{eq:A-a-i-definition} to find the partial derivatives $\pa{A_{\alpha, i}}{z^{k}}$ and $\pa{A_{\alpha, i}}{\xi^{j}_{\beta}}$, we deduce that
\begin{equation}\label{eq:PS-main-lower-bound}
\begin{split}
& \int_{\bB^+} \big[ A_{\alpha, i}(\cdot, \widehat{u}_{m}, \nabla\widehat{u}_m) - A_{\alpha, i}(\cdot, \widehat{u}_{n}, \nabla\widehat{u}_n) \big] \big(\partial_{\alpha}\widehat{u}^{i}_m - \partial_{\alpha}\widehat{u}^{i}_n \big)\zeta^2\\
\geq\ & c\int_{\bB^+}\Big(  \int_{0}^{1}1 +  |t\nabla \widehat{u}_m + (1 - t)\nabla\widehat{u}_n|^{p-2} dt\Big) |\nabla \widehat{u}_m - \nabla \widehat{u}_n|^2\zeta^2 \\
& - C\int_{\bB^+} (1 + |\nabla\widehat{u}_{m}|^{p-2} + |\nabla \widehat{u}_{n}|^{p-2})\cdot (|\nabla\widehat{u}_m| + |\nabla \widehat{u}_{n}|)\cdot  |\widehat{u}_{m} - \widehat{u}_{n}| |\nabla\widehat{u}_{m} - \nabla\widehat{u}_{n}|\zeta^2.
\end{split}
\end{equation}
At this point we recall another standard fact (again see for example~\cite[page 129]{Tolksdorf1984}), namely that 
\[
\begin{split}
\int_{0}^{1} |t\xi + (1 - t)\eta|^{p-2}dt \geq c_{p}(|\xi|^{p-2} + |\eta|^{p-2}), \text{ for all }\xi, \eta \in \RR^{3 \times 2},
\end{split}
\]
as can be verified by restricting the integration to $[0, \frac{1}{4}]$ if $|\xi| \leq |\eta|$, and to $[\frac{3}{4}, 1]$ if otherwise. Applying this to the first term on the right-hand side of~\eqref{eq:PS-main-lower-bound} and substituting the result back into~\eqref{eq:PS-estimate-difference}, we obtain
\[
\begin{split}
&\langle \delta E_{\ep, p, f}(u_m), \psi_{mn}^{(m)} \rangle - \langle \delta E_{\ep, p, f}(u_n), \psi_{mn}^{(n)} \rangle \\
\geq\ &  c\int_{\bB^+} (1 + |\nabla\widehat{u}_{m}|^{p-2} + |\nabla\widehat{u}_{n}|^{p-2})  |\nabla \widehat{u}_m - \nabla \widehat{u}_n|^2\zeta^2\\
& - C\int_{\bB^+}(1 + |\nabla \widehat{u}_m|^{p-2} + |\nabla \widehat{u}_n|^{p-2})(|\nabla \widehat{u}_m| + |\nabla\widehat{u}_n|) |\widehat{u}_m - \widehat{u}_n| |\nabla \zeta|\zeta\\
& - C\int_{\bB^+} (1 + |\nabla\widehat{u}_{m}|^{p-2} + |\nabla \widehat{u}_{n}|^{p-2})(|\nabla\widehat{u}_m|^2 + |\nabla \widehat{u}_{n}|^2) |\widehat{u}_{m} - \widehat{u}_{n} |\zeta^2,
\end{split}
\]
which gives~\eqref{eq:PS-estimate} since $\zeta$ is equal to $1$ on $\bB^+_{\frac{1}{2}}$, and since 
\[
\begin{split}
|\nabla\widehat{u}_m  - \nabla\widehat{u}_n|^{p} =\ & |\nabla\widehat{u}_m  - \nabla\widehat{u}_n|^{p-2}|\nabla\widehat{u}_m  - \nabla\widehat{u}_n|^2 \\
\leq\ & C_p (| \nabla\widehat{u}_m  |^{p-2} + | \nabla\widehat{u}_n|^{p-2})|\nabla\widehat{u}_m  - \nabla\widehat{u}_n|^2.
\end{split}
\]
\end{proof}

\section{Estimates for critical points of the perturbed functional}\label{sec:estimates}
It is the argument in this section that imposes restrictions on how close $p$ needs to be to $2$. Section~\ref{subsec:smooth} mostly concerns qualitative smoothness of critical points, with the exception of Remark~\ref{rmk:W22}, whereas in Sections~\ref{subsec:a-priori} and~\ref{subsec:max-principle} we derive quantitative estimates.
\subsection{Smoothness of critical points}\label{subsec:smooth}
Suppose $H_0, H_1, L > 0$, $\ep \in (0, 1]$, $p > 2$, and take a bounded smooth function $f:\RR^3 \to \RR$ such that $\int_{\Omega}f > 0$ and 
\begin{equation}\label{eq:f-C1}
\|f\|_{\infty; \RR^3} \leq H_0,\  \|\nabla f\|_{\infty;\RR^3} \leq H_1.
\end{equation}
Let $u \in \cM_p$ be a critical point of $E_{\ep, p, f}$ which satisfies, for some $r \in (0, r_{\bB}]$, 
\begin{equation}\label{eq:W14-scale}
r^{p-2}\int_{\bB_r(x) \cap \bB} |\nabla u|^p \leq L^p \text{ for all }x \in \overline{\bB},
\end{equation}
and let $\theta = \theta(p, \Sigma, L) \in (0, \frac{1}{4})$ be as in Section~\ref{subsec:spaces-and-coordinates}. Of course, for a single map $u$, given any $L$ we can always find a small enough $r$ such that~\eqref{eq:W14-scale} holds. Nonetheless, we introduce the condition as it shows up in \textit{a priori} estimates and determines the scale on which these estimates hold.

In proving that critical points are smooth, we will focus only on boundary regularity. Note that by our choice of $\theta$, we have 
\[
u(\overline{\bB_{\theta r}(x)} \cap \overline{\bB}) \subseteq B_{\rho_\Sigma}(u(x)), \text{ for all }x \in \partial \bB,
\]
and thus, with $U, V, \Psi, \Upsilon, ...$ as in Section~\ref{subsec:spaces-and-coordinates}, we may define
\[
\widehat{u} = (\Psi \circ u \circ F)(r\ \cdot\ ):(\overline{\bB_\theta^+}, \bT_\theta) \to (\RR^3, \{y^3 = 0\}).
\]
Since $u$ is a critical point of $E_{\ep, p, f}$, we get, letting $\widetilde{\ep} = \frac{\ep}{r}$ and $\widetilde{\lambda} = \lambda( r\ \cdot\ )$, that
\begin{align}
&\int_{\bB_\theta^{+}} \big[ 1 + (\wep)^{p-2}(r^2 + \widetilde{\lambda}^{-2}g_{kl}(\widehat{u})\nabla \widehat{u}^k \cdot \nabla \widehat{u}^l )^{\frac{p}{2} - 1} \big] g_{ij}(\widehat{u}) \nabla \widehat{u}^j \cdot \nabla \widehat{\psi}^i \nonumber\\
&+ \int_{\bB_\theta^+}  \big[ 1 + (\wep)^{p-2}(r^2 + \widetilde{\lambda}^{-2}g_{lm}(\widehat{u})\nabla \widehat{u}^l \cdot \nabla \widehat{u}^m )^{\frac{p}{2} - 1} \big]\frac{1}{2}g_{jk,i}(\widehat{u})(\nabla\widehat{u}^j\cdot \nabla\widehat{u}^k) \widehat{\psi}^i   \nonumber\\
& + \int_{\bB_\theta^+}\widehat{f}(\widehat{u})Q_{ijk}(\widehat{u}) \widehat{u}^j_{x^1}  \widehat{u}^k_{x^2} \widehat{\psi}^i = 0, \label{eq:fermi-pde}
\end{align}
for all $\widehat{\psi} \in W^{1, p}(\bB^+_{\theta}; \RR^3)$ such that $\widehat{\psi} = 0$ on $\partial \bB_{\theta} \cap \RR^2_+$ and $\widehat{\psi}^3 = 0$ on $\bT_{\theta}$. Recall that $\widehat{f} = f\circ \Upsilon$.

We now introduce some further abbreviations, as we did in the proof of Proposition~\ref{prop:PS}, to elucidate the structure of the equation~\eqref{eq:fermi-pde}. We shall see that it is of the form considered in~\cite[Chapter 4]{LadyzhenskayaUraltseva1968}. Specifically, recalling that $V$ is the domain of $\Upsilon$ and hence of the components $g_{ij}$ and $Q_{ikl}$ defined in (f2) of Section~\ref{subsec:spaces-and-coordinates}, for $(x, z, \xi) \in \overline{\bB_{\theta}^{+}} \times V \times \RR^{3 \times 2}$ we define (below $\alpha, \beta \in \{1, 2\}$ and $i, j, k, \dots \in \{1, 2, 3\}$)
\begin{align*}
N(x, z, \xi) &= (\widetilde{\lambda}(x))^{-2} g_{kl}(z) \xi_\alpha^{k}\xi_\alpha^{l}\\
A_{\alpha, i}(x, z, \xi) &= \big[ 1 + (\wep)^{p-2}(r^2 + N(x, z, \xi))^{\frac{p}{2} - 1} \big]g_{ij}(z)\xi_{\alpha}^j\\
A_{i}(x, z, \xi) &= \big[ 1 + (\wep)^{p-2}(r^2 + N)^{\frac{p}{2}-1} \big] \frac{1}{2}g_{kl, i}(z)\xi_\alpha^k \xi_\alpha^l + \widehat{f}(z)Q_{ikl}(z) \xi^{k}_1\xi^{l}_2.
\end{align*}
We then have the following estimates. (The constants $c, C$ depend only on $p, \Sigma$ and $H_0$ except for the bound on $\pa{A_i}{z^m}$, which depends in addition on $H_1$. In any case, they do not depend on $\wep \in (0, \infty)$ and $r \in (0, r_{\bB}]$.)  
\begin{align*}
c|\xi|^2 & \leq N(x, z, \xi) \leq C|\xi|^2\\
A_{\alpha, i}(x, z, \xi)\xi_{\alpha}^i &\geq c\big[ 1 + (\wep)^{p-2}(r^2 + |\xi|^2)^{\frac{p}{2} - 1} \big]|\xi|^2,\\
\pa{A_{\alpha, i}}{\xi_{\beta}^j}\eta_{\alpha}^i \eta_{\beta}^j & \geq c\big[ 1 + (\wep)^{p-2}(r^2 + |\xi|^2)^{\frac{p}{2} - 1} \big]|\eta|^2,\\
|\xi||A_{\alpha, i}| + |A_i| &\leq C\big[ 1 + (\wep)^{p-2}(r^2 + |\xi|^2)^{\frac{p}{2} - 1} \big]|\xi|^2,\\
|\xi| \left| \pa{A_{\alpha, i}}{\xi_{\beta}^j} \right| + \left| \pa{A_{\alpha, i}}{z^m} \right| &\leq C\big[ 1 + (\wep)^{p-2}(r^2 + |\xi|^2)^{\frac{p}{2} - 1} \big]|\xi|, \\
|\xi| \left| \pa{A_{i}}{\xi_{\beta}^j} \right| + \left| \pa{A_{i}}{z^m} \right| &\leq C \big[ 1 + (\wep)^{p-2}(r^2 + |\xi|^2)^{\frac{p}{2} - 1} \big]|\xi|^2,\\
|\xi| \left| \pa{A_{\alpha, i}}{x^{\gamma}} \right| + \left| \pa{A_{i}}{x^{\gamma}} \right| &\leq C (\wep)^{p-2}(r^2 + |\xi|^2)^{\frac{p}{2} - 1} |\xi|^2.
\end{align*}
Also,~\eqref{eq:fermi-pde} becomes
\begin{equation}\label{eq:fermi-abbr}
\int_{\bB_\theta^+} A_{\alpha, i}(x, \widehat{u}, \nabla \widehat{u}) \partial_{\alpha}\psi^{i} + A_{i}(x, \widehat{u}, \nabla \widehat{u})\psi^i = 0, 
\end{equation}
for all $\psi \in W^{1, p}(\bB^+_{\theta}; \RR^3)$ such that $\psi = 0$ on $\partial \bB_{\theta} \cap \RR^2_+$ and $\psi^3 = 0$ on $\bT_{\theta}$. Formally integrating by parts, dividing both sides by $1 + (\wep)^{p-2}(r^2 + N(x, \widehat{u}, \nabla\widehat{u}))^{\frac{p}{2} -1}$ and inverting $(g_{ij}(\widehat{u}))_{1\leq i, j \leq 3}$, we obtain 
\begin{equation}\label{eq:fermi-nondiv}
\Delta \widehat{u}^i + (p-2)\frac{(\wep)^{p-2}(r^2 + N(x, \widehat{u}, \nabla \widehat{u}))^{\frac{p}{2} -1}}{1 + (\wep)^{p-2}(r^2 + N(x, \widehat{u}, \nabla \widehat{u}))^{\frac{p}{2} -1}}\frac{\widetilde{\lambda}^{-2} g_{kl}(\widehat{u})\partial_{\alpha}\widehat{u}^i \partial_{\beta}\widehat{u}^k }{r^2 + N(x, \widehat{u}, \nabla \widehat{u})}\partial_{\alpha\beta}\widehat{u}^l =b_i(x, \widehat{u}, \nabla\widehat{u}),
\end{equation}
for $i = 1, 2, 3$, where $b_i$ satisfies
\begin{equation}\label{eq:f-estimate}
|b_i(x, z, \xi)| \leq C(|\xi| + |\xi|^2)
\end{equation}
with $C$ depending only on $p, \Sigma$ and $H_0$. On the other hand, there is a constant $A_0$ which depends only on $\Sigma$ such that 
\begin{equation}\label{eq:deviation}
\left| \frac{(\wep)^{p-2}(r^2 + N(x, \widehat{u}, \nabla \widehat{u}))^{\frac{p}{2} -1}}{1 + (\wep)^{p-2}(r^2 + N(x, \widehat{u}, \nabla \widehat{u}))^{\frac{p}{2} -1}}\frac{\widetilde{\lambda}^{-2} g_{kl}(\widehat{u})\partial_{\alpha}\widehat{u}^i \partial_{\beta}\widehat{u}^k }{r^2 + N(x, \widehat{u}, \nabla \widehat{u})} \right| \leq A_0, 
\end{equation}
for all $\alpha, \beta \in \{1, 2\}$ and $i, l  \in \{1, 2, 3\}$. Next, for all $q> 1$, consider the space 
\[
X_{q, 0} = \{v \in W^{2, q}(\bB^+; \RR^3)\ |\ v = 0 \text{ on }\partial \bB \cap \RR^2_+, \partial_2 v^1 = \partial_2 v^2 = v^3 = 0 \text{ on }\bT\},
\]
where the vanishing of $\partial_2 v^1, \partial_2 v^2$ on $\bT$ is in the trace sense. For all $f \in L^q(\bB^+; \RR^3)$ there exists a unique $v \in X_{q, 0}$ such that $\Delta v = f$ and 
\begin{equation}\label{eq:CZ}
\|v\|_{2, q; \bB^+} \leq K_q\|f\|_{q; \bB^+}.
\end{equation}
Below we choose $2 < p_0 < 3$ satisfying 
\begin{equation}\label{eq:p0-choice}
9A_0(K_2 + K_4)(p_0 - 2) < \frac{1}{2}. 
\end{equation}
\begin{prop}\label{prop:regularity}
Let $p \in (2, p_0]$ and suppose $u \in \cM_p$ is a critical point of $E_{\ep, p, f}$. Then $u$ is smooth on $\overline{\bB}$.
\end{prop}
\begin{proof}
Interior regularity follows in essentially the same way as in~\cite{Sacks-Uhlenbeck81}. In particular $u$ is smooth in $\bB$. Below we focus on boundary regularity. As remarked above, given $u \in \cM_p$ we can always find $r \in (0, r_{\bB}]$ and $L > 0$ such that~\eqref{eq:W14-scale} holds. Given $x_0 \in \partial\bB$ we can then change coordinates as in the start of this section and get from $u|_{\overline{\bB_{\theta r}(x_0)} \cap \overline{\bB}}$ a map $\widehat{u}: (\overline{\bB_\theta^+}, \bT_\theta) \to (\RR^3, \{y^3 = 0\})$ satisfying~\eqref{eq:fermi-abbr} for all $\psi \in W^{1, p}(\bB^+_{\theta}; \RR^3)$ such that $\psi = 0$ on $\partial \bB_{\theta} \cap \RR^2_+$ and $\psi^3 = 0$ on $\bT_{\theta}$. 
\vskip 2mm
\noindent\textbf{Step 1: $W^{2, 2}$-estimate up to boundary}
\vskip 2mm
The most important ingredient is the following estimate due to Ladyzhenskaya and Uralt'seva~\cite{LadyzhenskayaUraltseva1968}, which we state without proof. We only note that the test function $\psi = \zeta^2 (\widehat{u} - \widehat{u}(0))$ used in~\cite{LadyzhenskayaUraltseva1968} continues to be admissible here. 
\begin{lemm}[\cite{LadyzhenskayaUraltseva1968}, Chapter 4, Lemma 1.3, or~\cite{Morrey1968}, Lemma 5.9.1]
\label{lemm:LU}
There exist $C> 0$ and $\theta_1 \in (0, \theta]$ depending on $p, H_0, \Sigma, L$ such that for all $\rho \leq \theta_1$ we have
\begin{align*}
\int_{\bB_{\rho}^+} &\big[ 1 + (\wep)^{p-2}(r^2 + |\nabla \widehat{u}|^2)^{\frac{p}{2} - 1} \big]\zeta^2 |\nabla \widehat{u}|^2 \leq C\rho^{2(1 - \frac{2}{p})} \int_{\bB_{\rho}^+} \big[ 1 + (\wep)^{p-2}(r^2 + |\nabla \widehat{u}|^2)^{\frac{p}{2} - 1} \big] |\nabla\zeta|^2,
\end{align*}
for all $\zeta \in W^{1, p}_0 (\bB_{\rho})$.
\end{lemm}
By Lemma~\ref{lemm:LU} and a standard difference quotient argument (see for instance~\cite[Chapter 4, Section 5]{LadyzhenskayaUraltseva1968} or~\cite[Theorem 1.11.1$'$]{Morrey1968}), there exist $\theta_2 < \frac{\theta_1}{2}$ and $C$, depending on $p, \Sigma, H_0, H_1$ and $L$, such that $\nabla\widehat{u}_{x^1} \in L^2(\bB_{\theta_2}^+)$, and that 
\begin{equation}\label{eq:hori-W22}
\int_{\bB_{\theta_2}^+}\big[ 1 + (\wep)^{p-2}(r^2 + |\nabla \widehat{u}|^2)^{\frac{p}{2} -1} \big]  |\nabla \widehat{u}_{x^1}|^2 \leq C\int_{\bB_{2\theta_2}^+}\big[ 1 + (\wep)^{p-2}(r^2 + |\nabla \widehat{u}|^2)^{\frac{p}{2} -1} \big] |\nabla \widehat{u}|^2
\end{equation}
In particular, scaling back on the right-hand side, we get
\begin{equation}\label{eq:hori-W22-result}
 \| \nabla \widehat{u}_{x^1} \|^2_{2; \bB_{\theta_2}^+} \leq C D_{\ep, p}(u; \bB_{2\theta_2r}(x_0) \cap \bB).
\end{equation}
To estimate $\widehat{u}_{x^2x^2}$, note that as $\widehat{u}$ is smooth on $\bB_{\theta_2}^+$, equation~\eqref{eq:fermi-nondiv} holds pointwise. Multiplying $(g_{ij}(\widehat{u}))_{1\leq i, j \leq 3}$ back, rearranging to leave only terms involving $\widehat{u}_{x^2 x^2}$ on the left-hand side, and then inverting the matrix $(a^{22}_{ij})_{1\leq i, j \leq 3}$ on both sides, where
\[
a^{\alpha\beta}_{ij} = g_{ij}(\widehat{u})\delta_{\alpha\beta} + (p-2)\frac{(\wep)^{p-2}(r^2 + N)^{\frac{p}{2} - 1}}{1 + (\wep)^{p-2}(r^2 + N)^{\frac{p}{2} - 1}}\frac{\widetilde{\lambda}^{-2}g_{ik}(\widehat{u})\partial_\alpha \widehat{u}^k g_{jl}(\widehat{u})\partial_\beta \widehat{u}^l}{r^2 + N},
\]
we see with the help of~\eqref{eq:f-estimate}, ~\eqref{eq:deviation} that
\begin{equation}\label{eq:22-component}
|\widehat{u}_{x^2 x^2}| \leq C(|\nabla \widehat{u}| + |\nabla\widehat{u}|^2 + |\nabla\widehat{u}_{x^1}|),
\end{equation}
with $C$ depending only on $p, \Sigma$ and $H_0$. This together with~\eqref{eq:hori-W22-result} gives $\widehat{u} \in W^{2, p/2}(\bB_{\theta_2}^+)$, and that
\begin{equation}\label{eq:2-p/2}
\|\nabla \widehat{u}\|_{1, p/2; \bB_{\theta_2}^+} \leq C( \|\nabla \widehat{u}\|_{p; \bB_{\theta_2}^+} + \|\nabla \widehat{u}\|_{p; \bB_{\theta_2}^+}^2 + \| \nabla \widehat{u}_{x^1}\|_{2; \bB_{\theta_2}^+} ),
\end{equation}
with $C$ depending on $p, \Sigma, L, H_0$ and $H_1$. (Recall that $\theta_2$ depends on these.)

We now define a sequence $(p_k)$ by letting $p_1 = p$ and 
\[
p_{k + 1} = \left\{
\begin{array}{ll}
\frac{2}{4 - p_k}p_k &, \text{ if }p_k < 4,\\
4 &, \text{ if }p_k \geq 4.
\end{array}
\right.
\]
We claim that $p_k \geq 4$ for all large enough $k$. From the above recurrence relation, it suffices to show that the inequality holds for some $k$. Suppose towards a contradiction that $p_k < 4$ for all $k$. Then, as $p > 2$, we can argue by induction to see that $p_{k} > 2$ for all $k$. Using again the assumption that the sequence stays less than $4$, we deduce that $(p_k)$ is strictly increasing, and hence $p_{k + 1} > \frac{2}{4 - p}\cdot p_k$ for all $k$. Consequently,
\[
4 > p_k > \Big( \frac{2}{4 - p} \Big)^{k-1}p, \text{ for all }k,
\]
which is a contradiction since $\frac{2}{4-p} > 1$ (recall that $2 < p < 3$). Thus, we may define
\[
k_0 = \max\{k \geq 1\ |\ p_k < 4\},
\]
which depends on $p$ only. Suppose by induction that $\nabla \widehat{u} \in W^{1, p_k/2}(\bB_{\theta_2}^+)$ for some $1 \leq k < k_0$. Then by Sobolev embedding $\nabla \widehat{u} \in L^{p_{k + 1}}(\bB_{\theta_2}^+)$, and hence with the help of~\eqref{eq:22-component}, we see that $\nabla \widehat{u} \in W^{1, p_{k + 1}/2}(\bB_{\theta_2}^+)$, with 
\begin{align*}
\|\nabla \widehat{u}\|_{1, p_{k + 1}/2; \bB_{\theta_2}^+} &\leq C(\|\nabla \widehat{u}\|_{p_{k + 1}; \bB_{\theta_2}^+} + \|\nabla \widehat{u}\|_{p_{k + 1}; \bB_{\theta_2}^+}^2 + \| \nabla \widehat{u}_{x^1}\|_{2; \bB_{\theta_2}^+} ) \nonumber\\
&\leq C(\|\nabla \widehat{u}\|_{1, p_{k}/2; \bB_{\theta_2}^+} + \|\nabla \widehat{u}\|_{1, p_{k}/2; \bB_{\theta_2}^+}^2 + \| \nabla \widehat{u}_{x^1}\|_{2; \bB_{\theta_2}^+} )  \label{eq:2-pk/2}.
\end{align*}
Iterating from $k = 1$ to $k = k_0-1$ and recalling~\eqref{eq:2-p/2},~\eqref{eq:hori-W22-result} and~\eqref{eq:W14-scale}, we get
\begin{equation}\label{eq:W14-almost}
\|\nabla \widehat{u}\|_{1, p_{k_0}/2; \bB_{\theta_2}^+} \leq C\big(p, \Sigma, H_0, H_1, L, D_{\ep, p}(u; \bB_{2\theta_2r}(x_0) \cap \bB)\big).
\end{equation}
Taking the $L^2$-norm over $\bB_{\theta_2}^+$ on both sides of~\eqref{eq:22-component}, we get
\[
\|\widehat{u}_{x^2 x^2}\|_{2; \bB_{\theta_2}^+} \leq C(\|\nabla \widehat{u}\|_{2; \bB_{\theta_2}^+} + \|\nabla \widehat{u}\|^2_{4; \bB_{\theta_2}^+} + \|\nabla \widehat{u}_{x^1}\|_{2; \bB_{\theta_2}^+}).
\]
Using~\eqref{eq:W14-almost} to control the $L^4$-norm and recalling~\eqref{eq:hori-W22-result}, we obtain
\begin{equation}\label{eq:W22}
\|\nabla \widehat{u}\|_{4; \bB_{\theta_2}^+} +  \|\nabla^2 \widehat{u}\|_{2; \bB_{\theta_2}^+} \leq C\big(p, \Sigma, H_0, H_1, L, D_{\ep, p}(u; \bB_{2\theta_2r}(x_0) \cap \bB)\big).
\end{equation}
Consequently, $A_{\alpha, i}(\cdot, \widehat{u}, \nabla\widehat{u})$ belongs to $(L^q \cap W^{1, s})(\bB_{\theta_2}^+)$ for all $q < \infty$ and $s < 2$, and we may now go back to~\eqref{eq:fermi-abbr} and integrate by parts to find that, for all $\psi \in W^{1, p}(\bB^+_{\theta_2}; \RR^3)$ such that $\psi = 0$ on $\partial \bB_{\theta_2} \cap \RR^2_+$ and $\psi^3 = 0$ on $\bT_{\theta_2}$, there holds
\[
0 = -\int_{\bT_{\theta_2}}A_{2, i}(\cdot, \widehat{u}, \nabla \widehat{u}) \psi^{i} + \int_{\bB^+_{\theta_2}}  \big[-
\paop{x^{\alpha}}\big(A_{\alpha, i}(\cdot, \widehat{u}, \nabla\widehat{u})\big) + A_i(\cdot, \widehat{u}, \nabla\widehat{u})\big] \psi^{i}.
\]
Recalling that~\eqref{eq:fermi-nondiv} holds pointwise in $\bB^+_{\theta_2}$, we see that the second integral on the right-hand side vanishes, and hence
\begin{equation*}
\int_{\bT_{\theta_2}}A_{2, i}(\cdot, \widehat{u}, \nabla \widehat{u}) \psi^{i} = 0, 
\end{equation*}
for all $\psi$ as above. As $g_{i, 3} \equiv 0$ for $i = 1, 2$, we deduce that $\widehat{u}^i_{x^2} = 0$ on $\bT_{\theta_2}$ in the trace sense for $i = 1, 2$.
\vskip 2mm
\noindent\textbf{Step 2: $C^{1, \frac{1}{2}}$-regularity and smoothness}
\vskip 2mm
The idea is, as in Sacks-Uhlenbeck~\cite{Sacks-Uhlenbeck81}, to view the left-hand side of~\eqref{eq:fermi-nondiv} as a perturbation of the Laplace operator. Consider for $q > 1$ the following space
\[
X_{q, \widehat{u}}^{(\theta_2)} = \{v \in W^{2, q}(\bB_{\theta_2}^+; \RR^3)\ |\ v = \widehat{u} \text{ on } \partial \bB_{\theta_2} \cap \RR^2_+, v_{x^2}^1 = v_{x^2}^2 = v^3 = 0 \text{ on } \bT_{\theta_2}\}.
\] 
For all $f \in L^q(\bB_{\theta_2}^+; \RR^3)$, there exists a unique $v \in X^{(\theta_2)}_{q, 0}$ such that $\Delta v = f$ and that 
\begin{equation}\label{eq:CZ-half}
\theta_2^{-2}\|v\|_{q; \bB_{\theta_2}^+} + \theta_2^{-1}\|\nabla v\|_{q; \bB_{\theta_2}^+} +  \|\nabla^2 v\|_{q;\bB_{\theta_2}^+} \leq K_q\|f\|_{q; \bB_{\theta_2}^+},
\end{equation}
with $K_q$ as in~\eqref{eq:CZ}. On the other hand, define $L:X^{(\theta_2)}_{q, \widehat{u}} \to L^{q}(\bB_{\theta_2}^+; \RR^3)$ by
\begin{equation*}
(Lw)^i = \Delta w^i + (p-2)\frac{(\wep)^{p-2}(r^2 + N)^{\frac{p}{2} -1}}{1 + (\wep)^{p-2}(r^2 + N)^{\frac{p}{2} -1}}\frac{\widetilde{\lambda}^{-2} g_{kl}(\widehat{u})\partial_{\alpha}\widehat{u}^i \partial_{\beta}\widehat{u}^k }{r^2 + N}\partial_{\alpha\beta}w^l,
\end{equation*}
and then $T_q: X^{(\theta_2)}_{q, \widehat{u}} \to X^{(\theta_2)}_{q, \widehat{u}}$ by 
\[
T_q(w) = w - \Delta^{-1}(Lw - b(\cdot, \widehat{u}, \nabla\widehat{u})).
\]
Note that since $\widehat{u} \in W^{2,2}(\bB^+_{\theta_2})$, by~\eqref{eq:f-estimate} and Sobolev embedding $b(\cdot, \widehat{u}, \nabla\widehat{u})$ lies in $L^q$. Also, $\widehat{u}$ is a fixed point of $T_2$. Next, by~\eqref{eq:CZ-half}, the definition of $L$ and~\eqref{eq:deviation}, and recalling our choice of $p_0$ in~\eqref{eq:p0-choice}, we see that both $T_4$ and $T_2$ are contractions. The unique fixed point of $T_4$ in $X_{4, \widehat{u}}^{(\theta_2)}$ must then coincide in $X_{2, \widehat{u}}^{(\theta_2)}$ with $\widehat{u}$, and hence $\widehat{u}\in W^{2, 4}(\bB_{\theta_2}^+; \RR^3)$, which embeds continuously into $C^{1, \frac{1}{2}}(\overline{\bB_{\theta_2}^+}; \RR^3)$. As in~\cite[Proposition 1.4]{Fraser2000}, we may then apply Schauder theory for linear elliptic systems to~\eqref{eq:fermi-nondiv} to obtain smoothness.
\end{proof}
\begin{rmk}\label{rmk:W22}
Taking into account~\eqref{eq:F-interpolate}, the estimate~\eqref{eq:W22} formulated in terms of $u$ implies
\[
r^{\frac{1}{2}}\|\nabla u\|_{4; \bB_{\frac{3}{4}\theta_2 r}(x) \cap \bB} + r\|\nabla^2 u\|_{2; \bB_{\frac{3}{4}\theta_2 r}(x) \cap \bB} \leq C(p, \Sigma, H_0, H_1, L, D_{\ep, p}(u; \bB_{2\theta_2r}(x) \cap \bB)),
\]
for all $x \in \partial \bB$. We emphasize that $\theta_2$ depends on $L$, among other things. There is also an interior version of this estimate, whose proof is much simpler as there is no need to flatten the constraint or to separately treat tangential and normal derivatives.
\end{rmk}
\subsection{A priori estimates for critical points}\label{subsec:a-priori}
In this section we derive several \textit{a priori} estimates that will play important roles in Section~\ref{sec:passage}, the main estimate being Proposition~\ref{prop:small-regular-2}. Most of the arguments here are adapted from~\cite[Section 1.2]{Fraser2000}.

The next proposition will later be used in conjunction with Remark~\ref{rmk:W22} with $r = \ep$ (Proposition~\ref{prop:rescale}), and is also an ingredient in the proof of Proposition~\ref{prop:small-regular-2}, which we will rely on to get estimates that hold on scales independent of $\ep$. 
\begin{prop}[See also~\cite{Fraser2000}, Lemma 1.5]
\label{prop:W14-W2q}
With $p_0$ as in Proposition~\ref{prop:regularity}, given $H_0, L > 0$, $\ep \in (0, 1]$ and $q \geq 2$, there exist $\theta_0 \in (0, \frac{1}{4})$ depending only on $L$ and $\Sigma$, and $p_1 \in (2, p_0]$ depending only on $q$ and $\Sigma$, such that if $u \in \cM_p$ is a critical point of $E_{\ep, p, f}$ with $p\leq p_1$, $\|f\|_{\infty} \leq H_0$, and if
\begin{equation}\label{eq:W14-assumption}
r^{2}\int_{\bB_{r}(x_0) \cap \bB} |\nabla u|^4 \leq L^4 \text{ for some }x_0 \in \partial \bB, r \in (0, r_{\bB}],
\end{equation}
then for all $\sigma < \rho \leq \theta_0$, there holds
\begin{align}\label{eq:W14-W2q}
r^{2 - \frac{2}{q}}&\|\nabla^2 u\|_{q; \bB_{\frac{3}{4}\sigma r}(x_0) \cap \bB} \leq C(\rho, \rho - \sigma, p, q, \Sigma, H_0, L) r^{\frac{1}{2}}\|\nabla u\|_{4; \bB_{\rho r}(x_0)\cap \bB}.
\end{align}
\end{prop}
\begin{proof}
The structure of the proof follows that of~\cite[Lemma 1.5]{Fraser2000}. By assumption and Sobolev embedding, there exists $\theta_0 \in (0, \frac{1}{4})$ depending only on $L, \Sigma$ such that 
\[
u(\overline{\bB_{\theta_0 r}(x_0)} \cap \overline{\bB}) \subseteq B_{\rho_\Sigma}(u(x_0)),
\] 
so we may let $U, V, \Psi, \Upsilon, g, Q, F$ and $\lambda$ be as before on similar occasions, and define 
\[
\widehat{u} = (\Psi \circ u \circ F)(r\ \cdot\ ) \text{ on }\bB^+_{\theta_0}.
\]
Then $\widehat{u}$ is a smooth solution to~\eqref{eq:fermi-nondiv} on $\bB^+_{\theta_0}$ with boundary conditions 
\begin{equation}\label{eq:fermi-bc}
\widehat{u}^{1, 2}_{x^2} = 0,\ \widehat{u}^3 = 0 \text{ on }\bT_{\theta_0}.
\end{equation}
Given $\sigma < \rho \leq \theta_0$, we let $\rho' = \frac{\rho + \sigma}{2}$ and take a cut-off function $\zeta \in C^{\infty}_{0}(\bB_{\rho'})$ satisfying $\zeta = 1$ on $\bB_{\sigma}$ and $\|\nabla^i \zeta \|_{\infty} \leq C(\rho - \sigma)^{-i}$, where $C$ is independent of $\sigma$ and $\rho$. By a reflection we can also arrange that $ \zeta_{x^2} = 0$ on $\bT_{\rho}$. Writing $\widehat{u}$ as $u$ for simplicity, we get for all $a \in \RR^3$ with $a^3 = 0$ that 
\begin{align}
|\Delta (\zeta (u-a))| \leq\ &  6A_0(p-2)|\nabla^2 (\zeta (u-a))| + C(\rho - \sigma)^{-1}|\nabla u| \nonumber\\
& + C(\rho - \sigma)^{-2}|u-a| + C(|\nabla u| + |\nabla u|^2). \label{eq:curved-Delta-estimate}
\end{align}
Under the boundary conditions~\eqref{eq:fermi-bc} and how $\zeta$ is chosen, standard elliptic theory gives
\[
\|\nabla^2(\zeta (u-a))\|_{q; \bB^+_{\rho}} \leq K_q \|\Delta (\zeta (u-a))\|_{q; \bB_{\rho}^+}.
\]
Combining this with~\eqref{eq:curved-Delta-estimate} gives
\begin{align*}
\| \nabla^2(\zeta (u-a)) \|_{q; \bB^+_{\rho}} \leq\ & 6A_0 K_q (p-2) \| \nabla^2 (\zeta (u-a)) \|_{q; \bB^+_{\rho}} + CK_q(\rho - \sigma)^{-1}\| \nabla u \|_{q; \bB^+_{\rho'}}\\
& + CK_q(\rho - \sigma)^{-2}\|u-a\|_{q; \bB^+_{\rho'}} + CK_q ( \|\nabla u\|_{q; \bB^+_{\rho'}} +\|\nabla u\|^2_{2q; \bB^+_{\rho'}}).
\end{align*}
Requiring that 
\begin{equation*}
6A_0 K_q (p_1 - 2) < \frac{1}{2},
\end{equation*}
and choosing $a^1, a^2$ appropriately to apply Poincar\'e's inequality, we get 
\begin{align}\label{eq:W2q-W12q}
\|\nabla^2 u\|_{q; \bB_{\sigma}^+} \leq\ & C[1 + (\rho - \sigma)^{-2}\rho] \|\nabla u\|_{q; \bB_{\rho'}^+} + C \|\nabla u\|_{2q; \bB_{\rho'}^+}^2.
\end{align}
Repeating the above argument with $2, \rho', \rho$ in place of $q, \sigma, \rho'$, respectively, we see that, further requiring
\begin{equation*}
6A_0 K_2 (p_1 - 2) < \frac{1}{2},
\end{equation*}
we get 
\begin{equation}\label{eq:W22-W14}
\|\nabla^2 u\|_{2; \bB_{\rho' }^+} \leq C[ 1+ (\rho - \sigma)^{-2}\rho ] \|\nabla u\|_{2; \bB_{\rho }^+}
+ C\|\nabla u\|_{4; \bB^+_{\rho }}^2.
\end{equation}
The Sobolev embedding of $W^{2, 2}$ into $W^{1, q}$ and $W^{1, 2q}$ allows us to string~\eqref{eq:W2q-W12q} and~\eqref{eq:W22-W14} together and obtain (reverting to the notation $\widehat{u}$)
\[
\|\nabla\widehat{u}\|_{2q; \bB^+_{\sigma}} + \|\nabla^2\widehat{u}\|_{q; \bB^+_{\sigma}} \leq C(\rho, \rho - \sigma, p, q, \Sigma, H_0, L)\|\nabla\widehat{u}\|_{4; \bB_\rho^+},
\]
which implies~\eqref{eq:W14-W2q}.
\end{proof}
\begin{rmk}\label{rmk:interior-W2q}
Decreasing $p_1$ depending only on $q$ if necessary, a similar proof as above gives an interior version of Proposition~\ref{prop:W14-W2q} where $\bB_r(x_0)\cap \bB$ is replaced with some $\bB_r(x_0) \subseteq \bB$. In this case, since there is no need to locally flatten the constraint, we can allow any $\sigma < \rho \leq 1$. Also, the counterpart of the estimate~\eqref{eq:W14-W2q} will be independent of $\Sigma$. 
\end{rmk}

Remark~\ref{rmk:W22} gives~\eqref{eq:W14-assumption} with a suitable $L$ provided we have a bound of the form~\eqref{eq:W14-scale} to begin with. However, in application we only expect a bound on $D_{\ep, p}$, which does not imply~\eqref{eq:W14-scale} unless $r$ is controlled by $\ep$. The next proposition gives estimates on scales independent of $\ep$ when the Dirichlet energy does not concentrate. The statement and its proof are analogous to~\cite[Lemmas 1.6 and 1.7]{Fraser2000}. 

\begin{prop}[See also~\cite{Fraser2000}, Lemmas 1.6 and 1.7]
\label{prop:small-regular-2}
Suppose $u\in \cM_p$ is a critical point of $E_{\ep, p, f}$, with $\ep \in (0, 1]$, $p \in (2, p_1]$ and $\|f\|_{\infty} \leq H_0$, where $p_1$ is given by Proposition~\ref{prop:W14-W2q} and Remark~\ref{rmk:interior-W2q} with $q = 4$. There exists $\eta \in (0, \frac{1}{16})$ depending only on $p, \Sigma$ and $H_0$ such that if 
\begin{equation}\label{eq:small-assumption}
\int_{\bB_{r}(x_0)\cap \bB} |\nabla u|^2 < \eta \text{ for some }x_0 \in \partial \bB, r \in (0, r_{\bB}], 
\end{equation}
then we have 
\begin{equation}\label{eq:local-lipschitz}
r|\nabla u| \leq 8 \text{ on }\overline{\bB_{\frac{r}{2}}(x_0)} \cap \overline{\bB}.
\end{equation}
Moreover, letting $\theta_0$ be as in Proposition~\ref{prop:W14-W2q} with $L = 1$, we have
\begin{equation}\label{eq:ep-reg-estimate}
r^{\frac{1}{2}}\| \nabla u \|_{4; \bB_{\frac{1}{8}\theta_0 r}(x_0) \cap \bB} + r^{\frac{3}{2}}\| \nabla^2 u \|_{4; \bB_{\frac{1}{8}\theta_0 r}(x_0) \cap \bB} \leq C(p,\Sigma, H_0)\|\nabla u\|_{2; \bB_{\frac{1}{2}\theta_0 r}(x_0) \cap \bB}.
\end{equation}
\end{prop}
\begin{proof}
Let $x_1$ be a point where $\sup_{\bB_{r}(x_0) \cap \bB}(r - |x - x_0|)|\nabla u(x)|$ is attained, and set $e_0 = |\nabla u(x_1)|$ and $\rho_0 = \frac{1}{2}(r - |x_1 - x_0|)$. Then 
\begin{equation}\label{eq:local-bound}
|\nabla u(x)| \leq 2 e_0, \text{ on }\bB_{\rho_0}(x_1)\cap \bB.
\end{equation}
Now assume that $\rho_0 e_0\geq 2$, so that by~\eqref{eq:local-bound} we have 
\begin{equation}\label{eq:local-W14}
e_0^{-2}\int_{\bB_{e_0^{-1}}(x_1)\cap \bB} |\nabla u|^4 \leq L^4
\end{equation}
for some numerical constant $L$, which in turn determines a $\theta_0 \in (0, \frac{1}{4})$ as in Proposition~\ref{prop:W14-W2q}. Next, by~\eqref{eq:small-assumption}, there exists some numerical constant $C_0$ such that
\begin{equation}\label{eq:du-lowerbound}
\inf_{\bB_{\eta^{\frac{1}{3}}e_0^{-1}}(x_1)\cap \bB}|\nabla u|^2 < C_0\eta^{\frac{1}{3}}e_0^2.
\end{equation}
Since $|\nabla u(x_1)|^2 = e_0^2$, we deduce that, with $\alpha = \frac{1}{2}$,
\begin{equation}\label{eq:holder-lowerbound}
[|\nabla u|^2]_{\alpha; \bB_{\eta^{\frac{1}{3}}e_0^{-1}}(x_1)\cap \bB} \geq \frac{e_0^2 (1 - C_0\eta^{\frac{1}{3}})}{\eta^{\frac{\alpha}{3}}e_0^{-\alpha}} > \frac{1}{2}\eta^{-\frac{\alpha}{3}}e_0^{2 + \alpha},
\end{equation}
provided $C_0\eta^{\frac{1}{3}} < \frac{1}{2}$. On the other hand, with $q = 4$, by~\eqref{eq:local-bound} and Sobolev embedding,
\begin{align}
(e_0^{-1})^{2 + \alpha} [|\nabla u|^2]_{\alpha; \bB_{\eta^{\frac{1}{3}}e_0^{-1}}(x_1)\cap \bB} \leq\ & C(e_0^{-1})^{1 + \alpha} [\nabla u]_{\alpha; \bB_{\eta^{\frac{1}{3}} e_0^{-1}}(x_1)\cap \bB}\nonumber\\
\leq\ & C(e_0^{-1})^{2 - \frac{2}{q}}\| \nabla^2 u \|_{q; \bB_{\eta^{\frac{1}{3}} e_0^{-1}}(x_1)\cap \bB}.\label{eq:holder-sobolev-transition}
\end{align}
Next suppose in addition that $\eta^{\frac{1}{3}} < \frac{\theta_0}{8}$. If $1 - |x_1| \geq \frac{\theta_0}{4}e_0^{-1}$, then by Remark~\ref{rmk:interior-W2q} and~\eqref{eq:local-W14},
\[
(e_0^{-1})^{2 - \frac{2}{q}}\| \nabla^2 u \|_{q; \bB_{\eta^{\frac{1}{3}} e_0^{-1}}(x_1)}  \leq C(\theta_0, p, H_0, L)e_0^{-\frac{1}{2}}\|\nabla u\|_{4; \bB_{\frac{\theta_0}{4}e_0^{-1}}(x_1) \cap \bB} \leq \widetilde{C}(\theta_0, p,  H_0, L).
\]
On the other hand, if $1 - |x_1| <  \frac{\theta_0}{4}e_0^{-1}$, then we let $x_2 = \frac{x_1}{|x_1|}$ and note that
\begin{align*}
(e_0^{-1})^{2 - \frac{2}{q}}\| \nabla^2 u \|_{q; \bB_{\eta^{\frac{1}{3}} e_0^{-1}}(x_1)\cap \bB} \leq\ & (e_0^{-1})^{2 - \frac{2}{q}}\| \nabla^2 u \|_{q; \bB_{(\eta^{\frac{1}{3}} + \frac{\theta_0}{4}) e_0^{-1}}(x_2)\cap \bB}\\
 \leq\ & C(\theta_0, p, \Sigma, H_0, L) e_0^{-\frac{1}{2}}\|\nabla u\|_{4; \bB_{\theta_0 e_0^{-1}}(x_2)\cap \bB}\\
 \leq\ &  \widetilde{C}(\theta_0, p, \Sigma, H_0, L),
\end{align*}
where we used Proposition~\ref{prop:W14-W2q} with $\sigma = \frac{\theta_0}{2}$ and $\rho = \theta_0$ to obtain the second line. Next we combine the above estimates with~\eqref{eq:holder-lowerbound} and~\eqref{eq:holder-sobolev-transition}. Since $\theta_0$ depends only on $\Sigma$ and $L$, and the latter is merely a numerical constant, we obtain
\[
\eta^{-\frac{\alpha}{3}} \leq C(p, \Sigma, H_0),
\]
subject to the requirements $C_0\eta^{\frac{1}{3}} < \frac{1}{2}$ and $\eta^{\frac{1}{3}} < \frac{\theta_0}{8}$. In any case, provided $\eta$ is small enough depending only on $p, \Sigma$ and $H_0$, we have $\rho_0 e_0 \leq 2$, in which case~\eqref{eq:local-lipschitz} follows. 

To get the second conclusion, note that~\eqref{eq:small-assumption} and~\eqref{eq:local-lipschitz} implies~\eqref{eq:W14-assumption} with $L = 1$ and with $\frac{r}{2}$ in place of $r$. Let $\theta_0, \widehat{u}$ and $\zeta$ be as in the proof of Proposition~\ref{prop:W14-W2q} with $r$ in the definition of $\widehat{u}$ replaced with $\frac{r}{2}$. Then by~\eqref{eq:local-lipschitz} again there holds
\begin{equation}\label{eq:local-lipschitz-fermi}
\sup_{\bB_{\theta_0}^+} |\nabla \widehat{u}| \leq C,
\end{equation}
with $C$ depending only on $\Sigma$. Consequently the term $|\nabla u|^2$ in~\eqref{eq:curved-Delta-estimate} can be absorbed into the term $|\nabla u|$ that precedes it. Repeating the argument that follows with $q = 4$, $\sigma = \frac{\theta_0}{2}$ and $\rho = \theta_0$ then yields~\eqref{eq:W2q-W12q} and~\eqref{eq:W22-W14} without the squared terms at the end. The Sobolev embedding from $W^{2, 2}$ into $W^{1, 4}$ bridges these two estimates to give
\[
\|\nabla \widehat{u}\|_{4; \bB_{\frac{\theta_0}{2}}^+} +  \|\nabla^2 \widehat{u}\|_{4; \bB_{\frac{\theta_0}{2}}^+} \leq C(p, \Sigma, H_0)\|\nabla\widehat{u}\|_{2; \bB_{\theta_0}^+}.
\]
Combining this with~\eqref{eq:local-lipschitz-fermi} gives~\eqref{eq:ep-reg-estimate}.
\end{proof}
\begin{rmk}\label{rmk:small-regular-int}
Again, Proposition~\ref{prop:small-regular-2} has an interior version with $\bB_{r}(x_0) \cap \bB$ replaced by some $\bB_{r}(x_0) \subseteq \bB$. In this case both $\eta$ and the analogue of~\eqref{eq:ep-reg-estimate} are independent of $\Sigma$, and the latter holds with the $\theta_0$-factor removed from the radii of the balls. 
\end{rmk}
\begin{rmk}\label{rmk:ep-zero}
With essentially the same proofs, Propositions~\ref{prop:W14-W2q} and~\ref{prop:small-regular-2}, as well as their interior counterparts, remain valid when $\ep = 0$, that is, when $u: (\overline{\bB}, \partial\bB) \to (\RR^3, \Sigma)$ is assumed to be a smooth solution to~\eqref{eq:perturbed-EL} with $\ep = 0$ and $\|f\|_{\infty} \leq H_0$. In particular, in this case both the threshold $\eta$ and the estimate~\eqref{eq:ep-reg-estimate} are independent of $p$.
\end{rmk}

An important consequence of Proposition~\ref{prop:small-regular-2} is an energy lower bound independent of $\ep$ for non-constant critical points of $E_{\ep, p, f}$. The proof is similar enough to~\cite[Theorem 1.8]{Fraser2000} and is therefore omitted.
\begin{prop}[See~\cite{Fraser2000}, Theorem 1.8]
\label{prop:uniform-lower-bound} 
With $p_1$ as in Proposition~\ref{prop:small-regular-2}, there exist $p_2 \in (2, p_1]$, depending only on $\Sigma$, and, for all $p \in (2, p_2]$, some $\beta > 0$ depending only on $p, \Sigma$ and $H_0$ such that if $\ep \in (0, 1]$, $\|f\|_{\infty} \leq H_0$ and $u$ is a critical point of $E_{\ep, p, f}$ with $D(u) < \beta$, then $u$ is constant.
\end{prop}
\subsection{A maximum principle}\label{subsec:max-principle}
The estimates derived so far are on the derivatives of $u$. The main result of this section is an $L^{\infty}$-estimate independent of $\ep$. Below we let $\Sigma'$ be a closed surface which has mean curvature $H_{\Sigma'}$ bounded below by some $H_0 > 0$, and is strictly convex in the sense that its second fundamental form $A_{\Sigma'}$ is everywhere positive definite. Then $\Sigma'$ bounds a convex domain $\Omega'$, and we assume that $\Sigma \subseteq \overline{\Omega'}$. Letting $d = \dist(\cdot, \Omega')$, which is a smooth convex function on $\RR^3 \setminus \overline{\Omega'}$, we define, for $t > 0$,
\[
\Sigma'_t = \{y \in \RR^3\ |\ d(y) = t\},\ \Omega'_t = \text{the bounded open region enclosed by }\Sigma'_t.
\] 
Given $H \in (0, H_0)$, there exists $t_0 \in (0, 1/4)$ such that 
\[
H_{\Sigma'_t} \geq \frac{H_0 + H}{2} \text{ for all } t \in (0, t_0],
\]
and we choose $f:\RR^3 \to [0, H]$ to be a smooth function satisfying 
\[
f(y) = \left\{
\begin{array}{cl}
H & \text{ if } y \in \overline{\Omega'_{t_0/4}},\\
0 & \text{ if } y \not\in \Omega'_{3t_0/4}.
\end{array}
\right.
\]
\begin{prop}\label{prop:max-principle}
Let $f$ be as above and suppose $r \in [0, 1]$.
\vskip 1mm
\begin{enumerate}
\item[(a)] Let $u \in \cM_p$ be  a smooth critical point of $E_{\ep, p, rf}$. Then $u(\overline{\bB}) \subseteq \overline{\Omega'_{t_0}}$.
\vskip 1mm
\item[(b)] Suppose $u: (\overline{\bB},\partial\bB) \to (\overline{\Omega'_{t_0}}, \Sigma)$ is a smooth solution to
\begin{equation}\label{eq:cmc-rf}
\left\{
\begin{array}{ll}
\Delta u = rf(u) u_{x^1} \times u_{x^2} & \text{ in }\bB,\\
u_r(x) \perp T_{u(x)}\Sigma & \text{ for all }x \in \partial \bB.
\end{array}
\right.
\end{equation}
Then $u(\overline{\bB}) \subseteq \overline{\Omega'}$.
\end{enumerate}
\end{prop}
\begin{proof}
For part (a), let $C_+ = \{x \in \overline{\bB}\ |\ d(u(x)) > t_0\}$. Since $u(\partial\bB) \subseteq \Sigma$, we see that $d(u) = t_0$ on $\partial C_+$. Next, our choice of $f$ implies that
\begin{equation*}\label{eq:harmonic}
\Div\Big( [1 + \ep^{p-2}(1 + |\nabla u|^2)^{\frac{p}{2} - 1}] \nabla u\Big) = 0 \text{ on }C_+.
\end{equation*}
Therefore, by the convexity of $d$, we have
\[
\Div\Big( [1 + \ep^{p-2}(1 + |\nabla u|^2)^{\frac{p}{2} - 1}] \nabla(d(u))\Big) \geq 0 \text{ on }C_+.
\]
The maximum principle then forces $C_+$ to be empty, and hence $u(x) \in \overline{\Omega'_{t_0}}$ for all $x \in \overline{\bB}$.

For part (b), we first note that $u$ is weakly conformal by~\eqref{eq:cmc-rf} and a Hopf differential argument (see~\cite[page 27]{Struwe88}; the function $\Phi$ there is still holomorphic in our case). Next, we define $C'_+ = \{x \in \overline{\bB} \ |\ d(u(x)) > 0\}$ and observe that $d(u) = 0$ on $\partial C'_+$ since $u(\partial\bB) \subseteq \Sigma$. Also, by assumption, $0 < d(u) \leq t_0$ on $C'_+$. Similar to the proof of the Hopf boundary point lemma, we let $F = e^{ad}$, with $a > 0$ to be determined, and compute, on $\RR^3 \setminus \overline{\Omega'}$,
\begin{align*}
\nabla^2 F =\ & e^{ad}(a\nabla^2 d + a^2 \nabla d \otimes \nabla d) = ae^{ad}(A_{\Sigma'_d} + a \nabla d \otimes \nabla d).
\end{align*}
Thus, letting $\lambda_1 \leq \lambda_2 \leq \lambda_3$ be the eigenvalues of $\nabla^2 F$, provided $a$ is large enough, we get that 
\begin{equation}\label{eq:two-convex}
\lambda_1 + \lambda_2 = ae^{ad}H_{\Sigma'_d} \geq ae^{ad}\frac{H_0 + H}{2}, \text{ when }d \in (0, t_0].
\end{equation}
Now, recalling that $d(u) \in (0, t_0]$ on $C'_+$, by a direct computation using~\eqref{eq:two-convex}, the weak conformality of $u$ and our choice of $f$, we have on $C'_+$ that
\begin{align}
\Delta (F(u)) =\ & ae^{ad(u)}(\nabla d)_u\cdot (\Delta u) + (\nabla^2 F)_u(u_{x^1}, u_{x^1}) + (\nabla^2 F)_u(u_{x^2}, u_{x^2}) \nonumber\\
\geq\ & - ae^{ad(u)} rf(u)|u_{x^1} \times u_{x^2}| + \frac{|\nabla u|^2}{2} (\lambda_1+ \lambda_2)\nonumber\\
\geq \ & ae^{ad(u)}\frac{H_0 - H}{2}\frac{|\nabla u|^2}{2} \geq 0.
\end{align}
Thus, by the maximum principle, $\sup_{C'_+} F(u) \leq \sup_{\partial C'_+}F(u) = 1$. Consequently $C'_+$ must be empty, and we are done.
\end{proof}
\section{Second variation and index}\label{sec:2nd-var}
Suppose $p_0, p_1$ and $p_2$ are respectively as in Propositions~\ref{prop:regularity},~\ref{prop:small-regular-2} and~\ref{prop:uniform-lower-bound}. Below, and throughout the remainder of the paper, we fix a single $p \in (2, p_2]$ (recall that $2 < p_2 \leq p_1 \leq p_0 < 3$). For now, $f$ is any bounded smooth function from $\RR^3$ to $\RR$ with $\int_{\Omega}f > 0$. 

The purpose of this section is to derive some useful local estimates for $E_{\ep, p, f}$ near a critical point, which will be crucial in the proof of Proposition~\ref{prop:bypassing}. First, a direct computation shows that at a critical point $u$ of $E_{\ep, p, f}$, which is smooth by our choice of $p$, the second variation formula is given as follows. For all $\psi \in T_u\cM_p$,
\begin{align}
\delta^2 E_{\ep, p, f}(u)(\psi, \psi) =\ & \int_{\bB} [1 + \ep^{p-2}(1 + |\nabla u|^2)^{\frac{p}{2} - 1}] |\nabla \psi|^2 + \ep^{p-2}(p-2)(1 + |\nabla u|^2)^{\frac{p}{2}-2}\langle \nabla u, \nabla \psi \rangle^2 \nonumber\\
&+ \int_{\bB} \langle(\nabla f)_u , \psi\rangle  \langle \psi, u_x \times u_y \rangle + f(u)\langle \psi , \psi_{x^1} \times u_{x^2} + u_{x^1} \times \psi_{x^2} \rangle\nonumber\\
&+ \int_{\partial \bB} [1 + \ep^{p-2}(1 + |\nabla u|^2)^{\frac{p}{2} - 1}]A_\Sigma^{u_r}(\psi,\psi)d\theta, \label{eq:2nd-variation-formula}
\end{align}
where for $x\in \partial \bB$, $A_\Sigma^{u_r}|_x$ is the second fundamental form of $\Sigma$ at $u(x)$ in the direction of $u_r(x)$, the latter being normal to $\Sigma$ by~\eqref{eq:perturbed-EL}. Since $u$ is smooth, $\delta^2 E_{\ep, p, f}(u)$ extends to a bounded, symmetric bilinear form satisfying a G$\mathring{a}$rding inequality on the Hilbert space
\[
H_u = \{\psi \in W^{1, 2}(\bB; \RR^3)\ |\ \psi|_{\partial \bB}(x) \in T_{u(x)}\Sigma \text{ for a.e. }x \in \partial \bB  \},
\]
where $\psi|_{\partial \bB}$ denotes the trace of $\psi$ on $\partial\bB$. Standard Hilbert space theory implies that $H_u$ has a basis which is orthonormal with respect to the $L^2$ inner product and consists of eigenfunctions of $\delta^2 E_{\ep, p, f}(u)$, with the corresponding eigenvalues all real and forming a sequence tending to infinity. Furthermore, polarizing~\eqref{eq:2nd-variation-formula}, we see by elliptic regularity that each eigenfunction is smooth on $\overline{\bB}$.
\begin{defi}\label{defi:index}
Given a critical point $u$ of $E_{\ep, p, f}$, we define its index, $\Ind_{\ep, p, f}(u)$, to be the number of negative eigenvalues of $\delta^2 E_{\ep, p, f}(u)$ counted with multiplicity. Moreover, denote by $V_u^-$ the sum of the negative eigenspaces of $\delta^2 E_{\ep, p, f}(u)$, and by $V_u^\perp$ the set of elements in $T_u \cM_p$ which are $L^2$-orthogonal to $V_u^-$. Then both $V_u^-$ and $V_u^\perp$ are closed subspaces of $T_u\cM_p$, with $\dim V_u^- = \Ind_{\ep, p, f}(u)$, and we have 
\[
T_u\cM_p = V^{-}_u \oplus V^{\perp}_u.
\]
Accordingly, given $\xi \in T_u\cM_p$, we write $\xi = \xi^- + \xi^{\perp}$. Note that since both direct summands are closed, $\|(\cdot)^-\|_{1, p} + \|(\cdot)^\perp\|_{1, p}$ and $\|\cdot\|_{1, p}$ are equivalent norms on $T_u\cM_p$.
\end{defi}

\begin{defi} Given a smooth solution $u: (\overline{\bB}, \partial\bB) \to (\RR^3, \Sigma)$ to 
\[
\left\{
\begin{array}{ll}
\Delta u = f(u)u_{x^1} \times u_{x^2}& \text{ in }\bB,\\
u_r(x) \perp T_{u(x)}\Sigma & \text{ for all }x \in \partial \bB,
\end{array}
\right.
\]
we let $\delta^2 E_{f}(u)$ denote the bilinear form on $H_u$ obtained by setting $\ep = 0$ in~\eqref{eq:2nd-variation-formula}, and let $\Ind_f(u)$ denote its number of negative eigenvalues counted with multiplicity.
\end{defi}

To state the main result of this section, we need some additional notation. Let $u \in \cM_p$ be a critical point of $E_{\ep, p, f}$ and suppose $\cA$ is a simply-connected neighborhood of $u$. Then for $s$ small enough, $\Theta_u: \psi \mapsto \exp^h_u(\psi)$ restricts to a chart on $\cB_{s} = \{\psi \in T_u\cM_p\ |\ \|\psi\|_{1,p} < s\}$ with image contained in $\cA$. Given a local reduction $E^{\cA}_{\ep, p, f}$, we write $\widetilde{E} = E_{\ep, p, f}^{\cA} \circ \Theta_u$. Then $(\delta \widetilde{E})_0 = 0 \text{ and }(\delta^2 \widetilde{E})_0 = (\delta^2 E_{\ep, p, f})_u$. Also, by Lemma~\ref{lemm:C2}, $\delta^2 \widetilde{E}$ is continuous from $\cB_s$ into the space of bounded bilinear forms on $T_u\cM_p$.
\begin{prop}\label{prop:coordinates}
There exist $r \in (0, \frac{s}{3})$, $a > 0$ and $\theta \in (0, 1)$ depending on $u$ such that for any local reduction $E^{\cA}_{\ep, p, f}$, in the above notation we have the following.
\begin{enumerate}
\item[(a)] For all $\xi \in T_u\cM_p$ such that $\|\xi\|_{1, p} \leq r$ and  $\|\xi^\perp\|_{1,p} \leq \theta\|\xi^-\|_{1, p}$, we have
\[
\widetilde{E}(0) - \widetilde{E}(\xi) \geq a\| \xi^- \|_{1, p}^2.
\]
\vskip 1mm
\item[(b)] For all $\varphi \in T_u\cM_p$ and $\xi^- \in V_u^-$ such that $\|\varphi\|_{1, p} \leq r$, $\|\xi^-\|_{1,p } = 1$ and $(\delta \widetilde{E})_\varphi(\xi^-) \leq 0$, we have 
\[
\widetilde{E}(\varphi) - \widetilde{E}(\varphi + t\xi^-) \geq at^2, \text{ for }0 \leq t \leq r.
\]
\end{enumerate} 
\end{prop}
\begin{proof}
By the last paragraph of Section~\ref{subsec:perturbed}, it suffices to prove (a) and (b) for just one local reduction. Since $(\delta^2 \widetilde{E})_0$ is a bounded bilinear form on $T_u\cM_p$, there exists $C_u > 0$ such that 
\[
|(\delta^2 \widetilde{E})_0(\xi, \xi)| \leq C_u \|\xi\|_{1, p}^2, \text{ for all }\xi \in T_u\cM_p. 
\]
On the other hand, since $V_u^-$ is finite-dimensional, the $W^{1, p}$- and $L^2$-norms restricted to $V_u^-$ are equivalent, and hence there exists $c_u > 0$ such that 
\[
(\delta^2 \widetilde{E})_{0}(\xi, \xi)  \leq  -c_u \|\xi\|_{1, p}^2 \text{ for all }\xi \in V_u^-. 
\]
Next, by the continuity of $\delta^2 \widetilde{E}$, provided $r \in (0, \frac{s}{3})$ is sufficiently small, we have
\[
|(\delta^2 \widetilde{E})_{\psi}(\xi, \xi) - (\delta^2 \widetilde{E})_{0}(\xi, \xi)| \leq \frac{c_u}{16} \|\xi\|_{1, p}^2, \text{ for all }\xi \in T_u\cM_p \text{ and } \|\psi\|_{1, p}\leq 2r.
\]
Thus, given $\xi$ as in (a), we have by Taylor expansion that
\begin{align}
\widetilde{E}(\xi) =\ & \widetilde{E}(0) + \frac{1}{2}(\delta^2 \widetilde{E})_{0}(\xi^{-}, \xi^{-}) + \frac{1}{2}(\delta^2 \widetilde{E})_{0}(\xi^{\perp}, \xi^{\perp})\nonumber\\
& + \int_{0}^1 (1 - t)\big( (\delta^2 \widetilde{E})_{t\xi}(\xi, \xi) - (\delta^2 \widetilde{E})_{0}(\xi, \xi) \big) dt \nonumber\\
\leq\ & \widetilde{E}(0) - \frac{c_u}{2} \|\xi^{-}\|_{1, p}^2 + \frac{C_u}{2} \| \xi^{\perp} \|_{1, p}^2 + \frac{c_u}{32} \|\xi\|_{1, p}^{2}, \label{eq:compare-with-center}
\end{align}
where in the first line we used the fact that $(\delta^2\widetilde{E})_{0}(V_u^-, V_u^\perp) = 0$. Now, under the assumption that $\|\xi^\perp\|_{1,p }\leq \theta \|\xi^-\|_{1, p}$, we have $\|\xi\|_{1, p} \leq 2\|\xi^-\|_{1, p}$. Thus we deduce from~\eqref{eq:compare-with-center} that
\[
\widetilde{E}(\xi)  \leq \widetilde{E}(0) - \frac{1}{2}(c_u - C_u\theta^2 - \frac{c_u}{4}) \|\xi^-\|_{1, p}^2.
\]
Choosing $\theta$ so that $\theta^2 < \frac{c_u}{4C_u}$, we get (a) with $a = \frac{c_u}{4}$.

To prove (b), note that with $\varphi, \xi^-$ and $t$ as in the statement, we have $\varphi + t\xi^- \in \cB_{2r}$. We then compute,
\begin{align}
\widetilde{E}(\varphi + t\xi^{-}) -\widetilde{E}(\varphi) =\ & t (\delta\widetilde{E})_{\varphi}(\xi^{-})  +\frac{t^2}{2} (\delta^2 \widetilde{E})_{0}(\xi^{-}, \xi^{-}) \nonumber\\
&+ \int_{0}^t ( t- s)\big( (\delta^2 \widetilde{E})_{\varphi + s\xi^{-}}(\xi^{-}, \xi^{-}) - (\delta^2 \widetilde{E})_{0}(\xi^{-}, \xi^{-}) \big) ds. \nonumber
\end{align}
Thus, since $(\delta \widetilde{E})_{\varphi}(\xi^-) \leq 0$ by assumption,
\[
\widetilde{E}(\varphi) - \widetilde{E}(\varphi + t\xi^-) \geq \frac{t^2}{2} c_u \|\xi^-\|_{1, p}^2 - \frac{t^2}{2} \frac{c_u}{16} \|\xi^{-}\|_{1, p}^{2} \geq \frac{c_u}{4}t^2,
\]
which gives (b) with $a = \frac{c_u}{4}$.
\end{proof}

\section{Existence of critical points of the perturbed functional}
\label{sec:existence-perturbed}
We continue to assume that $p \in (2, p_2]$ and that $f$ is any smooth, real-valued function on $\RR^3$ with $\|f\|_{\infty} \leq H_0$ and $\int_{\Omega}f > 0$. Specific choices of $H_0$ and $f$ will be made right before Proposition~\ref{prop:existence-with-index}, which is the main result of this section. It gives non-constant critical points of the perturbed functional with $D_{\ep, p}$ bounded uniformly in $\ep$, index at most $1$, and image contained in a fixed compact set independent of $\ep$.

To set the stage for the min-max construction, we recall that each continuous path $\gamma:[0, 1] \to \cM_p$ starting and ending at constants induces a continuous map $F_\gamma: (B^3, \partial B^3) \to (\RR^3, \Sigma)$, and define
\[
\cP = \{\gamma \in C^0([0, 1]; \cM_p)\ |\ \gamma(0) = \text{constant, }\gamma(1) = \text{constant, }\deg(F_\gamma|_{\partial B^3}) = 1\},
\]
whose elements we refer to as (admissible) sweepouts. The construction in the proof of~\cite[Lemma 2.3]{Struwe88} shows that $\cP$ is non-empty. The min-max value is then defined as
\[
\omega_{\ep, p, f} = \inf_{\gamma \in \cP}\Big[\sup_{t \in [0, 1]} E_{\ep, p, f}(\gamma(t), \gamma|_{[0, t]})\Big].
\]
For $C > 0$, we let 
\[
\cK_{\ep, p, f, C} = \{u \in \cM_p\ |\ \delta E_{\ep, p, f}(u) = 0,\ D_{\ep, p}(u) \leq C, E_{\ep, p, f}(u, \gamma) = \omega_{\ep, p, f} \text{ for some }\gamma \in \cE(u)\}.
\]
By our choice of $p$ and Proposition~\ref{prop:regularity}, every $u \in \cK_{\ep, p, f, C}$ is smooth. Also, by Proposition~\ref{prop:PS} and the estimate~\eqref{eq:segment-vol}, we see that $\cK_{\ep, p, f, C}$ is compact. Next, for $\alpha, C > 0$ we introduce the subcollection $\cP_{\ep, p, f, \alpha, C}$ of $\cP$ consisting of sweepouts $\gamma$ such that 
\begin{enumerate}
\item[(i)] $E_{\ep, p, f}(\gamma(t), \gamma|_{[0, t]}) \leq \omega_{\ep, p, f} + \alpha$ for all $t \in [0, 1]$,
\vskip 1mm
\item[(ii)] $D_{\ep, p}(\gamma(t)) \leq C$ whenever $E_{\ep, p, f}(\gamma(t), \gamma|_{[0, t]}) \geq \omega_{\ep, p, f} - \alpha$.
\end{enumerate}

As in~\cite{ChengZhou-cmc}, to apply the mountain pass theorem, or rather the idea of its proof, it is important to produce minimizing sequences of sweepouts satisfying a uniform $D_{\ep, p}$-bound on almost-maximal slices, and we again use the monotonicity trick of Struwe (see~\cite{Struwe88}, especially equation (4.2) and Lemma 4.1) to achieve that. The result is summarized in the lemma below, whose proof is omitted as it is essentially the same as~\cite[Proposition 3.2, Lemma 3.3]{ChengZhou-cmc}. We only mention that the following identity should be used instead of equation (3.2) in~\cite{ChengZhou-cmc}: For all $u \in \cM_p$, $\gamma \in \cE(u)$ and $0 < r' < r \leq 1$, 
\[
\frac{1}{r - r'}\big( \frac{E_{\ep, p, r'f}(u, \gamma)}{r'} - \frac{E_{\ep, p, rf}(u, \gamma)}{r} \big) = \frac{1}{r'\cdot r} D_{\ep, p}(u).
\]
\begin{lemm}\label{lemm:monotonicity-trick}
We have the following.
\begin{enumerate}
\item[(a)] Given a sequence $\ep_j \to 0$, for almost every $r_0 \in (0, 1]$ there exists $c> 0$ and a subsequence of $\ep_j$, which we do not relabel, such that, for all $j$, the function $r \mapsto \frac{\omega_{\ep_j, p, rf}}{r}$ is differentiable at $r = r_0$ and
\[
0 \leq \frac{d}{dr}\left( -\frac{\omega_{\ep_j, p, rf}}{r} \right)\Big|_{r = r_0} \leq c.
\]
\vskip 1mm
\item[(b)] Suppose for some $\ep, r_0 \in (0, 1]$ and $c > 0$ we have
\[
\frac{d}{dr}\left( -\frac{\omega_{\ep, p, rf}}{r} \right)\Big|_{r = r_0} \leq c.
\] 
Then for $k$ sufficiently large, the collection $\cP_{\ep, p, r_0f, \frac{r_0}{k}, 8cr_0^2}$ is non-empty.
\end{enumerate}
\end{lemm}

Lemma~\ref{lemm:monotonicity-trick} feeds into Proposition~\ref{prop:good-PS} below to produce critical points of the perturbed functional. Before that, two more preliminary results are needed. First, we need to guarantee that the critical points thus obtained are non-constant (Lemma~\ref{lemm:non-trivial-lower-bound}). Secondly, we observe that $E_{\ep, p, f}$ has a single local reduction near the level set of $\omega_{\ep, p, f}$ (Lemma~\ref{lemm:local-simple}). 
\begin{lemm}[See also~\cite{Struwe88}, Lemma 2.4]
\label{lemm:non-trivial-lower-bound}
There exist $\eta_1, \eta_2 > 0$, depending only on $p, \Sigma, \ep, H_0$, such that given $\gamma \in \cP$, if $D_{\ep, p}(\gamma(t_0)) < \eta_1$ for some $t_0 \in [0, 1]$, then
\[
E_{\ep, p, f}(\gamma(t_0), \gamma|_{[0, t_0]}) \leq \sup_{t \in [0, 1]} E_{\ep, p, f}(\gamma(t), \gamma|_{[0, t]}) - \eta_2.
\]
\end{lemm}
\begin{proof}
With $\delta_0$ as in Lemma~\ref{lemm:homotopy}, we see that $\{u \in \cM_p\ |\ \osc_{\overline{\bB}}u < \delta_0\}$ deformation retracts onto the set of constant mappings in $\cM_p$. Also, Poincar\'e's inequality gives $\osc_{\overline{\bB}}u \leq C_{p, \ep}(D_{\ep, p}(u))^{1/p}$. Thus, choosing $\alpha_0 > 0$ such that $C_{p, \ep}\alpha_0^{1/p} < \delta_0$, we get
\[
\sup_{t \in [0, 1]} D_{\ep, p}(\gamma(t)) \geq \alpha_0 \text{, for all }\gamma \in \cP.
\]
Now take $\eta_1 < \alpha < \alpha_0$, with $\eta_1$ and $\alpha$ to be determined. Then by continuity, there exists $t_1 \neq t_0$ such that $D_{\ep, p}(\gamma(t_1)) = \alpha$. Without loss of generality, we may assume $t_0 < t_1$ and that 
\begin{equation}\label{eq:contract-tube}
D_{\ep, p}(\gamma(t)) \leq \alpha \text{ for all }t \in [t_0, t_1].
\end{equation}
By our choice of $\alpha_0$, for $i = 0, 1$ there exist $y_i \in \Sigma$ with $\|y_i - \gamma(t_i)\|_{\infty} \leq C_{p, \ep} \alpha^{\frac{1}{p}} < \delta_0$ and a path $\gamma_i$ in $\cM_p$ leading from $y_i$ to $\gamma(t_i)$. Moreover, by~\eqref{eq:slice-energy} and~\eqref{eq:segment-vol} we have, decreasing $\alpha$ if necessary, that
\begin{equation}\label{eq:small-slice}
D_{\ep, p}(\gamma_i(t)) \leq C_{\Sigma, p}\alpha < \alpha_0 \text{ for all }t \in [0, 1],
\end{equation}
\begin{equation}\label{eq:cap-vol}
V_f(\gamma_i) \leq C_{\Sigma, p, \ep, H_0}\alpha^{\frac{1}{p}} \|\nabla \gamma(t_i)\|_{2}^2 \text{ for }i = 1, 2.
\end{equation}
By~\eqref{eq:contract-tube},~\eqref{eq:small-slice} and our choice of $\alpha_0$, the concatenation $\gamma_0 + \gamma|_{[t_0, t_1]} + (-\gamma_1)$ is homotopic to a path through constant mappings in $\cM_p$, and hence
\[
V_f(\gamma_0 + \gamma|_{[t_0, t_1]} + (-\gamma_1)) = 0.
\]
Combining this with~\eqref{eq:cap-vol} gives
\[
\left| V_f(\gamma|_{[0, t_1]}) - V_f(\gamma|_{[0, t_0]}) \right| \leq C_{\Sigma, p, \ep, H_0}\alpha^{\frac{1}{p}}\left( D_{\ep, p}(\gamma(t_0)) + D_{\ep, p}(\gamma(t_1)) \right).
\]
We may now finish the proof as in~\cite[Lemma 3.4]{ChengZhou-cmc}.
\end{proof}

Below we let $\delta  = \frac{1}{4}\min\{\eta_2, \int_{\Omega}f\}$, with $\eta_2$ given by Lemma~\ref{lemm:non-trivial-lower-bound}, and define
\[
\cN = \cN_{\ep, p, f} = \{u \in \cM_p\ |\ \text{there exists } \gamma \in \cE(u) \text{ with }|E_{\ep, p, f}(u, \gamma) - \omega_{\ep, p, f}| < \delta \}.
\]
Given $u \in \cN$, we choose some $\gamma \in \cE(u)$ as in the above definition and set $E_{\ep, p, f}^\cN(u) = E_{\ep, p, f}(u, \gamma)$. Note that $E_{\ep, p, f}^{\cN}$ is well-defined thanks to Lemma~\ref{lemm:vol-prop} and our choice of $\delta$.

\begin{lemm}\label{lemm:local-simple}
$\cN$ is an open subset of $\cM_p$, and $E_{\ep, p, f}^\cN$ is a $C^2$-functional on $\cN$.
\end{lemm}
\begin{proof}
For $u_0 \in \cN$, choose $\gamma_0 \in \cE(u_0)$ such that $|E_{\ep, p, f}(u_0, \gamma_0) - \omega_{\ep, p, f}| < \delta$ and consider the local reduction $E_{\ep, p, f}^\cA$ induced by $(u_0, \gamma_0)$ on a simply-connected neighborhood $\cA$ of $u_0$. In particular $|E_{\ep, p, f}^{\cA}(u_0) - \omega_{\ep, p, f}| < \delta$. Since $E_{\ep, p, f}^{\cA}$ is continuous, there exists a neighborhood $\cU$ of $u_0$ in $\cA$ such that $|E_{\ep, p, f}^\cA(u) - \omega_{\ep, p, f}| < \delta$ for all $u \in \cU$. Consequently $\cU \subseteq \cN$, and $E_{\ep, p, f}^{\cN} = E_{\ep, p, f}^{\cA}$ on $\cU$. This yields both conclusions of the lemma.
\end{proof}

The following proposition, analogous to the classical mountain pass theorem, is our basic existence result for critical points of the perturbed functional. 
\begin{prop}\label{prop:good-PS}
Suppose for some $C_0 > 0$ and sequence $\alpha_k \to 0$, we have a path $\gamma_k \in \cP_{\ep, p, f, \alpha_k, C_0}$ for each $k$. Then, up to taking a subsequence, there exist $t_k \in [0, 1]$ such that
\vskip 1mm
\begin{enumerate}
\item[(a)] $|E_{\ep, p, f}(\gamma_k(t_k), \gamma_k|_{[0, t_k]}) - \omega_{\ep, p, f}| \leq \alpha_k$ for all $k$.
\vskip 1mm
\item[(b)] $\gamma_k(t_k)$ converges strongly in $W^{1, p}$ to some $u \in \cK_{\ep, p, f, C_0}$.
\vskip 1mm
\item[(c)] $D_{\ep, p}(u) \geq \eta_1$, where $\eta_1$ is as in Lemma~\ref{lemm:non-trivial-lower-bound}.
 \end{enumerate}
\end{prop}
\begin{proof}
The proof is essentially the same as~\cite[Proposition 3.6]{ChengZhou2021} and hence is omitted. We only note that the estimate~\eqref{eq:segment-vol} should be used in place of~\cite[Lemma 2.4(c)]{ChengZhou2021}, and that Lemma~\ref{lemm:non-trivial-lower-bound} is used to deduce (c) from (a)(b).
\end{proof}

We next want to refine Proposition~\ref{prop:good-PS} to include an index upper bound on the resulting critical point. This is done with a deformation procedure similar to the ones used by Marques-Neves~\cite{Marques-Neves16} and Song~\cite{Song19} in the setting of the Almgren-Pitts min-max theory.
\begin{prop}\label{prop:bypassing}
Suppose $\cK'$ is a compact subset of $\cK_{\ep, p, f, C_0+1}$ such that $\Ind_{\ep, p, f}(v) \geq 2$ for all $v \in \cK'$. Let $C_0, (\alpha_k)$ and $(\gamma_k)$ be as in Proposition~\ref{prop:good-PS}. Then there exists another sequence of sweepouts $(\widetilde{\gamma}_k)$ with the following properties.
\begin{enumerate}
\item[(a)] $\widetilde{\gamma}_k \in \cP_{\ep, p, f, \alpha_k, C_0 + 1}$ for all $k$.
\vskip 1mm
\item[(b)] For all $u \in \cK'$, there exist $k_0(u) \in \NN$ and $d_0(u) > 0$ such that 
\[
\inf\{  \|\widetilde{\gamma}_k(t) - u \|_{1, p}  \ |\ E_{\ep, p, f}(\widetilde{\gamma}_k(t), \widetilde{\gamma}_k|_{[0, t]}) \geq \omega_{\ep, p, f} - \alpha_k \} \geq d_0(u) \text{ for all }k \geq k_0(u).
\]
\end{enumerate}
\end{prop}
\begin{proof}
As with~\cite[Lemma 3.9]{ChengZhou-cmc}, the proof consists of five steps modeled after~\cite[Theorem 7]{Song19}. The main source of technicality is again the fact that $\cK'$ is not discrete in general. Nonetheless, this proof is slightly shorter than that of ~\cite[Lemma 3.9]{ChengZhou-cmc} thanks to the neighborhood $\cN$ introduced above. The deformation procedure can be summed up as ``partition, perturb and replace'', the details of which occupy Steps 3 and 4. Steps 1 and 2 construct coverings that help us localize the problem, one for replacement and the other for perturbation. In Step 5 we verify all the required properties.
\vskip 2mm
\noindent\textbf{Step 1: Covering for replacement}
\vskip 2mm
For each $v \in \cK'$, since $\cK' \subseteq \cN$, we can find a simply-connected neighborhood $\cA_v$ of $v$ in $\cN$. In particular there is a local reduction $E^{\cA_v}_{\ep, p, f}$ that agrees with $E^\cN_{\ep, p, f}$ on $\cA_v$. Proposition~\ref{prop:coordinates} then yields $r_v, \theta_v, a_v > 0$ such that $\Theta_v: \psi \mapsto \exp_u^h(\psi)$ restricts to a chart on 
\[
\cB_v := \{\psi \in T_v\cM_p \ |\ \|\psi\|_{1, p}  < 3r_v\},
\]
with image contained in $\cA_v$, and that conclusions (a) and (b) of Proposition~\ref{prop:coordinates} hold with $\widetilde{E}_v = E_{\ep, p, f}^\cN \circ \Theta_v$. Decreasing $r_v$ if necessary, we may also arrange that 
\begin{equation}\label{eq:energy-osc}
|D_{\ep, p}(w) - D_{\ep, p}(w')| < 1, \text{ for all }w, w' \in \Theta_v(\cB_v).
\end{equation}
Below we define, for $s \in [1, 4]$, 
\begin{align*}
s\cC_v =\ & \Theta_v \Big(\{\xi \in T_v\cM_p \ |\ \| \xi^- \|_{1, p}\leq \frac{r_v}{4}s,\ \| \xi^\perp \|_{1, p} \leq \frac{\theta_v r_v}{4} s \}\Big),\\
\partial_- (s\cC_v) =\ & \Theta_v \Big(\{\xi \in T_v\cM_p \ |\ \| \xi^- \|_{1, p} = \frac{r_v}{4}s,\ \| \xi^\perp \|_{1, p} \leq  \frac{\theta_v  r_v}{4}s \}\Big).
\end{align*}
Note that by Proposition~\ref{prop:coordinates}(a), with $b_v = a_v \frac{r_v^2}{4}$ we have
\begin{equation}\label{eq:drop-on-shell}
E_{\ep, p, f}^\cN(w) \leq E_{\ep, p, f}^\cN(v) - b_v \text{ for all } w\in\partial_{-}(2\cC_v).
\end{equation}
Next, essentially due to the last remark in Definition~\ref{defi:index}, there exists $\rho(v) > 0$ such that 
\[
B_{2\rho(v)}^{1, p}(v) := \{w \in \cM_p\ |\ \|w - v\|_{1, p} < 2\rho(v)\} \subseteq \cC_v.
\]
As $\cK'$ is compact, we can cover it by a finite collection $\{B^{1, p}_{\rho(v_i)}(v_i)\}_{i = 1}^M$ and let 
\begin{align*}
&\underline{a} = \min_{1 \leq i \leq M}a_{v_i},\  \underline{b} = \min_{1 \leq i \leq M}b_{v_i},\ \underline{\rho} = \min_{1 \leq i \leq M}\rho(v_i).
\end{align*}

\vskip 2mm
\noindent\textbf{Step 2: Covering for perturbation}
\vskip 2mm
Let $w \in \cK'$. Thanks to Proposition~\ref{prop:coordinates}(a) and the continuity of both
\[
(u, \xi^-) \mapsto \Theta_w \left( \Theta_w^{-1}(u) +\xi^- \right) - u, \text{ and }
\]
\[
(u, \xi^-) \mapsto E^\cN_{\ep, p, f}(\Theta_w \left( \Theta_w^{-1}(u) +\xi^- \right)) - E^\cN_{\ep, p, f}(u),
\]
there exist $s(w) < \rho(w)$, $\xi^-(w) \in V_w^-$ and $c_w > 0$ such that $B^{1, p}_{s(w)}(w) \subseteq \cup_{i = 1}^M B^{1, p}_{\rho(v_i)}(v_i)$, and that for all $u \in B^{1, p}_{s(w)}(w)$, letting $u_t = \Theta_w \left( \Theta_w^{-1}(u) + t\xi^-(w) \right)$, we have 
\begin{align}
\|u_t - u\|_{1, p} &< \underline{\rho}, \text{ for all }t \in [0, 1], \label{eq:proximity}\\
E_{\ep, p, f}^\cN(u_1) &\leq E_{\ep, p, f}^\cN(u) - c_w. \label{eq:drop}
\end{align}
Next we cover $\cK'$ by a finite collection $\{B^{1, p}_{s(w_i)}(w_i)\}_{i = 1}^L$, and define 
\[
\underline{c} = \min_{1 \leq i \leq L}c_{w_i}.
\]
Again by the compactness of $\cK'$, we can fix $\theta_1 > 0$ small enough so that the neighborhood
\[
B_{\theta_1}(\cK'):= \{u \in \cM_p\ |\ \|u - v\|_{1, p} < \theta_1 \text{ for some }v \in \cK' \}
\]
is contained in $\cup_{j =1}^L B^{1, p}_{s_j}(w_j) \subseteq \cup_{i = 1}^M B^{1, p}_{\rho_i}(v_i) \subseteq \cN$, where we have written $s_j = s(w_j)$ and $\rho_i = \rho(v_i)$. In particular, for all $u \in B_{\theta_1}(\cK')$ there exists a continuous path $t \mapsto u_t$ in $\cN$ such that
\begin{enumerate}
\item[(p1)] $u_0 = u$.
\vskip 1mm
\item[(p2)] $u_t \in B^{1, p}_{2\rho_{v_i}}(v_i)$ for all $t \in [0, 1]$ whenever $u \in B^{1, p}_{\rho_{i}}(v_i)$. 
\vskip 1mm
\item[(p3)] $E^{\cN}_{\ep, p, f}(u_1) \leq E^{\cN}_{\ep, p, f}(u) - \underline{c}$.
\end{enumerate}
\vskip 2mm
\noindent\textbf{Step 3: Partitioning $[0, 1]$ and perturbing the endpoints}
\vskip 2mm
Recall that $\delta = \frac{1}{4}\min\{\int_{\Omega}f, \eta_2\}$, and define
\begin{align*}
J_k &= \{t \in [0, 1]\ |\ \gamma_k(t) \in B_{\theta_1}(\cK'), E_{\ep, p, f}(\gamma_k(t), \gamma_k|_{[0, t]}) > \omega_{\ep, p, f} - 2\delta \},\\
I_k &= \{ t\in [0, 1]\ |\ \gamma_k(t) \in \overline{B_{\frac{3\theta_1}{4}}(\cK')}, E_{\ep, p, f}(\gamma_k(t), \gamma_k|_{[0, t]}) > \omega_{\ep, p, f} - 2\delta \}.
\end{align*}
Since $(\alpha_k)$ converges to $0$, by Lemma~\ref{lemm:non-trivial-lower-bound}, Lemma~\ref{lemm:vol-prop} and our choice of $\delta$, we see that for $k$ sufficiently large we have $0, 1 \not\in J_k$ and that 
\begin{equation}\label{eq:agreement}
E_{\ep,p, f}(\gamma_k(t), \gamma_k|_{[0, t]}) = E^\cN_{\ep, p, f}(\gamma_k(t)), \text{ for all }t \in J_k.
\end{equation}
In particular, $E_{\ep,p, f}(\gamma_k(t), \gamma_k|_{[0, t]}) > \omega_{\ep, p, f} - \delta$ for $t \in J_k$, which shows that $I_k$ is closed in $[0, 1]$, and hence compact.

To continue, we drop the subscripts in $\gamma_k, \alpha_k, I_k$ and $J_k$. Since $J$ is open and $I \subseteq J$ is compact, there exist finitely many intervals $[a_j, b_j]$, $j = 1, \cdots, m$, such that 
\[
I \subseteq \cup_{j = 1}^m [a_j, b_j] \subseteq J, \text{ and that }\gamma(a_j), \gamma(b_j) \not\in B_{\frac{\theta_1}{2}}(\cK').
\]
For each $j = 1, \cdots, m$, since $\gamma([a_j, b_j]) \subseteq B_{\theta_1}(\cK') \subseteq \cup_{i = 1}^M B_{\rho_i}^{1, p}(v_i)$, we can further partition each $[a_j, b_j]$ into 
\[
a_j = t_0 < \cdots < t_n = b_j,
\]
such that for all $l \in \{0, \cdots, n-1\}$, there exists $i = i(l) \in \{1, \cdots, M\}$ with
\[
\gamma([t_l, t_{l+1}]) \subseteq B_{\rho_i}^{1, p}(v_i).
\]
For $l \in \{1, \cdots, n-1\}$, since $\gamma(t_l) \in B_{\theta_1}(\cK') \cap B_{\rho_{i(l-1)}}^{1, p}(v_{i(l-1)})  \cap B_{\rho_{i(l)}}^{1, p}(v_{i(l)})$, by the remarks at the end of the previous step, there exists a continuous path $P_l: [0,1] \to \cN$ such that 
\begin{enumerate}
\item[(a1)] $P_l(0) = \gamma(t_l)$.
\vskip 1mm
\item[(a2)] $P_l(t) \in B_{2\rho_{i(l-1)}}^{1, p}(v_{i(l-1)})  \cap B_{2\rho_{i(l)}}^{1, p}(v_{i(l)})$ for all $t \in [0, 1]$.
\vskip 1mm
\item[(a3)] $E^\cN_{\ep, p, f}(P_l(1)) \leq E^\cN_{\ep, p, f}(\gamma(t_l)) - \underline{c}$.
\end{enumerate} 
Now let $h_0 = \gamma|_{[t_0, t_1]} + P_1,\ h_{n-1} = (-P_{n-1}) + \gamma|_{[t_{n-1}, t_n]}$, and let
\[
h_{l} = (-P_{l}) + (\gamma([t_{l}, t_{l + 1}])) + P_{l+1} \text{ for }l = 1, \cdots, n-2.
\]
Then $h_0 + \cdots + h_{n-1}$ is homotopic to $\gamma|_{[a_j, b_j]}$, and the endpoints $h_0(0) = \gamma(a_j)$ and $h_{n-1}(1) = \gamma(b_j)$ lie outside of $B_{\frac{\theta_1}{2}}(\cK')$.
\vskip 2mm
\noindent\textbf{Step 4: Replacing the path on the subintervals}
\vskip 2mm
Fix $l \in \{0, \cdots, n-1\}$ and note that since $h_l$ maps into $B^{1,p}_{2\rho_{i(l)}}(v_{i(l)}) \subseteq \cC_{v_{i(l)}}$, we may define $
\widehat{h}_l = (\Theta_{v_{i(l)}})^{-1}\circ h_l$, which satisfies $\|\widehat{h}_l(t)^-\|_{1, p} \leq \frac{r_{v_{i(l)}}}{4}$ and $\|\widehat{h}_l(t)\|_{1, p} \leq \frac{r_{v_{i(l)}}}{2}$ for all $t \in [0, 1]$. To reduce notation, below we drop the subscript $l$ and write $v_i = v_{i(l)}$ and $h = h_l$.

For $j = 0, 1$, we choose $\xi^-_j \in V_{v_i}^-$ such that $\|\xi^-_j\|_{1, p} =1$ and $(\delta\widetilde{E}_{v_i})_{\widehat{h}(j)}(\xi^-_j) \leq 0$. Then there exists $t_j \geq 0$ such that 
\[
\| \widehat{h}(j)^- + t_j \xi^-_j \|_{1, p} = \frac{r_{v_i}}{2}, \text{ and that}
\]
\[
\| \widehat{h}(j)^- + t \xi^-_j \|_{1, p} \leq \frac{r_{v_i}}{2} \text{ for all }0 \leq t \leq t_j.
\] 
Since $\|\xi^-_j\|_{1, p} = 1$ and $\|\widehat{h}(j)^-\|_{1, p} \leq \frac{r_{v_i}}{4}$, we have by the triangle inequality that $t_j \leq r_{v_i}$. Next we define $q_j:[0, 2] \to 2\cC_{v_i} \subseteq \cN$ in terms of $\widehat{q}_j = \Theta_{v_i}^{-1}\circ q_j$ by letting
\begin{equation}
\widehat{q}_j(t) = \left\{
\begin{array}{ll}
\widehat{h}(j) + (t\cdot t_j)\xi^-_j,& \text{ for }t\in [0, 1],\\
\ &\\
\widehat{h}(j) + t_j\xi^-_j - (t-1)\widehat{h}(j)^\perp, & \text{ for }t \in [1, 2].
\end{array}
\right.
\end{equation}
Then we have for $j = 0, 1$ the following:
\begin{enumerate}
\item[(b1)] $q_j([1, 2]) \subseteq \partial_-(2\cC_{v_i})$, and $q_j(2) \in \Theta_{v_i}\left( \{\|\varphi^-\|_{1, p} = \frac{r_{v_i}}{2}, \varphi^{\perp} = 0\} \right)$.
\vskip 1mm
\item[(b2)] For $t \in [0, 1]$, there holds by Proposition~\ref{prop:coordinates}(b) that 
\[
E^\cN_{\ep, p,f}(q_j(0)) - E^\cN_{\ep, p,f}(q_j(t)) \geq \underline{a}(t \cdot t_j)^2 = \underline{a}\| \widehat{q}_j(t) - \widehat{q}_j(0) \|_{1, p}^2.
\]
\end{enumerate}
Since $\Ind_{\ep, p, f}(v_i) \geq 2$, the set $\Theta_{v_i}\left( \{\|\varphi^-\|_{1, p} = \frac{r_{v_i}}{2}, \varphi^{\perp} = 0\} \right)$ is connected. Hence there exists a continuous path
\[
q_2:[0, 1] \to \Theta_{v_i}\left( \{\|\varphi^-\|_{1, p} = \frac{r_{v_i}}{2}, \varphi^{\perp} = 0\} \right) \subseteq 2\cC_{v_i},
\] 
with $q_2(j) = q_j(2)$ for $j = 0, 1$. We then replace $h$ with the path
\[
\widetilde{h} = q_0 + q_2 -q_1.
\]
Note that since $h$, $\widetilde{h}$ both map into the simply-connected set $2\cC_{v_i}$ and share the same endpoints, they are homotopic. We now define $\widetilde{\gamma}$ to be the path obtained by first replacing each $\gamma|_{[a_j, b_j]}$ with $h_0 + \cdots + h_{n-1}$, and then replacing each $h_l$ by $\widetilde{h}_l$. Note that $\widetilde{\gamma}([a_j, b_j]) \subseteq \cN$ for all $j \in \{1, \cdots, m\}$. Also, reparametrizing if necessary, we may assume that 
\[
\widetilde{\gamma}|_{[t_l, t_{l + 1}]} = \widetilde{h}_l,
\]
with $q_0|_{[0, 1]}, q_0|_{[1, 2]}, q_2, -(q_1|_{[1, 2]})$ and $-(q_1|_{[0, 1]})$ successively parametrized by a fifth of $[t_l, t_{l + 1}]$. For later use, we note the following properties in addition to (b1) and (b2) above. 
\vskip 1mm
\begin{enumerate}
\item[(b3)] By~\eqref{eq:drop-on-shell}, we have that 
\[
E^\cN_{\ep, p, f}(q_j(t)) \leq \omega_{\ep, p, f} - \underline{b}\text{ for }j = 0, 1 \text{ and }t \in [1, 2], \text{ and }
\]
\[
E^\cN_{\ep, p, f}(q_2(t)) \leq \omega_{\ep, p, f} - \underline{b} \text{ for }t \in [0, 1].
\]
\vskip 1mm
\item[(b4)] Either $E^\cN_{\ep, p, f}(q_0(0)) = E^\cN_{\ep, p, f}(\gamma(a_j))$, or $E^\cN_{\ep, p, f}(q_0(0)) \leq E^\cN_{\ep, p, f}(\gamma(t_l)) - \underline{c}$ for some $l \in \{1, \cdots, n-1\}$. A similar statement holds for $q_1(0)$.
\end{enumerate}
\vskip 2mm
\noindent\textbf{Step 5: Verifying properties}
\vskip 2mm
Let $I' = \cup_{j = 1}^m [a_j, b_j]$. Since $\gamma(t) = \widetilde{\gamma}(t)$ for all $t \not\in I'$ and $\gamma|_{[a_j, b_j]}$ is homotopic to $\widetilde{\gamma}|_{[a_j, b_j]}$ for all $j$, we see by Corollary~\ref{coro:homotopy} that 
\begin{equation}\label{eq:agree-out}
E_{\ep, p, f}(\gamma(t), \gamma|_{[0, t]}) = E_{\ep, p, f}(\widetilde{\gamma}(t), \widetilde{\gamma}|_{[0, t]}), \text{ for all }t \not\in I'.
\end{equation}
In particular, 
\[
E_{\ep, p, f}(\widetilde{\gamma}(t), \widetilde{\gamma}|_{[0, t]}) \leq \omega_{\ep, p, f} + \alpha, \text{ for all }t \not\in I'.
\]
On the other hand, for each $j$, note that both $E^{\cN}_{\ep, p, f}(\widetilde{\gamma}(\cdot))$ and $E_{\ep, p, f}(\widetilde{\gamma}(\cdot), \widetilde{\gamma}|_{[0, \cdot]})$ are continuous on $[a_j, b_j]$. Since they differ by a constant integer multiple of $\int_{\Omega}f$ and agree at $t = a_j$ by~\eqref{eq:agreement} and~\eqref{eq:agree-out}, they agree everywhere on $[a_j, b_j]$. That is,
\begin{equation}\label{eq:agreement-tilde}
E_{\ep, p, f}(\widetilde{\gamma}(t), \widetilde{\gamma}|_{[0, t]}) = E^{\cN}_{\ep, p, f}(\widetilde{\gamma}(t)) \text{ for all } t\in [a_j, b_j].
\end{equation}
Then, by (b2), (b3) and (b4) above, we see that, on $[a_j, b_j]$,
\[
E_{\ep, p, f}(\widetilde{\gamma}(t), \widetilde{\gamma}|_{[0, t]}) \leq \omega_{\ep, p, f} + \alpha.
\]
Therefore we have verified the first part of the definition of $\cP_{\ep, p, f, \alpha, C_0+1}$ for $\widetilde{\gamma}$. Next, if $t \in [0, 1]$ is such that 
\[
E_{\ep, p, f}(\widetilde{\gamma}(t), \widetilde{\gamma}|_{[0, t]}) \geq \omega_{\ep, p, f} - \alpha,
\]
then since $(\alpha_k)$ converges to $0$, again by (b2), (b3) and (b4), for large enough $k$ there are only three possibilities: (i) $t \not\in I'$, or (ii) $t$ belongs to the first fifth of $[t_0, t_1]$ (for some $j$), or (iii) $t$ lies in the last fifth of $[t_{n-1}, t_n]$ (for some $j$). 

In case (i) we have $\gamma(t) = \widetilde{\gamma}(t)$ and $E_{\ep, p, f}(\gamma(t), \gamma|_{[0, t]}) = E_{\ep, p, f}(\widetilde{\gamma}(t), \widetilde{\gamma}|_{[0, t]}) \geq \omega_{\ep, p, f} - \alpha$. Thus, since $\gamma \in \cP_{\ep, p, f, \alpha, C_0}$, we deduce that
\[
D_{\ep, p}(\widetilde{\gamma}(t)) = D_{\ep, p}(\gamma(t)) \leq C_0.
\]
In the case (ii), we have by~\eqref{eq:agreement},~\eqref{eq:agreement-tilde} and property (b2) that
\[
E_{\ep, p, f}(\widetilde{\gamma}(t), \widetilde{\gamma}|_{[0, t]}) = E^\cN_{\ep, p, f}(\widetilde{\gamma}(t)) \leq E^{\cN}_{\ep, p, f}(\gamma(a_n)) = E_{\ep, p, f}(\gamma(a_j), \gamma|_{[0, a_j]}).
\]
Thus $E_{\ep, p, f}(\gamma(a_j), \gamma|_{[0, a_j]}) \geq \omega_{\ep, p, f} - \alpha$ and consequently $D_{\ep, p}(\gamma(a_j)) \leq C_0$. Now by the definition of $q_0$, we see that $\widetilde{\gamma}(t)$ and $\gamma(a_j)$ both belong to $2\cC_{v_{i(0)}}$. Hence by~\eqref{eq:energy-osc} we get 
\[
D_{\ep, p}(\widetilde{\gamma}(t)) \leq D_{\ep, p}(\gamma(a_j)) + 1 \leq C_0 + 1.
\]
Finally case (iii) is similar. This proves that $\widetilde{\gamma} \in \cP_{\ep, p, f, \alpha, C_0 + 1}$.

To prove conclusion (b), assume towards a contradiction that there exist $u \in \cK'$, a subsequence of $\widetilde{\gamma}_k$, which we do not relabel, and a sequence $t_k \in [0, 1]$, such that 
\begin{equation}\label{eq:large-slices}
E_{\ep, p, f}(\widetilde{\gamma}_k(t_k), \widetilde{\gamma}_k|_{[0, t_k]}) \geq \omega_{\ep, p, f} - \alpha_k,
\end{equation}
\begin{equation}\label{eq:convergence}
\| \widetilde{\gamma}_k(t_k) - u \|_{1, p} \to 0 \text{ as }k \to \infty.
\end{equation}
To continue, we need to put back the indices dropped previously. Thus $I_k \subseteq \cup_{j = 1}^{m_k}[a_{k, j}, b_{k, j}] \subseteq J_k$, and for $j = 1, \cdots, m_k$, the partition of $[a_{k, j}, b_{k, j}]$ reads
\[
a_{k, j} = t_{k, j, 0} < \cdots < t_{k, j, n_{k, j}} = b_{k, j}.
\]
Also, for $l = 0, \cdots, n_{k, j }-1$,  the paths $q_0, q_1$ are $q_{k, j, l, 0}, q_{k, j, l, 1}$, respectively. Conditions~\eqref{eq:large-slices} and~\eqref{eq:convergence} imply that we eventually have $t_k \in \cup_{j = 1}^{m_k}[a_{k, j}, b_{k, j}]$. Furthermore, by~\eqref{eq:agreement-tilde},~\eqref{eq:large-slices} together with (b2), (b3), (b4) above, we have for sufficiently large $k$ that
\[
\widetilde{\gamma}_k(t_k) \in q_{k, j(k), l(k), a(k)}([0, 1]),
\] 
for some $j(k) \in \{1, \cdots, m_k\}$, $(l(k), a(k)) = (0, 0)$ or $(n_{k, j} - 1, 1)$. Also, by construction, $q_{k, j(k), l(k), a(k)}([0, 1]) \subseteq 2\cC_{v_{i_k}}$ for some $i_k \in \{1, \cdots, M\}$. Without loss of generality we may assume, after passing to a subsequence, that $(l(k), a(k))= (0, 0)$ and that $i_k = 1$ for all $k$. In particular $u \in 2\cC_{v_1}$ by~\eqref{eq:convergence}.

Below we let $Q_k = q_{k, j(k), 0, 0}$ and write $Q_k(t_k)$ for $\widetilde{\gamma}_k(t_k)$ by abuse of notation. Recalling that $Q_k(0) = \gamma_k(a_{k, j(k)})$ and that $\gamma_k \in \cP_{\ep, p, f, \alpha_k, C_0}$, we deduce using~\eqref{eq:agreement} that eventually 
\[
E^{\cN}_{\ep, p, f}(Q_k(0)) \leq \omega_{\ep, p, f} + \alpha_k.
\] 
On the other hand, by~\eqref{eq:agreement-tilde} and~\eqref{eq:large-slices} we have $E^\cN_{\ep, p, f}(Q_k(t_k)) \geq \omega_{\ep, p, f} - \alpha_k$. Therefore 
\[
E^{\cN}_{\ep, p, f}(Q_k(0)) - E^\cN_{\ep, p, f}(Q_k(t_k)) \leq 2\alpha_k \to 0\text{ as }k \to \infty.
\]
By property (b2) from Step 4, this implies 
\[
\lim_{k \to \infty}\| \Theta_{v_1}^{-1}(Q_k(0)) - \Theta_{v_1}^{-1}(Q_k(t_k))  \|_{1, p} =0.
\]
Combining this with~\eqref{eq:convergence} gives
\[
\lim_{k \to \infty} \|\Theta_{v_1}^{-1}(Q_k(0)) - \Theta_{v_1}^{-1}(u)\|_{1, p} = 0.
\]
Since $Q_k(0) = \gamma_{k}(a_{k, j(k)}) \not\in B_{\frac{\theta_1}{2}}(\cK')$, this is a contradiction. The proof of Proposition~\ref{prop:bypassing} is complete.
\end{proof}

Lemma~\ref{lemm:monotonicity-trick}, Proposition~\ref{prop:good-PS} and Proposition~\ref{prop:bypassing} together give the main result of this section. Below, in addition to the choice of $p$ already made, we assume that $\Sigma', H_0$ are as in Section~\ref{subsec:max-principle}. Also, given $H \in (0, H_0)$ we choose $t_0$ and $f$ as in that section. 
\begin{prop}\label{prop:existence-with-index}
For almost every $r \in (0, 1]$, there exist $C_0 > 0$ and a sequence $\ep_j \to 0$ such that for each $j$ there is a non-constant critical point $u_j$ of $E_{\ep_j, p, rf}$ with the following properties.
\vskip 1mm
\begin{enumerate}
\item[(a)] $D_{\ep_j, p}(u_j) \leq C_0+1$ and $D(u_j) \geq \beta$, where $\beta$ is from Proposition~\ref{prop:uniform-lower-bound}.
\vskip 1mm
\item[(b)] $\Ind_{\ep_j, p, rf}(u_j) \leq 1$.
\vskip 1mm
\item[(c)] $u_j$ is smooth on $\overline{\bB}$ and $u_j(\overline{\bB})\subseteq \overline{\Omega'_{t_0}}$.
\end{enumerate}
\end{prop}
\begin{proof}
By Lemma~\ref{lemm:monotonicity-trick}, for almost every $r \in (0, 1]$ there exist $C_0 > 0$ and a sequence $\ep_j \to 0$ such that for each $j$, we can find $\gamma_k \in \cP_{\ep_j, p, rf, \frac{r}{k}, C_0}$ for $k$ sufficiently large. Fixing $j$, the following compact subset of $\cK_{\ep_j, p, rf, C_0 + 1}$ is then non-empty by Proposition~\ref{prop:good-PS} and Proposition~\ref{prop:uniform-lower-bound}:
\[
\cK' = \{v \in \cK_{\ep_j, p, rf, C_0 + 1}\ |\  D(v) \geq \beta \}.
\]
 Now assume by contradiction that $\Ind_{\ep_j, p, rf}(v) \geq 2 \text{ for all }v \in \cK'$. Then, letting $(\widetilde{\gamma}_k)$ be the sweepouts obtained by applying Proposition~\ref{prop:bypassing} to $(\gamma_k)$, we may use Proposition~\ref{prop:good-PS} on $(\widetilde{\gamma_k})$ to get a sequence $t_k \in [0, 1]$ such that 
\begin{equation}\label{eq:min-max-slices}
|E_{\ep_j, p, rf}(\widetilde{\gamma}_k(t_k), \widetilde{\gamma}_k|_{[0, t_k]}) - \omega_{\ep_j, p, rf}| \leq \frac{r}{k},
\end{equation}
and that, up to taking a subsequence, $\widetilde{\gamma}_k(t_k)$ converges strongly in $W^{1, p}$ to some non-constant $u \in \cK_{\ep_j, p, rf, C_0 + 1}$, which by Proposition~\ref{prop:uniform-lower-bound} must lie in $\cK'$. However, by~\eqref{eq:min-max-slices} and Proposition~\ref{prop:bypassing}(b), we get $d_0 > 0$ and $k_0 \in \NN$ such that
\[
\| \widetilde{\gamma}_k(t_k) - u \|_{1, p} \geq d_0  \text{ for all }k \geq k_0,
\]
which is a contradiction. Thus there exists some $u_j \in \cK'$ with $\Ind_{\ep_j, p, rf}(u_j) \leq 1$. This gives a non-constant critical point of $E_{\ep_j, p, rf}$ having properties (a) and (b). Smoothness of $u_j$ and the fact that it maps into $\overline{\Omega'_{t_0}}$ follow respectively from Propositions~\ref{prop:regularity} and~\ref{prop:max-principle}(a).
\end{proof}
\section{Passage to the limit as $\ep \to 0$}\label{sec:passage}
We continue to fix $p \in (2, p_2]$, $H \in (0, H_0)$ and let $t_0$ and $f$ be as in Section~\ref{subsec:max-principle}. In this section we analyze the critical points $u_j$ produced in Proposition~\ref{prop:existence-with-index} as $j \to \infty$. We first recall a removable singularity result, Lemma~\ref{lemm:remove}, that will be used repeatedly. Then in Proposition~\ref{prop:convergence-mod-bubbling} we prove that $(u_j)$ converges modulo energy concentration. Propositions~\ref{prop:rescale} and~\ref{prop:no-interior-bubble} guarantee that when energy concentrates somewhere, we still get a desired constant mean curvature disk after rescaling. Theorem~\ref{thm:main-1} is proved at the end.

\begin{lemm}[\cite{Gruter1984}, Theorem 2 and \cite{Fraser2000}, Theorem 1.10] 
\label{lemm:remove}
Suppose for some $x_0 \in \overline{\bB}$ and $\rho \in (0, r_{\bB})$ that $u: (\bB_{\rho}(x_0)\setminus \{x_0\}) \cap \overline{\bB} \to \RR^3$ is a smooth map satisfying
\begin{equation}\label{eq:finite-energy}
\int_{\bB_\rho(x_0) \cap \bB} |\nabla u|^2 <\infty,
\end{equation}
and that
\begin{equation}\label{eq:cmc-remove}
\left\{
\begin{array}{ll}
\Delta u= h(u) u_{x^1} \times u_{x^2} & \text{ on }(\bB_{\rho}(x_0)\setminus \{x_0\}) \cap \overline{\bB},\\
u(x)\in \Sigma \text{ and }\pa{u}{n}(x) \perp T_{u(x)}\Sigma & \text{ for }x \in (\bB_{\rho}(x_0)\setminus \{x_0\}) \cap \partial \bB,
\end{array}
\right.
\end{equation}
for some bounded smooth function $h:\RR^3 \to \RR$. Then $u$ extends smoothly across $x_0$. 
\end{lemm}
\begin{proof}
The case where $x_0 \in \bB$ is covered in~\cite[Theorem 2]{Gruter1984}. Below we briefly outline the argument for the case $x_0 \in \partial \bB$, which is essentially the same as~\cite[Theorem 1.10]{Fraser2000}. By the conformal invariance of the Dirichlet integral and the equation~\eqref{eq:cmc-remove}, it suffices to consider the case where $u$ is defined on $\overline{\bB_\rho^+} \setminus \{0\}$ and is to be extended across $0$. By Remark~\ref{rmk:ep-zero} and the finite energy assumption~\eqref{eq:finite-energy}, decreasing $\rho$ if necessary, we have 
\begin{equation}\label{eq:ring-estimate}
|x||\nabla u(x)|  \leq C \|\nabla u\|_{2;\bB_{2|x|}^+} \text{ for all } x \in \overline{\bB_{\frac{\rho}{3}}^+}\setminus \{0\}.
\end{equation}
Thus, decreasing $\rho$ further, we can argue as in~\cite[page 948]{Fraser2000} to see that $u( \overline{\bB_{\frac{\rho}{3}}^+}\setminus \{0\})$ is contained in a tubular neighborhood of $\Sigma$, which allows us to extend $u$ to $\bB_{\frac{\rho}{3}}\setminus \{0\}$ by reflection as indicated on the same page in~\cite{Fraser2000}. Denoting the extension by $\widetilde{u}$, from~\eqref{eq:cmc-remove} we obtain analogues of (1.16) and (1.17) in~\cite{Fraser2000}, and consequently our $\widetilde{u}$ still satisfies
\begin{equation}\label{eq:quadratic-nonlinear}
|\Delta \widetilde{u}| \leq C|\nabla \widetilde{u}|^2, \text{ almost everywhere on }\bB_{\frac{\rho}{3}}.
\end{equation}
Next, since $u_{x^1} \times u_{x^2} \perp u_{x^1}, u_{x^2}$, it follows from~\eqref{eq:cmc-remove} that the Hopf differential defined in~\cite[Lemma 1.12]{Fraser2000} is still holomorphic on $\bB^+_{\frac{\rho}{3}}$ and real on $\bT_{\frac{\rho}{3}}\setminus\{0\}$. We then obtain as in that lemma an equipartition among the radial and angular energy. We finish the proof following~\cite[pages 949-951]{Fraser2000}, using this equipartition of energy,~\eqref{eq:ring-estimate}, and~\eqref{eq:quadratic-nonlinear}. In this final step we may need to decrease $\rho$ again.
\end{proof}

\begin{prop}\label{prop:convergence-mod-bubbling}
Let $r > 0$ be in the full-measure subset of $[0, 1]$ given by Proposition~\ref{prop:existence-with-index}, and let $C_0, \ep_j$ and $u_j$ be as in the conclusion there. Then there exist a finite set $\cS \subseteq \overline{\bB}$ and a smooth map $u:(\overline{\bB}, \partial \bB) \to (\overline{\Omega'}, \Sigma)$ satisfying
\begin{equation}\label{eq:fb-cmc}
\left\{
\begin{array}{ll}
\Delta u= rH \cdot u_{x^1} \times u_{x^2} & \text{ on }\overline{\bB},\\
u_r(x) \perp T_{u(x)}\Sigma & \text{ for all }x \in \partial \bB,
\end{array}
\right.
\end{equation}
such that, up to taking a subsequence, $u_j$ converges in $C^1_{\loc}(\overline{\bB}\setminus \cS)$ to $u$.
\end{prop}
\begin{proof}
Since $D(u_j) \leq D_{\ep_j, p}(u_j) \leq C_0 + 1$ and $u_j(\overline{\bB}) \subseteq \overline{\Omega'_{t_0}}$ for all $j$, by Proposition~\ref{prop:small-regular-2}, Remark~\ref{rmk:small-regular-int} and a standard argument, we get a finite set $\cS \subseteq \overline{\bB}$, a positive constant $\eta = \eta(p, H_0, \Sigma)$, and a map $u\in C^{1}_{\loc}(\overline{\bB}\setminus \cS)$ with image contained in $\overline{\Omega'_{t_0}}$, such that passing to a subsequence of $(u_j)$ if necessary, we have the following. First of all, $u_j \to u$ in $C^1_{\loc}(\overline{\bB} \setminus \cS)$. 
Secondly, for all $x_0 \in \cS$ and $t > 0$,
\begin{equation}\label{eq:concentration}
\liminf_{j \to \infty}\int_{\bB_t(x_0) \cap \bB} |\nabla u_j|^2 \geq \eta/2.
\end{equation}
The $C^1$-convergence away from $\cS$ implies that $u(x) \in \Sigma$ and $u_r(x) \perp T_{u(x)}\Sigma$ for all $x \in \partial \bB \setminus \cS$. Moreover,
\[
\int_{\bB} \langle \nabla u, \nabla \psi \rangle + rf(u)\langle \psi, u_{x^1} \times u_{x^2} \rangle = 0, 
\]
whenever $\psi \in C^1_c(\overline{\bB}\setminus \cS; \RR^3)$ is such that $\psi(x) \in T_{u(x)}\Sigma$ for all $x \in \partial \bB \setminus \cS$. Standard elliptic theory then implies that $u$ is smooth on $\overline{\bB} \setminus \cS$. Next, since
\[
\int_{K} |\nabla u|^2 = \lim_{j \to \infty}\int_{K}|\nabla u_j|^2 \leq 2(C_0 + 1) \text{ for all compact }K \subseteq\overline{\bB} \setminus \cS,
\]
Lemma~\ref{lemm:remove} implies that $u$ extends to a smooth map $(\overline{\bB}, \partial \bB) \to (\overline{\Omega'_{t_0}}, \Sigma)$ satisfying~\eqref{eq:cmc-rf}. The fact that $u$ maps into $\overline{\Omega'}$ and solves~\eqref{eq:fb-cmc} now follows from Proposition~\ref{prop:max-principle}(b).
\end{proof}

When $\cS \neq \emptyset$, we will need to rescale each $u_j$ appropriately. Given $x_0 \in \cS$, the center and rate of rescaling are chosen as follows. Fix $d \in (0, r_{\bB}]$ such that $d \leq \frac{1}{2}\min_{x, x' \in \cS, x \neq x'} |x - x'|$ and define, for each $j$,
\[
Q_j(t) =  \max_{x \in \overline{\bB_{d}(x_0)} \cap \overline{\bB}} \int_{\bB_t(x) \cap \bB} |\nabla u_j|^2.
\]
By~\eqref{eq:concentration} and the $C^1$-convergence of $u_j$ away from $\cS$, we get sequences $t_k \in [0, 1/k]$ and $x_k \to x_0$, and a subsequence $u_{j_k}$ of $u_j$, such that writing $u_k = u_{j_k}$ and $Q_k = Q_{j_k}$ by abuse of notation, we have
\[
\int_{\bB_{t_k}(x_k) \cap \bB} |\nabla u_k|^2 = Q_{k}(t_k) = \eta/3 \text{, for all }k.
\]
The next proposition guarantees that the parameter $\ep_k (= \ep_{j_k})$ still converges to zero after we rescale by $t_k$.
\begin{prop}\label{prop:rescale}
In the above setting, we have $\liminf_{k \to \infty} \frac{\ep_k}{t_k} = 0$.
\end{prop}
\begin{proof}
Assume towards a contradiction that there exists $\alpha > 0$ such that $\frac{\ep_k}{t_k} \geq \alpha$ for all $k$ large enough, and define 
\[
v_k(y) = u_k(x_k + \ep_k y),  \text{ for all }y \in \bB'_k := \bB_{\frac{1}{\ep_k}}(-\frac{x_k}{\ep_k}).
\]
Then we have
\begin{enumerate}
\item[(i)] $v_k(\overline{\bB_k'}) \subseteq \overline{\Omega'_{t_0}}$ and
\[
\int_{\bB'_k} \frac{|\nabla v_k|^2}{2} + \frac{1}{p}(\ep_k^2 + |\nabla v_k|^2)^{\frac{p}{2}} = \int_{\bB}\frac{|\nabla u_k|^2}{2} + \frac{\ep_k^{p-2}}{p}(1 + |\nabla u_k|^2)^{\frac{p}{2}} \leq C_0+1.
\]
\vskip 1mm
\item[(ii)] 
\begin{align*}
\int_{\bB_{\frac{1}{\alpha}}(0) \cap \bB'_k}|\nabla v_k|^2 =  \int_{\bB_{\frac{\ep_k}{\alpha}}(x_k) \cap \bB} |\nabla u_k|^2 \geq \int_{\bB_{t_k}(x_k) \cap \bB} |\nabla u_k|^2 = \eta/3.
\end{align*}
\vskip 1mm
\item[(iii)] $\Div\left( [1 + (\ep_k^2 + |\nabla v_k|^2)^{\frac{p}{2} - 1}]\nabla v_k \right) = rf(v_k) (v_k)_{x^1} \times (v_k)_{x^2}$ on $\bB'_k$. Also, $v_k(x) \in \Sigma$ and 
\[
\pa{v_k}{n}(x) \perp T_{v_k(x)}\Sigma \text{ for all }x \in\partial \bB'_k.
\]
\end{enumerate}

To continue, we shall only consider the case where $\limsup_{k \to \infty}\ep_k^{-1}(1 - |x_k|) < \infty$, as the other case is simpler. We first conformally transform the domain to obtain a sequence of maps defined on the upper half-plane. Specifically, let $z_k = \frac{1}{\ep_k}(\frac{x_k}{|x_k|} - x_k)$, take isometries $I_k$ of $\RR^2$ such that $I_k(z_k) = 0$ and $I_k(\bB_k') = \bB_k'' := \bB_{\frac{1}{\ep_k}}((0, \frac{1}{\ep_k}))$, and find, as in~\cite[page 957]{Fraser2000}, conformal maps $F_k$ from $\overline{\RR^2_+}$ onto $\overline{\bB''_k} \setminus \{(0, \frac{2}{\ep_k})\}$ that converge in $C^2_{\loc}(\overline{\RR^2_+})$ to the identity map. Since $(z_k)$ is a bounded sequence, so is $(I_k(0))$. Also, the conformal factors $\lambda_k:=  |(F_k)_{z}|$ converge to $1$ in $C^1_{\loc}(\overline{\RR^2_+})$. The new sequence of maps are then given by
\[
w_k = v_k \circ I_k^{-1} \circ F_k.
\]

Next, note that (i) implies~\eqref{eq:W14-scale} with $u = u_k$, $r = \ep_k$ and $L^p = p (C_0 + 1)$. Hence by Remark~\ref{rmk:W22}, Proposition~\ref{prop:W14-W2q} and Remark~\ref{rmk:interior-W2q}, and recalling that each $w_k$ maps into $\overline{\Omega'_{t_0}}$, we see that up to taking a subsequence, $w_k$ converges in $C^{1}_{\loc}(\overline{\RR^2_+})$ to a limit $w$ which is non-constant by (ii) and the boundedness of $(I_k(0))$. Moreover, we deduce from (i) that
\begin{equation}\label{eq:finite-2p}
\int_{\RR^2_+} |\nabla w|^2 + |\nabla w|^p < + \infty.
\end{equation}
On the other hand, for all $R > 0$, we have by (iii) and the conformality of $F_k$ that, for sufficiently large $k$,
\begin{equation}\label{eq:composed-pde}
\Div\left( [1 + (\ep_k^2 + \lambda_k^{-2}|\nabla w_k|^2)^{\frac{p}{2} - 1}]\nabla w_k \right) = rf(w_k) (w_k)_{x^1} \times (w_k)_{x^2}, \text{ on }\bB_R^+,
\end{equation}
along with $w_k(x) \in \Sigma \text{ and }(w_k)_{x^2}(x) \perp T_{w_k(x)}\Sigma$ for all $x \in \mathbf{T}_R$. Let $\zeta \in C^{\infty}_c(\bB)$ be a cut-off function which equals $1$ on $\bB_{\frac{1}{2}}$, and let $\zeta_R = \zeta\left( \frac{\cdot}{R} \right)$. Testing~\eqref{eq:composed-pde} against $\psi_k = (x^1 (w_k)_{x^1} + x^2 (w_k)_{x^2}) \zeta_R^2$ and noting that $(w_k)_{x^1} \times (w_k)_{x^2} \perp (w_k)_{x^1}, (w_k)_{x^2}$, we obtain
\begin{align*}
0=\ & -\int_{\bT_R} \left[1 + (\ep_k^2 + \lambda_k^{-2}|\nabla w_k|^2)^{\frac{p}{2} - 1}\right] \psi_k \cdot (w_k)_{x^2} - \int_{\bB^+_{R}} \left[1 + (\ep_k^2 + \lambda_k^{-2}|\nabla w_k|^2)^{\frac{p}{2} - 1}\right] \nabla w_k \cdot \nabla \psi_k. 
\end{align*}
Since for all $x \in \bT_R$ we have $\psi_k(x) \in T_{w_k(x)}\Sigma$ while $(w_k)_{x^2}(x) \perp T_{w_k(x)}\Sigma$, the integral over $\bT_R$ vanishes, and hence so does the integral over $\bB^+_R$. Thus, observing that 
\[
\begin{split}
\nabla w_{k} \cdot \nabla \psi_k =\ & (w_{k})_{x^i} (\psi_k)_{x^i}\\
=\ & (w_k)_{x^i} \big( \delta_{ij}(w_k)_{x^j}\zeta^2 + x^j (w_k)_{x^jx^i} \zeta^2 + 2x^j (w_k)_{x^j}\zeta \zeta_{x^i} \big)\\
=\ & |\nabla w_k|^2 \zeta^2 + x^j \big( \frac{|\nabla w_k|^2}{2} \big)_{x^j}\zeta^2 + 2\zeta \zeta_{x^i} x^j (w_k)_{x^i}(w_k)_{x^j},
\end{split}
\]
we obtain
\begin{equation}\label{eq:pohozaev-1}
\begin{split}
0 =\ &\int_{\bB^+_{R}} \left[1 + (\ep_k^2 + \lambda_k^{-2}|\nabla w_k|^2)^{\frac{p}{2} - 1}\right] \nabla w_k \cdot \nabla \psi_k  \\
 =\ & \int_{\bB^+_{R}}\left[1 + (\ep_k^2 + \lambda_k^{-2}|\nabla w_k|^2)^{\frac{p}{2} - 1}\right] \left( |\nabla w_k|^2 + x^j \big(\frac{|\nabla w_k|^2}{2} \big)_{x^j} \right)\zeta^2 \\
 & +2 \int_{\bB^+_{R}}\left[1 + (\ep_k^2 + \lambda_k^{-2}|\nabla w_k|^2)^{\frac{p}{2} - 1}\right] \zeta \zeta_{x^i} x^j (w_k)_{x^i} (w_k)_{x^j}.
\end{split}
\end{equation}
Next, multiplying both sides of the identity below by $\zeta^2 x^j$ (and summing over $j$),
\[
\begin{split}
&(\ep_k^2 + \lambda_k^{-2}|\nabla w_k|^2)^{\frac{p}{2} - 1}\big( \frac{|\nabla w_k|^2}{2} \big)_{x^j} \\=\ & \frac{\lambda_k^2}{p}\big[ (\ep_k^2 + \lambda_k^{-2}|\nabla w_k|^2)^{\frac{p}{2}} \big]_{x^j} + (\lambda_k)_{x^j} \lambda_k^{-1}(\ep_k^2 + \lambda_k^{-2}|\nabla w_k|^2)^{\frac{p}{2} - 1} |\nabla w_k|^2,
\end{split}
\]
we have upon integrating by parts that 
\[
\begin{split}
& \int_{\bB^+_{R}}  \zeta^2x^j  \big(\frac{|\nabla w_k|^2}{2} \big)_{x^j} + (\ep_k^2 + \lambda_k^{-2}|\nabla w_k|^2)^{\frac{p}{2} - 1}  \zeta^2x^j  \big(\frac{|\nabla w_k|^2}{2} \big)_{x^j} \\
 =\ & \int_{\bT_R}(-x^2) \left( \frac{|\nabla w_k|^2}{2} + \frac{\lambda_k^2(\ep_k^2 + \lambda_k^{-2}|\nabla w_k|^2)^{\frac{p}{2}}}{p} \right)\zeta^2 \\
 &- \int_{\bB_R^+} |\nabla w_k|^2 \zeta^2 + \zeta x^j \zeta_{x^j} |\nabla w_k|^2 \\
 & - \int_{\bB_R^+} \big(\frac{2\zeta^2 \lambda_k^2}{p} +  \frac{2\zeta  \zeta_{x^j} x^j\lambda_k^2}{p} + \frac{2\zeta^2 x^j \lambda_k (\lambda_k)_{x^j}}{p} \big) (\ep_k^2 + \lambda_k^{-2}|\nabla w_k|^2)^{\frac{p}{2}}\\
 & + \int_{\bB_R^+}\zeta^2 x^j (\lambda_k)_{x^j} \lambda_k^{-1}(\ep_k^2 + \lambda_k^{-2}|\nabla w_k|^2)^{\frac{p}{2} - 1} |\nabla w_k|^2.
\end{split}
\]
Combining this with~\eqref{eq:pohozaev-1} yields
\[
\begin{split}
0 =\ & \int_{\bB^+_{R}} \left( \ep_k^2 + \lambda_k^{-2}|\nabla w_k|^2 \right)^{\frac{p}{2}-1}|\nabla w_k|^2 \zeta^2 - \frac{2\zeta^2\lambda_k^2}{p}\left( \ep_k^2 + \lambda_k^{-2}|\nabla w_k|^2 \right)^{\frac{p}{2}} \\
&-\int_{\bB^+_{R}} \zeta x^j \zeta_{x^j}\left( |\nabla w_k|^2 +\frac{2\lambda_k^2}{p}\left( \ep_k^2 + \lambda_k^{-2}|\nabla w_k|^2 \right)^{\frac{p}{2}}  \right)\\
& + \int_{\bB^+_{R}} \zeta^2 x^j (\lambda_k)_{x^j} \lambda_k^{-1} \left( \ep_k^2 + \lambda_k^{-2}|\nabla w_k|^2 \right)^{\frac{p}{2}-1}|\nabla w_k|^2 - \frac{2\zeta^2x^j\lambda_k (\lambda_k)_{x^j}}{p} \Big( \ep_k^2 + \lambda_k^{-2}|\nabla w_k|^2 \Big)^{\frac{p}{2}}\\
& + 2 \int_{\bB^+_{R}}\left[1 + (\ep_k^2 + \lambda_k^{-2}|\nabla w_k|^2)^{\frac{p}{2} - 1}\right] \zeta \zeta_{x^i} x^j (w_k)_{x^i} (w_k)_{x^j}\\
& + \int_{\bT_{R}}  (-x^2) \left( \frac{|\nabla w_k|^2}{2} + \frac{\lambda_k^2(\ep_k^2 + \lambda_k^{-2}|\nabla w_k|^2)^{\frac{p}{2}}}{p} \right)\zeta^2. 
\end{split}
\]
The last integral vanishes since $x^2 = 0$ on $\bT_R$. Letting $k \to \infty$, and recalling that $w_k \to w$ and $\lambda_k \to 1$ in $C^1_{\loc}(\overline{\RR^2_+})$, we get
\begin{align*}
0 =\ & \int_{\bB_R^+} \frac{p-2}{p}\zeta_R^2 |\nabla w|^p\\
& -\int_{\bB_R^+ \setminus \bB_{\frac{R}{2}}^+} x^j (\zeta_R)_{x^j} \zeta_R \Big(|\nabla w|^2 + \frac{2}{p}|\nabla w|^p\Big) + 2\int_{\bB_R^+ \setminus \bB_{\frac{R}{2}}^+} (1 + |\nabla w|^{p-2}) x^j (\zeta_R)_{x^i} \zeta_R w_{x^i}w_{x^j}.
\end{align*}
Letting $R \to \infty$, by~\eqref{eq:finite-2p} and how we defined $\zeta_R$, we get
\[
\frac{p-2}{p}\int_{\RR^2_+}|\nabla w|^p = 0.
\]
As $p > 2$, this contradicts the fact that $w$ is non-constant.
\end{proof}
Having addressed the comparison between $\ep_k$ and $t_k$ in Proposition~\ref{prop:rescale}, we can now rule out energy concentration in the interior and show that the domain of any rescaled limit is a half-plane, as opposed to all of $\RR^2$. This is the content of the next proposition.
\begin{prop}\label{prop:no-interior-bubble}
Given $x_0 \in \cS$, again let $t_k, x_k$ be as defined before Proposition~\ref{prop:rescale}. Then 
\[
\limsup_{k \to \infty}\frac{1 - |x_k|}{t_k} < \infty.
\]
In particular, $\cS \subseteq \partial \bB$.
\end{prop}
\begin{proof}
We follow the idea in~\cite{Struwe88}. Assume towards a contradiction that $t_k^{-1}(1 - |x_k|) \to \infty$ along a subsequence which we do not relabel. Without loss of generality we assume that $|x_k - x_0| \leq \frac{d}{2}$ for all $k$. Next, we define $\widetilde{\ep}_k = \frac{\ep_k}{t_k}$ and
\[
v_k(y) = u_k(x_k + t_k y) \text{ for }y \in \bB_k' := \bB_{\frac{1}{t_k}}(-\frac{x_k}{t_k}).
\]
Then $\wep_k \to 0$ by Proposition~\ref{prop:rescale}. Also, $v_k$ has the following properties.
\vskip 1mm
\begin{enumerate}
\item[(i)] $v_k$ maps into $\overline{\Omega'_{t_0}}$ and
\[
\int_{\bB_k'} \frac{(\widetilde{\ep}_k)^{p-2}( t_k^2 + |\nabla v_k|^2)^{\frac{p}{2}}}{p} + \frac{|\nabla v_k|^2}{2} = \int_{\bB} \frac{\ep_k^{p-2}(1 + |\nabla u_k|^2)^{\frac{p}{2}}}{p} + \frac{|\nabla u_k|^2}{2} \leq C_0 + 1.
\]
\vskip 1mm
\item[(ii)] 
\[
\int_{\bB_1(y) \cap \bB_k'} |\nabla v_k|^2 = \int_{\bB_{t_k}(x_k + t_k y) \cap \bB} |\nabla u_k|^2 \leq \frac{\eta}{3},
\]
for all $y \in \overline{\bB_k'} \cap \overline{\bB_{\frac{d}{2t_k}}(0)}$. Moreover, equality holds when $y = 0$.
\vskip 1mm
\item[(iii)]  $\Div\left( \left[ 1 + (\widetilde{\ep}_k)^{p-2}(t_k^2 + |\nabla v_k|^2)^{\frac{p}{2} - 1}  \right] \nabla v_k\right) = rf(v_k) (v_k)_{x^1} \times (v_k)_{x^2}$ on $\bB'_k$, along with $v_k(x) \in \Sigma$ and $\pa{v_k}{n}(x) \perp T_{v_k(x)}\Sigma \text{ for all } x \in \partial \bB'_k$.
\end{enumerate}

Since $t_k \to 0$ and $t_k^{-1}(1 - |x_k|) \to \infty$, for all $R > 0$ we eventually have $B_{R + 1}(0) \subseteq \bB_k' \cap \bB_{\frac{d}{2t_k}}(0)$. Thus, by (ii) and Remark~\ref{rmk:small-regular-int}, we see that, passing to a subsequence if necessary, $v_k$ converges in $C^1_{\loc}(\RR^2)$ to a limit $v: \RR^2 \to \overline{\Omega'_{t_0}}$ which is non-constant by the equality case in (ii), and satisfies by (i) and (iii) that
\begin{equation}\label{eq:cmc-R2}
\left\{
\begin{array}{l}
\displaystyle\int_{\RR^2} |\nabla v|^2 \leq 2(C_0 + 1) \text{, and }\\
\Delta v = rf(v) v_{x^1} \times v_{x^2} \text{ weakly on }\RR^2.
\end{array}
\right.
\end{equation}
Elliptic regularity implies that $v$ is smooth, and then by Lemma~\ref{lemm:remove} we obtain a non-constant smooth map $w: S^2 \to \overline{\Omega'_{t_0}}$ such that, in terms of isothermal coordinates,
\[
\Delta w = rf(w)w_{x^1} \times w_{x^2} \text{ on }S^2.
\]
Since the domain is $S^2$, a standard argument using the Hopf differential shows that $w$ must be weakly conformal. Then, computing as in the proof of Proposition~\ref{prop:max-principle}(b), we get 
\[
\Delta (e^{ad(w)}) \geq ae^{ad(w)}\frac{H_0 - H}{2}\frac{|\nabla w|^2}{2} \text{ on }\{x \in S^2\ |\ d(w(x)) > 0\} =: C_+.
\]
In particular $C_+ \neq S^2$, as $w$ is non-constant. Now since $d(w) = 0$ on $\partial C_+$, we obtain by the maximum principle that $C_+ = \emptyset$. In other words, $w(S^2) \subseteq \overline{\Omega'}$, and hence $w$ in fact solves
\[
\Delta w = rH w_{x^1} \times w_{x^2} \text{ on }S^2.
\]
Since $w$ is non-constant, by~\cite[Lemma A.1]{BrezisCoron1985} the image of $w$ is a sphere of mean curvature $rH$ contained $\overline{\Omega'}$, a contradiction since $\Sigma'$ has mean curvature at least $H_0 > rH$. Thus we must have $\limsup_{k \to \infty}t_k^{-1}(1 - |x_k|) < \infty$. The second conclusion follows since $t_k \to 0$.
\end{proof}

\subsection*{Proof of Theorem~\ref{thm:main-1}}
It suffices to prove that for any $H \in (0, H_0)$, the asserted existence holds with mean curvature $rH$ for almost every $r \in (0, 1]$. To that end we choose $t_0$ and $f$ as in Section~\ref{subsec:max-principle}, let $r$ be in the full-measure set yielded by Proposition~\ref{prop:existence-with-index} and suppose $u_j, \ep_j, C_0$ are as in the conclusion there. By Propositions~\ref{prop:convergence-mod-bubbling} and~\ref{prop:no-interior-bubble}, we get a smooth solution $u: (\overline{\bB}, \partial \bB) \to (\overline{\Omega'}, \Sigma)$ to~\eqref{eq:fb-cmc} and a finite set $\cS \subseteq \partial \bB$ such that a subsequence of $u_j$, which we do not relabel, converges in $C^1_{\loc}(\overline{\bB}\setminus \cS)$ to $u$. 
\vskip 2mm
\noindent\textbf{Case 1: $\cS = \emptyset$.}
\vskip 2mm
In this case, $u_j$ converges to $u$ in $C^1(\overline{\bB})$, and the uniform lower bound in Proposition~\ref{prop:existence-with-index}(a) guarantees that $u$ is non-constant. It remains to verify $\Ind_{rH}(u) \leq 1$. Recall that Proposition~\ref{prop:existence-with-index}(b) gives $\Ind_{\ep_j, p, rf}(u_j) \leq 1$ for all $j$.

The $C^1$-convergence of $u_j$ to $u$ on all of $\overline{\bB}$ gives $\rho > 0$ and $j_0 \in \NN$ such that for all $j \geq j_0$ and $x \in \partial \bB$, we have that $u(\overline{\bB_{2\rho}(x) \cap \bB})$ and $u_j(\overline{\bB_{2\rho}(x) \cap \bB})$ are contained in $B_{\rho_\Sigma}(u(x))$. Let $\bB_0, \cdots, \bB_N$ be a covering of $\overline{\bB}$ by open balls such that $\bB_0 \subseteq \bB$ and that $\bB_1, \cdots, \bB_N$ are centered on $\partial \bB$ with radius $\rho$, and let $\zeta_0, \zeta_1, \cdots, \zeta_N$ be a partition of unity subordinate to this covering. Then for $i = 1, \cdots, N$, there exist a neighborhood $U_i$ and diffeomorphisms $\Psi_i, \Upsilon_i$ as described in Section~\ref{subsec:spaces-and-coordinates}, such that $U_i$ contains $u(\overline{\bB_i \cap \bB}), u_j(\overline{\bB_i \cap \bB})$ for all $j$. For $\psi \in T_u\cM_p \cap C^{\infty}(\overline{\bB}; \RR^3)$ and $i \in \{1, \cdots, N\}$, define
\[
\widehat{\psi}_i := (d\Psi_i)_u(\psi) \text{ on }\overline{\bB_i \cap \bB},
\]
and note that for all $x \in \bB_i \cap \partial\bB$, since $\widehat{\psi}_i^3(x) = 0$, we have $\Big((d\Psi_i)_{u_j}\Big)^{-1} \big(\widehat{\psi}_i(x)\big) \in T_{u_j(x)}\Sigma$. Letting
\[
\psi^{(j)} = \zeta_0\psi + \sum_{i = 1}^N \zeta_i \cdot \Big((d\Psi_i)_{u_j}\Big)^{-1} \big(\widehat{\psi}_i\big),
\]
we see that $\psi^{(j)} \in T_{u_j}\cM_p$ and converges in $C^1$ to $\psi$ on $\overline{\bB}$. Consequently, 
\[
\lim_{j \to \infty}\delta^2 E_{\ep_j, p, rf}(u_j)(\psi^{(j)}, \psi^{(j)}) = \delta^2 E_{rf}(u)(\psi, \psi) = \delta^2 E_{rH}(u)(\psi, \psi),
\]
where the second equality follows since $f$ is equal to $H$ on a neighborhood of $\overline{\Omega'}$. This allows us to prove that the bound $\Ind_{\ep_j, p, rf}(u_j) \leq 1$ passes to the limit, giving $\Ind_{rH}(u) \leq 1$.
\vskip 1em
\noindent\textbf{Case 2: $\cS \neq \emptyset$.}
\vskip 1mm
Choose $x_0 \in \cS \subseteq \partial \bB$ and let $x_k, t_k$ be as introduced before Proposition~\ref{prop:rescale}. Again assume without loss of generality that $|x_k - x_0| \leq \frac{d}{2}$ for all $k$, and define $\widetilde{\ep}_k, v_k$ and $\bB_k'$ as in the proof of Proposition~\ref{prop:no-interior-bubble}. Then again $\wep_k \to 0$ by Proposition~\ref{prop:rescale}, and $v_k$ has properties (i), (ii) and (iii) from that proof. Letting $z_k = \frac{1}{t_k}\big( \frac{x_k}{|x_k|} - x_k \big)$, we see by Proposition~\ref{prop:no-interior-bubble} that $(|z_k|)$ is a bounded sequence. Next we let $I_k, F_k, \lambda_k$ and $w_k$ be as in the proof of Proposition~\ref{prop:rescale} with $t_k$ in place of $\ep_k$, and note that for all $R > 0$, eventually $w_k$ satisfies
\[
\Div\Big( \Big[ 1 + (\widetilde{\ep}_k)^{p-2}\Big(t_k^2 + \lambda_k^{-2}|\nabla w_k|^2\Big)^{\frac{p}{2} - 1}  \Big] \nabla w_k \Big) = rf(w_k)(w_k)_{x^1} \times (w_k)_{x^2} \text{ on }\bB_R^+,
\]
with $w_k(x) \in \Sigma$ and $(w_k)_{x^2}(x) \perp T_{w_k(x)}\Sigma \text{ for all }x \in \bT_R$. Moreover, by (ii), Proposition~\ref{prop:small-regular-2} and Remark~\ref{rmk:small-regular-int}, as well as the fact that each $w_k$ maps into $\overline{\Omega_{t_0}'}$, we see that up to taking a subsequence, $w_k$ converges in $C^{1}_{\loc}(\overline{\RR^2_+})$ to a limiting map $w: \overline{\RR^2_+} \to \overline{\Omega'_{t_0}}$ such that
\begin{equation}\label{eq:cmc-half-space}
\left\{
\begin{array}{l}
\displaystyle\int_{\RR^2_+}|\nabla w|^2 < +\infty,\\
\Delta w = rf(w)w_{x^1} \times w_{x^2} \text{ on }\RR^2_+,\\
w(x) \in \Sigma \text{ and }w_{x^2}(x) \perp T_{w(x)}\Sigma \text{ for all }x \in \partial \overline{\RR^2_+}.
\end{array}
\right.
\end{equation}
Furthermore, $w$ must be non-constant by the equality case in (ii) and the fact that the sequence $(I_k(0))$ is bounded. Letting $\tau$ be the standard conformal map from $\overline{\bB} \setminus \{(1, 0)\}$ to $\overline{\RR^2_+}$, then by~\eqref{eq:cmc-half-space} and Lemma~\ref{lemm:remove}, the map $\widetilde{w}:= w \circ \tau$ extends smoothly to all of $\overline{\bB}$. By Proposition~\ref{prop:max-principle}(b), the resulting map has image contained in $\overline{\Omega'}$ and consequently solves~\eqref{eq:fb-cmc} on $\overline{\bB}$.

It remains to prove that $\Ind_{rH}(\widetilde{w}) \leq 1$. To that end, let $\psi_1, \cdots, \psi_l \in T_{\widetilde{w}}\cM_p \cap C^{\infty}(\overline{\bB}; \RR^3)$ be linearly independent such that $\delta^2 E_{rH}(\widetilde{w})$ restricts to be negative definite on their span. By the logarithmic cut-off trick, we may assume that $\psi_1, \cdots, \psi_{l}$ are supported away from $(1, 0)$. Thus, there exists $R>0$ such that, letting 
\[
\xi_{a} = \psi_{a} \circ \tau^{-1}, \text{ for }a = 1, \cdots, l,
\]
we have $\xi_a(x) \in T_{w(x)}\Sigma$ on $\partial \RR_+^2$ and that $\supp(\xi_a) \subseteq \bB_{R}^+ \cup T_R$. Since $w_j$ converges to $w$ in $C^1_{\loc}(\overline{\RR^2_+})$, we may repeat the construction in Case 1 to get $\xi_a^{(j)}$ for all $j$ sufficiently large such that $\supp(\xi_a^{(j)}) \subseteq \bB_{R}^+ \cup T_R$, that $\xi_a^{(j)}(x) \in T_{w_j(x)}\Sigma \text{ for all }x \in \partial \RR_+^2$, and that 
\[
\xi_a^{(j)} \to \xi_a \text{ in } C^1(\overline{\bB_R^+}) \text{ as }j \to \infty, \text{ for all } a = 1, \cdots, l.
\] 
Consequently, letting $\tau_j(y) = x_j + t_j y$ and $\theta_{a}^{(j)} =\xi_{a}^{(j)}\circ F_j^{-1}\circ I_j \circ \tau_j^{-1}$, and noting how the second variation formula~\eqref{eq:2nd-variation-formula} transforms under conformal maps, we get after a straightforward computation that
\begin{align*}
\delta^2 E_{\ep_j, p, rf}(u_j)(\theta_{a}^{(j)}, \theta_{a}^{(j)}) 
\longrightarrow \delta^2 E_{rH}(\widetilde{w})(\psi_a, \psi_a), \text{ as }j \to \infty.
\end{align*}
Recalling our choice of $\psi_1, \cdots, \psi_l$ and polarizing, we see that eventually the matrix
\[
\Big( \delta^2 E_{\ep_j, p, rf}(u_j)(\theta_{a}^{(j)}, \theta_{b}^{(j)}) \Big)_{a, b \in \{1, \cdots, l\}}
\]
is negative definite. The bound $\Ind_{\ep_j, p, rf}(u_j) \leq 1$ then forces $l\leq 1$, and we are done.
\section{Improved existence result under convexity assumptions}\label{sec:improved}
In Section~\ref{subsec:index-comparison} we relate the index of $\delta^2 E_H(u)$ to that of another bilinear form $B_H(u)$. In Section~\ref{subsec:energy-bound} we prove a uniform energy upper bound for free boundary, constant mean curvature disks lying inside $\overline{\Omega}$ with index at most $1$, using the classical Hersch trick. At the end we put everything together and prove Theorem~\ref{thm:main-2} by approximating from the full-measure set obtained in Theorem~\ref{thm:main-1}, which we may apply with $\Sigma' = \Sigma$ thanks to the convexity assumptions on $\Sigma$.


\subsection{An index comparison result}\label{subsec:index-comparison}

Suppose we have a smooth, non-constant solution $u: (\overline{\bB}, \partial \bB) \to (\RR^3, \Sigma)$ of 
\begin{equation}\label{eq:cmc-fb-7}
\left\{
\begin{array}{ll}
\Delta u= H \cdot u_{x^1} \times u_{x^2} & \text{ on }\overline{\bB},\\
|u_{x^1}|  = |u_{x^2}|,\ \langle u_{x^1}, u_{x^2}\rangle = 0 & \text{ on }\overline{\bB},\\
u_r(x) \perp T_{u(x)}\Sigma & \text{ for all }x \in \partial \bB.
\end{array}
\right.
\end{equation}
In terms of complex coordinates $z = x + \sqrt{-1}y$, we have
\begin{equation}\label{eq:complex-cmc}
\left\{
\begin{array}{ll}
u_{z\overline{z}} = \displaystyle\frac{\sqrt{-1}H}{2}u_{\overline{z}} \times u_z & \text{ on }\overline{\bB},\\
\langle u_z, u_z \rangle = 0 & \text{ on }\overline{\bB},\\
\re(zu_z) \perp T_{u}\Sigma & \text{ on } \partial \bB.
\end{array}
\right.
\end{equation}
It is standard that $u_z$ vanishes only at a finite set $\cS\subseteq \overline{\bB}$ of branch points~\cite[Section 5]{Hildebrandt-Nitsche1979}. Moreover, $du((\overline{\bB}\setminus \cS)\times \RR^2)$ extends to a smooth, oriented, rank-two subbundle of $\overline{\bB} \times \RR^3$, which we denote by $\xi$. Orientability of $\xi$ implies that its orthogonal complement, denoted $\nu$, is trivial, and we fix a unit-length smooth section $\mathbf{n}$ of $\nu$.  On the other hand, the orientation and bundle metric on $\xi$ makes it a complex line bundle, and the complexification $\xi_\CC$ splits into $\xi^{0, 1} \oplus \xi^{1, 0}$, which, away from $\cS$, are spanned over $\CC$ by $u_{\overline{z}}$ and $u_z$, respectively. Finally, the Levi-Civita connection $\nabla$ on $\RR^3$ induces connections $\nabla^\perp$ and $\nabla^T$ on $\nu$ and $\xi$, and both $\xi^{1, 0}$ and $\xi^{0,1}$ are invariant under $\nabla^T$.

We next introduce the relevant bilinear forms, namely 
\begin{align}\label{eq:2nd-var-energy}
\delta^2 E_{H}(u)(v, v) =\ & \int_{\bB} |\nabla v|^2 dx^1 \wedge dx^2 + H \int_{\bB} \langle v, v_{x^1} \times u_{x^2} + u_{x^1} \times v_{x^2} \rangle dx^1 \wedge dx^2 \nonumber\\
& + \int_{\partial \bB} A_\Sigma^{u_r}(v, v) d\theta,
\end{align}
defined for all $v \in C^{\infty}(\overline{\bB}; \RR^3)$ satisfying $v(x) \in T_{u(x)}\Sigma$ for all $x \in \partial \bB$, and 
\begin{align}\label{eq:2nd-var-area}
B_H(u)(s, s) = \int_{\bB} |\nabla f|^2 - f^2\frac{|\nabla u|^2}{2}\cdot \frac{H^2}{2} dx^1 \wedge dx^2   + \int_{\partial \bB} f^2 A_\Sigma^{u_r}(\mathbf{n}, \mathbf{n}) d\theta,
\end{align}
defined for all $s \in \Gamma(\nu)$, where $f = \langle s , \mathbf{n} \rangle$. Note that for all $x \in \partial \bB \setminus \cS$ we have $\mathbf{n}(x) \in T_{u(x)}\Sigma$, since $\langle \mathbf{n}(x), u_r(x) \rangle = 0$ and $u_r(x)$ spans the orthogonal complement to $T_{u(x)}\Sigma$. As $\cS$ is a discrete set of points, we deduce that in fact
\[
\mathbf{n}(x) \in T_{u(x)}\Sigma \text{ for all }x \in \partial \bB.
\] 
We now state the main result of this section, which is a comparison between the indices of the above two bilinear forms.
\begin{prop}\label{prop:index-comparison}
Let $u:(\overline{\bB}, \partial\bB) \to (\RR^3, \Sigma)$ be a smooth, non-constant solution to~\eqref{eq:complex-cmc}. Then the index of $B_H(u)$ is less than or equal to $\Ind_H(u)$.
\end{prop}
This is a free boundary analogue of~\cite[Proposition 5.1]{ChengZhou-cmc}. Results of this type go back to the work of Ejiri-Micallef~\cite{EM} on closed minimal surfaces. More recently, Lima~\cite{Lima2017} obtained index comparison results for free boundary minimal surfaces. As in these references, the proof of Proposition~\ref{prop:index-comparison} comes down to complementing a negative direction $s \in \Gamma(\nu)$ for $B_H(u)$ by a suitable section $\sigma \in \Gamma(\xi)$ so that $\sigma + s$ is a negative direction for $\delta^2 E_H(u)$. The following lemma exposes the equation that needs to be solved to achieve this.

\begin{lemm}
\label{lemm:index-calculation}
Given $\sigma \in \Gamma(\xi)$ and $s \in \Gamma(\nu)$ with $\sigma(x) \in T_{u(x)}\Sigma \text{ for all }x \in \partial \bB$, let $v = s + \sigma$ and define $\zeta \in \Gamma(\xi^{0, 1})$ by $\zeta = (\nabla_z s)^{0, 1} + \nabla_z^{T}\sigma^{0, 1}$. Then 
\begin{equation}\label{eq:index-computation}
\delta^2 E_{H}(u)(v, v) \leq B_H(u)(s, s) + 8\int_{\bB} |\zeta|^2 dx^1 \wedge dx^2.
\end{equation}
\end{lemm}
\begin{proof}
By the logarithmic cut-off trick, it suffices to treat the case where $s$ and $\sigma$ are supported away from $\cS$. We first rewrite the interior integral in $\delta^2 E_{H}(u)(v, v)$ to get
\begin{align}\label{eq:d2E-complex}
\delta^2 E_{H}(u)(v, v) = 4&\int_{\bB} |\nabla_{z} v|^2 dx^1 \wedge dx^2+4\int_{\bB} \re\langle \nabla_{z}v, \overline{\sqrt{-1}H \cdot u_z\times v} \rangle dx^1 \wedge dx^2\nonumber\\
& + \int_{\partial \bB} A_\Sigma^{u_r}(v, v) d\theta.
\end{align}
Note that this rewriting does not involve integration by parts. On the other hand, we did integrate by parts on three other occasions during the calculations in~\cite[Section 5]{ChengZhou-cmc}. The corresponding steps here introduce boundary terms as follows.
\begin{align*}
\int_{\bB}|\nabla_z \sigma^{1, 0}|^2 - \int_{\bB}|\nabla_z \sigma^{0, 1}|^2 =\ &  \int_{\partial \bB} \langle \sigma^{1, 0}, \nabla_z \sigma^{0, 1} \rangle\frac{\sqrt{-1}}{2}dz  +\langle \sigma^{1, 0}, \nabla_{\overline{z}} \sigma^{0, 1} \rangle\frac{\sqrt{-1}}{2}d\overline{z} =: I_1. 
\end{align*}
\begin{align*}
\int_{\bB} \langle \nabla_z s, \nabla_{\overline{z}}\sigma^{0, 1} \rangle - \int_{\bB} \langle \nabla_{\overline{z}}s, \nabla_z\sigma^{0, 1} \rangle =\ &  \int_{\partial \bB} \langle s, \nabla_z \sigma^{0, 1} \rangle\frac{\sqrt{-1}}{2}dz + \langle s, \nabla_{\overline{z}} \sigma^{0, 1} \rangle\frac{\sqrt{-1}}{2}d\overline{z} = : I_2
\end{align*}
\begin{align*}
2\re\int_{\bB}\langle \nabla_z s, (\nabla_{\overline{z}}\sigma^{1, 0})^\perp \rangle - \langle (\nabla_z s)^{1, 0}, \overline{\nabla_z \sigma^{1, 0}} \rangle =\ & \re\sqrt{-1}H \int_{\bB} \langle \nabla_z s, u_{\overline{z}} \times \sigma^{1, 0} \rangle + \langle u_z \times s, \overline{\nabla_z \sigma^{1, 0}} \rangle\\
=\ & \re\sqrt{-1}H \int_{\partial \bB} \langle s, u_{\overline{z}} \times \sigma^{1, 0} \rangle \frac{\sqrt{-1}}{2}d\overline{z}\\
=\ & 2\re \int_{\partial \bB} \langle s, \nabla_{\overline{z}}\sigma^{1, 0} \rangle \frac{\sqrt{-1}}{2}d\overline{z} = : I_3,
\end{align*}
where in deriving the third identity, besides~\eqref{eq:complex-cmc}, we used that $2(\nabla_{\overline{z}}\sigma^{1, 0})^\perp = \sqrt{-1}H u_{\overline{z}} \times \sigma^{1, 0}$ and that $-2(\nabla_z s)^{1, 0} = \sqrt{-1}H u_z \times s$, both simple consequences of~\eqref{eq:complex-cmc}. Recalling the computations in~\cite[Lemma 5.2]{ChengZhou-cmc}, we get
\begin{align}
&\int_{\bB}|\nabla_z v|^2 + \re \langle  \nabla_z v, \overline{\sqrt{-1}H \cdot u_z\times v} \rangle \frac{\sqrt{-1}}{2}dz \wedge d\overline{z} \nonumber\\
=\  & \Big(2\int_{\bB}|\zeta|^2 + \int_{\bB} |\nabla_z^\perp s|^2 - |\nabla_z^T s|^2\Big) + 2\re (I_2) + I_3 + I_1 \nonumber\\
=\ & 2\int_{\bB}|\zeta|^2 + \int_{\bB} |\nabla_z^\perp s|^2 - |\nabla_z^T s|^2 + 2\re \int_{\partial \bB}  \langle s, \nabla_{\overline{z}} \sigma^{0, 1} \rangle\frac{\sqrt{-1}}{2}d\overline{z}\nonumber\\
&+\int_{\partial \bB} \langle \sigma^{1, 0}, \nabla_z \sigma^{0, 1} \rangle\frac{\sqrt{-1}}{2}dz  +\langle \sigma^{1, 0}, \nabla_{\overline{z}} \sigma^{0, 1} \rangle\frac{\sqrt{-1}}{2}d\overline{z}.\label{eq:E-B-almost}
\end{align}
Noting that $\int_{\bB} |\nabla_z^\perp s|^2 - |\nabla_z^T s|^2$ is part of the complex form of the second variation of the area, we would next like to show that 
\begin{align}
&\int_{\partial \bB} \langle \sigma^{1, 0}, \nabla_z \sigma^{0, 1} \rangle\frac{\sqrt{-1}}{2}dz  +\langle \sigma^{1, 0}, \nabla_{\overline{z}} \sigma^{0, 1} \rangle\frac{\sqrt{-1}}{2}d\overline{z} + 2\re \int_{\partial \bB}  \langle s, \nabla_{\overline{z}} \sigma^{0, 1} \rangle\frac{\sqrt{-1}}{2}d\overline{z}\nonumber\\
=\ & -\frac{1}{4}\int_{\partial \bB} \big[A_\Sigma^{u_r}(v, v) - A_\Sigma^{u_r}(s, s)\big] d\theta. \label{eq:boundary-cleanup-1}
\end{align}
We begin by examining the right-hand side. Note that $d\theta = -\sqrt{-1}\frac{dz}{z}$ and $u_r = 2\re (zu_z)$ on $\partial\bB$. Also, since $\sigma$ is supported away from the branch points of $u$, we may write $\sigma = \overline{a}u_{\overline{z}} + au_z$ for some complex-valued function $a$. Now we compute
\begin{align*}
-\frac{1}{4}\int_{\partial \bB} \big[A_\Sigma^{u_r}(v, v) - A_\Sigma^{u_r}(s, s)\big] d\theta =\ & \re\int_{\partial \bB}\langle u_z, 2 \nabla_\sigma s + \nabla_\sigma \sigma^{0, 1} \rangle \frac{\sqrt{-1}}{2}dz,
\end{align*}
where $\nabla_\sigma$ applied to $s$ and $\sigma^{0, 1}$ is understood to mean $\overline{a}\nabla_{\overline{z}} + a\nabla_z$, and we have used the fact that $\langle  u_z, \nabla_{\sigma}\sigma^{1, 0} \rangle = 0$ as $\nabla_{\sigma}^T\sigma^{1, 0}$ is of type $(1, 0)$ still. For the term involving $\nabla_\sigma s$ on the right-hand side, as $s \in \Gamma(\nu)$ and the ambient connection is flat, we have
\begin{align}
&2\re\int_{\partial \bB} \langle u_z, \nabla_\sigma s \rangle \frac{\sqrt{-1}}{2}dz\nonumber\\
=\ & -2\re\int_{\partial \bB} \langle s, \nabla_z \sigma^{1, 0} + \nabla_z \sigma^{0, 1} \rangle\frac{\sqrt{-1}}{2}dz\nonumber\\
=\ & 2\re\int_{\partial \bB} \langle s, \nabla_{\overline{z}} \sigma^{0, 1}  \rangle\frac{\sqrt{-1}}{2}d\overline{z} + \re \sqrt{-1}H\int_{\partial \bB} \langle s,  u_z \times \sigma^{0, 1} \rangle\frac{\sqrt{-1}}{2}dz, \label{eq:sig-s-calculation}
\end{align}
where in getting the very last term we again used that $2(\nabla_z \sigma^{0, 1})^{\perp} = -\sqrt{-1}Hu_z \times \sigma^{0, 1}$. On the other hand, 
\begin{align}\label{eq:sig-sig-calculation}
\re\int_{\partial \bB}\langle u_z, \nabla_\sigma \sigma^{0, 1} \rangle \frac{\sqrt{-1}}{2}dz =\ & \re \int_{\partial \bB}\Big(\overline{a} \langle u_z, \nabla_{\overline{z}}\sigma^{0, 1} \rangle + a\langle u_z, \nabla_z \sigma^{0, 1} \rangle\Big) \frac{\sqrt{-1}}{2}dz.
\end{align}
Next we turn to the left-hand side of the desired identity~\eqref{eq:boundary-cleanup-1}. Noting that the first integral is a real number by considering $d|\sigma^{1, 0}|^2$ on $\partial \bB$, we find that
\begin{align*}
&\int_{\partial \bB} \langle \sigma^{1, 0}, \nabla_z \sigma^{0, 1} \rangle\frac{\sqrt{-1}}{2}dz  +\langle \sigma^{1, 0}, \nabla_{\overline{z}} \sigma^{0, 1} \rangle\frac{\sqrt{-1}}{2}d\overline{z}\\
=\ & \re\int_{\partial \bB} \langle \sigma^{1, 0}, \nabla_z \sigma^{0, 1} \rangle\frac{\sqrt{-1}}{2}dz  +\langle \sigma^{1, 0}, \nabla_{\overline{z}} \sigma^{0, 1} \rangle\frac{\sqrt{-1}}{2}d\overline{z}\\
=\ & \re\int_{\partial \bB}a\langle u_z, \nabla_z\sigma^{0, 1} \rangle \frac{\sqrt{-1}}{2}dz + a\langle u_z, \nabla_{\overline{z}}\sigma^{0, 1} \rangle \frac{\sqrt{-1}}{2}d\overline{z}.
\end{align*}
Combining this with~\eqref{eq:sig-s-calculation} and~\eqref{eq:sig-sig-calculation}, we obtain
\begin{align*}
&\int_{\partial \bB} \langle \sigma^{1, 0}, \nabla_z \sigma^{0, 1} \rangle\frac{\sqrt{-1}}{2}dz  +\langle \sigma^{1, 0}, \nabla_{\overline{z}} \sigma^{0, 1} \rangle\frac{\sqrt{-1}}{2}d\overline{z} + 2\re \int_{\partial \bB}  \langle s, \nabla_{\overline{z}} \sigma^{0, 1} \rangle\frac{\sqrt{-1}}{2}d\overline{z}\\
& +\frac{1}{4}\int_{\partial \bB} \big[A_\Sigma^{u_r}(v, v) - A_\Sigma^{u_r}(s, s)\big] d\theta\\
=\ & -\re \sqrt{-1}H\int_{\partial \bB} \langle s,  u_z \times \sigma^{0, 1}\rangle \frac{\sqrt{-1}}{2}dz\\
& + \re\int_{\partial \bB} a \langle u_z, \nabla_{\overline{z}}\sigma^{0, 1} \rangle \frac{\sqrt{-1}}{2}d\overline{z} - \overline{a}\langle u_z, \nabla_{\overline{z}}\sigma^{0, 1} \rangle \frac{\sqrt{-1}}{2}dz\\
=\ &  -\re \sqrt{-1}H \int_{\partial \bB} \langle s, u_z \times u_{\overline{z}}\rangle \overline{a}z \frac{\sqrt{-1}}{2}\frac{dz}{z}+\re\int_{\partial \bB}  \langle u_z, \nabla_{\overline{z}}\sigma^{0, 1} \rangle\big( a\overline{z}\frac{\sqrt{-1}}{2}\frac{d\overline{z}}{\overline{z}} -  \overline{a}z\frac{\sqrt{-1}}{2}\frac{dz}{z}\big).
\end{align*}
On $\partial \bB$, since $\langle \sigma, u_r \rangle = 0$, we have $\re(z\overline{a}) = 0$. Noting that $\sqrt{-1}H \langle s,  u_z \times u_{\overline{z}}\rangle$, $\frac{\sqrt{-1}}{2}\frac{d\overline{z}}{\overline{z}}$ and $\frac{\sqrt{-1}}{2}\frac{dz}{z}$ are all real, we see that both integrals in the last line above vanish, and therefore we indeed obtain~\eqref{eq:boundary-cleanup-1}. Combining it with~\eqref{eq:E-B-almost}, we get that 
\begin{align*}
\delta^2E_H(u)(v, v)  =\ & 8\int_{\bB} |\zeta|^2 dx^1 \wedge dx^2+ \Big(4\int_{\bB} |\nabla_z^{\perp}s|^2 - |\nabla^T_z s|^2 dx^1 \wedge dx^2 + \int_{\partial\bB} A_\Sigma^{u_r}(s, s) d\theta\Big).
\end{align*}
We then rewrite and estimate $ |\nabla_z^{\perp}s|^2 - |\nabla^T_z s|^2$ as in~\cite[Lemma 5.2]{ChengZhou-cmc} to obtain finally that
\begin{align*}
\delta^2 E_H(u)(v, v) \leq B_H(u)(s, s) + 8\int_{\bB}|\zeta|^2 dx^1 \wedge dx^2.
\end{align*}
This completes the proof of Lemma~\ref{lemm:index-calculation}.
\end{proof}
\begin{proof}[Proof of Proposition~\ref{prop:index-comparison}] 
In view of~\eqref{eq:index-computation}, to estimate $\Ind_{H}(u)$ from below by the index of $B_H(u)$, we want to solve the following boundary value problem for $\sigma \in \Gamma(\xi^{0, 1})$ given $s \in \Gamma(\nu)$. Below we let $\eta$ be the inward unit normal to $\Sigma$. Then $\eta(u(x)) \in \xi_x$ for all $x \in \partial \bB$. For brevity, we denote $\eta\circ (u|_{\partial \bB})$ still by $\eta$.
\begin{equation}\label{eq:complementing-section}
\left\{
\begin{array}{ll}
\nabla_z^T \sigma = -(\nabla_z s)^{0, 1} & \text{ in }\bB,\\
\re \langle \sigma, \eta^{1, 0} \rangle = 0 & \text{ on }\partial \bB.
\end{array}
\right.
\end{equation}
Denote by $L^2\Gamma(\xi^{0, 1})$ the space of $L^2$-sections of $\xi^{0, 1}$, regarded as a real Hilbert space with the inner product 
\[
(\sigma, \alpha) \mapsto \re\int_{\bB} \langle \sigma, \overline{\alpha} \rangle \frac{\sqrt{-1}}{2} dz \wedge d\overline{z}.
\]
The space $W^{1, 2}\Gamma(\xi^{0, 1})$ is defined similarly. Next, letting
\[
\cD = \{\sigma \in W^{1, 2}\Gamma(\xi^{0, 1})\ |\ \re\langle \sigma, \eta^{1, 0} \rangle = 0 \text{ on }\partial \bB\},
\]
we have the following global \textit{a priori} estimate: 
\[
\|\sigma\|_{1, 2} \leq C(\|\sigma\|_{2} + \|\nabla_z^T \sigma\|_2), \text{ for all }\sigma \in \cD.
\]
In particular, $\nabla^T_z:\cD \to L^2\Gamma(\xi^{0, 1})$ has finite-dimensional kernel and closed range. Now suppose $\alpha \in L^2\Gamma(\xi^{0, 1}) \cap (\Ran(\nabla^T_z))^\perp$, so that
\[
\re\int_{\bB}\langle \nabla_z^T \sigma, \overline{\alpha} \rangle  \frac{\sqrt{-1}}{2} dz \wedge d\overline{z}=0 \text{, for all }\sigma \in \cD.
\]
Elliptic regularity gives $\alpha \in \Gamma(\xi^{0, 1})$, and we may integrate by parts to get
\begin{align*}  
\re\int_{\partial \bB} \langle \sigma, \overline{\alpha} \rangle \frac{\sqrt{-1}}{2}d\overline{z} - \re\int_{\bB} \langle \sigma, \overline{\nabla^T_{\overline{z}}\alpha} \rangle \frac{\sqrt{-1}}{2}dz \wedge d\overline{z} = 0, \text{ for all } \sigma \in \cD.
\end{align*}
In particular, $\nabla^T_{\overline{z}}\alpha = 0$. Moreover, writing the boundary integrand as
\begin{align*}
\langle \sigma, \overline{\alpha} \rangle \frac{\sqrt{-1}}{2}d\overline{z} =\ & 2\langle \sigma, \eta^{1, 0} \rangle \langle \eta^{0, 1}   , \overline{\alpha} \rangle \frac{\sqrt{-1}}{2}d\overline{z} =  2\langle \sigma, \eta^{1, 0} \rangle \langle \eta^{0, 1} , \overline{z}\overline{\alpha} \rangle \frac{\sqrt{-1}}{2}\frac{d\overline{z}}{\overline{z}},
\end{align*}
we deduce that the boundary condition on $\alpha$ is
\[
\im\langle z\alpha, \eta^{1, 0} \rangle = 0 \text{ on }\partial \bB.
\]
Noting that $ zu_z = c \eta^{1, 0}$ for some real-valued function $c$ on $\partial \bB$, we have
\begin{equation}\label{eq:section-bc}
\im z^2\langle \alpha, u_z \rangle = 0 \text{ on }\partial \bB.
\end{equation}
As $\nabla^T_{\overline{z}}\alpha = 0$ and $\nabla^T_{\overline{z}}u_z = 0$, the latter being a consequence of~\eqref{eq:complex-cmc}, we see that $z^2 \langle \alpha, u_z \rangle$ is a holomorphic function on $\bB$, and hence by~\eqref{eq:section-bc} it must be a real constant, which we denote by $\lambda$. Since $\frac{\lambda}{z^2} = \langle \alpha, u_z \rangle \in L^{1}(\bB)$, we conclude that $\lambda = 0$, which gives $\alpha = 0$, since $\alpha = \langle \alpha, u_z \rangle \frac{u_{\overline{z}}}{|u_z|^2}$ away from the branch points of $u$. Having shown that $(\Ran(\nabla^T_z))^\perp = \{0\}$, and recalling that $\Ran(\nabla^T_z)$ is closed, we see that~\eqref{eq:complementing-section} has a solution for all $s \in \Gamma(\nu)$. The asserted inequality now follows as in the proof of~\cite[Proposition 5.1]{ChengZhou-cmc}.
\end{proof}
\subsection{Uniform energy bound}\label{subsec:energy-bound}
\begin{prop}\label{prop:dirichlet-bound}
Suppose $A_\Sigma^\eta > 0$, where $\eta$ is the inward unit normal to $\Sigma$, and that $u: (\overline{\bB}, \partial \bB) \to (\overline{\Omega}, \Sigma)$ is a smooth solution to~\eqref{eq:cmc-fb-7} with the index of $B_H(u)$ at most $1$. Then 
\begin{equation}\label{eq:dirichlet-bound}
D(u) \leq \frac{16\pi}{H^2}.
\end{equation}
\end{prop}
\begin{proof}
It suffices to consider the case where $u$ is non-constant. Since $u(\overline{\bB}) \subseteq \overline{\Omega}$ and $u(\partial \bB) \subseteq \Sigma$, for all $x \in \partial \bB$ the condition $u_r(x) \perp T_{u(x)}\Sigma$ implies that $u_r(x)$ is a non-positive multiple of $\eta(u(x))$. Therefore we have
\begin{equation}\label{eq:convexity-sign}
\int_{\partial \bB} A_\Sigma^{u_r}(\mathbf{n}, \mathbf{n}) \leq  0.
\end{equation}
Since $H > 0$ and $u$ is not a constant, substituting $f \equiv 1$ into~\eqref{eq:2nd-var-area} then shows that the index of $B_H(u)$ is positive, and hence exactly $1$ under our assumptions. Letting $v > 0$ be a lowest eigenfunction for the operator $-\Delta  - \frac{|\nabla u|^2}{2}\frac{H^2}{2}$ subject to the boundary condition $v_r + A_\Sigma^{u_r}(\mathbf{n}, \mathbf{n})v = 0 \text{ on }\partial \bB$, we have $B_H(u)(f\mathbf{n}, f\mathbf{n}) \geq 0$ whenever $\int_{\bB} fv = 0$. 

The remainder of the proof is similar to a special case of~\cite[Lemma 48]{ABCS2019}. Take a conformal diffeomorphism $f$ from $\bB$ onto the upper hemisphere $S^2_+ \subseteq \RR^3$. Then by a degree theory argument, there exists a conformal diffeomorphism $F: S^2 \to S^2$ such that, viewing $F$ as a map into $\RR^3$, we have
\begin{equation}\label{eq:balanced}
\int_{\bB} (F^i\circ f) v = 0 \text{ for }i = 1, 2, 3.
\end{equation}
Hence 
\[
\int_{\bB} |\nabla (F^i \circ f)|^2 - (F^i \circ f)^2 \frac{H^2}{2}\frac{|\nabla u|^2}{2} + \int_{\partial \bB} (F^i \circ f)^2 A_\Sigma^{u_r}(\mathbf{n}, \mathbf{n}) \geq 0 \text{ for }i = 1, 2, 3.
\]
Summing from $i = 1$ to $i =3$ and using the conformality of $F\circ f$ give
\begin{equation}\label{eq:Hersch-outcome}
8\pi + \int_{\partial \bB} A_\Sigma^{u_r}(\mathbf{n}, \mathbf{n}) \geq \frac{H^2}{2}\int_{\bB} \frac{|\nabla u|^2}{2}.
\end{equation}
Recalling~\eqref{eq:convexity-sign} finishes the proof.
\end{proof}
\subsection*{Proof of Theorem~\ref{thm:main-2}}
The assumptions on $\Sigma$ allows us to take $\Sigma' = \Sigma$ in the proof of Theorem~\ref{thm:main-1}. Thus, given $H \in (0, H_0)$, we obtain a sequence $(r_n)$ increasing to $1$ such that for all $n$, there exists a non-constant $u_n:(\overline{\bB}, \partial\bB) \to (\overline{\Omega}, \Sigma)$ with
\begin{equation}\label{eq:cmc-n}
\left\{
\begin{array}{ll}
\Delta u_n = r_nH \cdot (u_n)_{x^1} \times (u_n)_{x^2} & \text{ on } \bB, \\
|(u_n)_{x^1}| = |(u_n)_{x^2}|,\ \langle (u_n)_{x^1}, (u_n)_{x^2} \rangle = 0 & \text{ on }\bB, \\
(u_n)_r(x) \perp T_{u_n(x)}\Sigma & \text{ on }\partial \bB,
\end{array}
\right.
\end{equation}
satisfying in addition that $\Ind_{r_nH}(u_n) \leq 1$. Moreover, we note that the thresholds $\eta$ and $\beta$ in Proposition~\ref{prop:small-regular-2}, Remark~\ref{rmk:small-regular-int} and Proposition~\ref{prop:uniform-lower-bound} are independent of $n$, and hence, as a consequence of the proof of Theorem~\ref{thm:main-1}, we have 
\[
D(u_n) \geq \min\{\frac{\eta}{6}, \beta\} \text{, for all }n.
\]
On the other hand, by the index bound $\Ind_{r_nH}(u_n)\leq 1$, Proposition~\ref{prop:index-comparison} and Proposition~\ref{prop:dirichlet-bound}, we get for all $n$ that
\[
D(u_n) \leq \frac{16\pi}{r_n^2H^2}.
\]
We can then finish the proof by a bubbling analysis similar to that in Section~\ref{sec:passage}. In this last part, the strict inequality $H< H_0 \ ( \leq H_{\Sigma})$ is required to derive a contradiction at the step corresponding to Proposition~\ref{prop:no-interior-bubble}.
\bibliographystyle{amsplain}
\bibliography{cmc-free-boundary}
\end{document}